\documentclass[reqno]{amsart}
\usepackage{amsmath,amsthm,amssymb,amsfonts}
\usepackage{mathrsfs}
\usepackage{color}
\usepackage{enumerate}
\usepackage{multirow} 
\newtheoremstyle{mystyle}
  {}
  {}
  {\normalfont}
  { }
  {\bfseries}
  {}
  {10pt}
  { }
\theoremstyle{mystyle}
\newtheorem{theorem}{Theorem}

\newtheorem{lemma}{Lemma}
\newtheorem{remark}{Remark}
\newtheorem{prf}{Proof of Theorem}
\DeclareMathOperator*{\arginf}{arg\,inf}
\usepackage{mathtools}
\usepackage[pdftex]{graphicx}
\def\diag{\mathop{\rm diag}\nolimits}
\def\Diag{\mathop{\rm Diag}\nolimits}
\def\tr{\mathop{\rm tr}\nolimits}
\def\vec{\mathop{\rm vec}\nolimits}
\def\vech{\mathop{\rm vech}\nolimits}
\def\det{\mathop{\rm det}\nolimits}
\def\rank{\mathop{\rm rank}\nolimits}
\def\Int{\mathop{\rm Int}\nolimits}
\newcommand{\dd}{\mathrm d}
\newcommand{\E}{\mathbb E}
\newcommand{\PP}{\mathbb P}
\allowdisplaybreaks[4]
\usepackage[foot]{amsaddr}
\usepackage[top=3cm, bottom=3cm, left=3cm,right=3cm]{geometry}
\usepackage{comment}
\title[Statistical inference in factor analysis for diffusion processes]{Statistical inference in factor analysis for diffusion processes from discrete observations}

\author[S. Kusano]{Shogo Kusano $^{1}$}
\author[M. Uchida]{Masayuki Uchida $^{1,2}$}
\address{$^{1}$Graduate School of Engineering Science, Osaka University}
\address{$^{2}$Center for Mathematical Modeling and Data Science (MMDS), Osaka University and JST CREST}
\begin{document}
\begin{abstract}
We consider statistical inference in factor analysis for ergodic and non-ergodic diffusion processes from discrete observations. 
Factor model based on high frequency time series data has been mainly discussed in the field of high dimensional covariance matrix estimation. 
In this field, the method based on principal component analysis has been mainly used. However, this method is effective only for high dimensional model. On the other hand, 
there is a method based on the quasi-likelihood. However, since the factor is assumed to be observable, we cannot use this method when the factor is latent. Thus, the existing methods are not effective when the factor is latent and the dimension of the observable variable is not so high. Therefore, we propose an effective method in the situation.
\end{abstract}


\maketitle

\section{Introduction}
We consider the following factor model
\begin{align}
    X_{t}=\Lambda_{k} f_{k,t}+e_{t}\quad (t\in [0,T]),
    \label{intro-1}
\end{align}
where $\{X_t\}_{t\geq 0}$ is a $p$-dimensional observable vector process, $\{f_{k,t}\}_{t\geq0}$ is a $k$-dimensional latent factor vector process, $\{e_t\}_{t\geq0}$ is a $p$-dimensional latent unique factor vector process, $\Lambda_{k}\in \mathbb{R}^{p\times k}$ is a constant factor loading matrix. We assume that $\{f_{k,t}\}_{t\geq0}$ satisfies the following stochastic differential equation:
\begin{align}
\begin{cases}
\dd f_{k,t}=b_{k}(f_{k,t})\dd t+S_{k}\dd W_t\quad (t\in [0,T]),\\
f_{k,0}=c_{1},
    \label{intro-2}
\end{cases}
\end{align}
where $W_t$ is an $r$-dimensional standard Wiener process, $b_{k}:\mathbb{R}^k\rightarrow\mathbb{R}^k$, $S_{k}\in \mathbb{R}^{k\times r}$, $c_1\in\mathbb{R}^k$. Furthermore, we assume that $\{e_t\}_{t\geq 0}$ satisfies
the following  stochastic differential equation:
\begin{align}
\begin{cases}
\dd e_t^{(i)}=B_{i}(e_t^{(i)})\dd t+\sigma_{i}\dd B_t^{(i)}
\quad (i=1,\cdots,p)\quad (t\in [0,T]),\\
e_0=c_2,
\label{intro-3}
\end{cases}
\end{align}
where for $i=1,\cdots,p$, $B_t^{(i)}$ is a one-dimensional standard Wiener process,
$B_t^{(1)},\cdots,B_t^{(p)}$ are independent, $B_t=(B_t^{(1)},\cdots,B_t^{(p)})^{\top}$ and $W_t$ are independent, 
$B_{i}:\mathbb{R}\rightarrow\mathbb{R}$, $\sigma_{i}\in\mathbb{R} \ (i=1,\cdots,p)$
and $c_2\in\mathbb{R}^p$. 
In order to avoid the rotational indeterminacy, 
we impose the following  identification conditions: 
$\Lambda_{k}=(I_k,A_{k})^{\top}$, where $I_k\in \mathbb{R}^{k\times k}$ is the identity matrix of size $k$, $A_{k}\in \mathbb{R}^{(p-k)\times k}$. Set
$\displaystyle\Sigma_{ff,k}=S_{k}S_{k}^\top$
and
$\theta_{k}=(\vec{A_{k}},\vech{\Sigma_{ff,k}},\sigma_{1}^2,
\cdots ,\sigma_{p}^2)^\top\in \Theta_{k}\subset\mathbb{R}^{q_k}$, 
where $\Theta_{k}$ is a convex compact space, $
q_{k}=(p-k)k+\frac{k(k+1)}{2}+p$,
$\vec$ is the vectorization of a matrix, $\vech$ is  the half-vectorization.
For the vectorization and the half-vectorization,
see for example Brown \cite{brown(1974)}. 
Let the true parameter be
$\theta_{k,0}=(\vec{A_{k,0}},\vech{\Sigma_{ff,k,0}},\sigma_{1,0}^2,
\cdots ,\sigma_{p,0}^2)^\top\in\Int{(\Theta_{k})}, 
$ where $A_{k,0}\in\mathbb{R}^{(p-k)\times k}$, $\Sigma_{ff,k,0}\in\mathbb{R}^{k\times k}$
and $\sigma_{i,0}\in\mathbb{R} \ (i=1,\cdots,p)$. 
$b_{k,0}(f_{k,t})$ denotes the true function of $b_{k}(f_{k,t})$.
For $i=1,\cdots, p$,
$B_{i,0}(e_{t}^{(i)})$ represents the true function of $B_{i}(e_{t}^{(i)})$. 
 Set $\Lambda_{k,0}=(I_{k},A_{k,0})^{\top}$.
$\{X_{t_{i}^n}\}_{i=1}^n$ are discrete observations,  where $t_{i}^n=ih_n$ and $T= nh_n$.
In this paper, we treat the factor models based on ergodic and non-ergodic diffusion processes.

Factor analysis (FA) 
is the method that describes the relationships between different variables. 
For FA, 
see Anderson and Amemiya \cite{anderson(1988)}, Ihara and Kano \cite{ihara(1986)}
and references therein. 
It has been used in various fields, e.g., behavioral science, 
economics and medical science. There are many studies on FA
for time series model. Most of them are studies for low frequency time series model, see Molenaar \cite{molenaar(1985)} and Toyoda \cite{toyoda(1997)}. 
On the other hand, there are not many studies for high frequency time series models. 
The factor model based on high frequency time series data was proposed in Fan and Xiu \cite{fan(2016)},
and
the factor model 
has been discussed mainly in the field of high dimensional covariance matrix estimation. 
Shephard and Xiu \cite{shephard(2017)}  considered parameter estimation 
based on the quasi-likelihood for a factor model with market microstructure noise
when the factor is observable.
On the other hand,  A{\"i}t-Sahalia and Xiu \cite{ait(2017)},
Dai et al.\ \cite{dai(2019)} and Pelger \cite{pelger(2019)} considered
the method when the factor is a latent. In these studies, 
the method based on principal component analysis is used for parameter estimation. 
Their proposed estimators have consistency for high dimensional model. 
However, for low dimensional model, 
this estimator do not have even consistency, see Bai \cite{bai}. 
Hence, we cannot use this method when the dimension of the observable variable is 
not so high.
Therefore, when the factor is latent and 
the dimension of the variable is not so high,
the methods of existing studies are not effective. 
Therefore, we propose the effective method in the situation. 
In this paper, 
since we assume that the factor is latent,
we discuss not only the parameter estimation but also statistical inference for the number of factors. 
This enables us to evaluate the model. Although we can apply this method to confirmatory FA and structural equation modeling (SEM),
we focus on exploratory FA in this paper. 
For details of confirmatory FA and SEM, see
Yuan et al. \cite{kano(1997)} and Kano \cite{kano(2001)}, respectively.

This paper is organized as follows. 
 In Section $2$, the notation and assumptions are stated. We propose the methods of parameter estimation for the factor models
based on ergodic and non-ergodic diffusion processes
and the test statistics for the number of factors in Section $3$. 
In Section $4$, we give an example and simulation studies.
Section $5$ is devoted to the proofs of the assertions in Section $3$. 

\section{Notation and assumptions}
Let \begin{align*}
    \partial_{\theta_{k}^{(i)}}=\frac{\partial}{\partial\theta_{k}^{(i)}}\  (i=1,\cdots,q_{k})\ ,\  \partial_{\theta_k}=(\partial_{\theta_k^{(1)}},\cdots,\partial_{\theta_k^{(q_k)}})^\top,\  \partial^2_{\theta_k}=\partial_{\theta_k}\partial_{\theta_k}^\top,
\end{align*}
where $\top$ is the transpose of a matrix. 
For a $p\times p$ matrix A, 
\begin{align*}
    \diag(A)=(a_{11},a_{22},\cdots,a_{pp})^\top.
\end{align*}
For a $p$-dimensional vector $a=(a_{1},a_{2},\cdots,a_{p})^\top$, 
    \begin{align*}
    \Diag(a)=\begin{pmatrix}
    a_{1} & 0 & 0 & 0\\
    0     & a_{2} & 0 & 0 \\
    \vdots & \vdots & \ddots & \vdots \\
    0 & 0 & 0 & a_{p}
    \end{pmatrix}. 
    \end{align*}
For a $p\times p$ symmetric matrix A, 
we define the $p^2\times\bar{p}$ matrix $D_p$ as the matrix that satisfies the following equation
\begin{align*}
    \vec{A}=D_p\vech{A},
\end{align*}
where 
\begin{align*}
    \bar{p}=\frac{p(p+1)}{2}.
\end{align*}
Let $D_{p}^{+}=(D_p^{\top}D_p)^{-1}D_p^\top$. Then, we have
\begin{align*}
    D_p^{+}\vec{A}=\vech{A}.
\end{align*}    
$R:\mathbb{R}\times \mathbb{R}^d\rightarrow \mathbb{R}$ 
is defined as
\begin{align*}
    |R(u,x)|\leq u C(1+|x|)^C
\end{align*}
for some $C>0$.
For an  $m\times n$ matrix A, 
\begin{align*}\|A\|^2=\tr{AA^\top}=\sum_{i=1}^m\sum_{j=1}^n A_{ij}^2.
\end{align*}
For $p \times p$ matrices $A$ and $B$, 
\begin{align*}
    (A\otimes B)_{st}=(A\otimes B)_{(i,j),(k,l)}=a_{ik}b_{jl}\quad (0\leq s\leq p^2,0\leq t\leq  p^2 ),
\end{align*}
where 
\begin{align*}
    s=p(i-1)+j\quad(1\leq i,j\leq p )\ , \  t=p(k-1)+l\quad(1\leq k,l\leq p ).
\end{align*}
For a smooth function $\phi_{k}:\mathbb{R}^k\rightarrow\mathbb{R}$, 
\begin{align*}
    \mathcal{L}(\phi_{k}(x))=\sum_{i=1}^{k} b_{k,0}^{(i)}(x)\frac{\partial \phi_{k}(x)}{\partial x^{(i)}}+\frac{1}{2}\sum_{i=1}^{k}\sum_{j=1}^{k}(\Sigma_{ff,k,0})_{ij}\frac{\partial^2\phi_{k}(x)}{\partial x^{(i)}\partial x^{(j)}}.
\end{align*}
Let
\begin{align*}
    \Delta_{k}=\frac{\partial}{\partial \theta_{k}}\vech{\Sigma(\theta_{k,0})}
\end{align*}
and
\begin{align*}
     \mathscr{F}^{n}_{i}=\sigma(W_{s},B_{s},s\leq t_{i}^n)\quad (0\leq i \leq n).
\end{align*}
Let $C^{k}_{\uparrow}(\mathbb R^d)$ be 
the space of all functions $f$ satisfying the following conditions.
\begin{itemize}
    \item[(i)] $f$ is continuously differentiable with respect to $x\in \mathbb{R}^d$ up to order k. 
    \item[(ii)] $f$ and all its derivatives are of polynomial growth in $x\in \mathbb{R}^d$, i.e., 
    $g$ is
of polynomial growth in $x\in \mathbb{R}^d$ if $\displaystyle g(x)=R(1,x)$. 
\end{itemize}
Let $N_{p}(\mu,\Sigma)$ be the $p$-dimensional-normal random variable
with mean $\mu\in\mathbb{R}^p$ and covariance matrix 
$\Sigma\in\mathbb{R}^{p\times p}$. 
$\chi^2_{r}$ denotes the random variable which has the chi-squared distribution with $r$ degrees of freedom. $\chi^2_{r}(\alpha)$ represents an upper $\alpha$ point of the chi-squared distribution with $r$ degrees of freedom, where $0\leq\alpha\leq 1$. 
\\

Furthermore, we make the following assumptions.\\
\begin{enumerate}
\renewcommand{\labelenumi}{{\textbf{[A\arabic{enumi}]}}}
\item 
\begin{enumerate}
\item
There exists a constant $C>0$ such that for any $x,y\in\mathbb R^{k}$, 
\begin{align*}
|b_{k}(x)-b_{k}(y)|\le C|x-y|.
\end{align*}
\item For all $\ell\geq 0$, 
$\displaystyle\sup_t\E\bigl[|f_{k,t}|^{\ell}\bigr]<\infty$.
\item 
$b_{k}\in C^{4}_{\uparrow}(\mathbb R^{k})$.
\end{enumerate}
\item 
The diffusion process  $f_{k,t}$ is ergodic with its invariant measure $\pi_{f, k}$. 
For any $\pi_{f, k}$-integrable function $g$, it holds that
\begin{align*}
\frac{1}{T}\int_{0}^{T}{g(f_{k,t})dt}\overset{P}{\longrightarrow}\int g(x)\pi_{f, k}(dx)
\end{align*}
as $T\to\infty$. 
\end{enumerate}

\begin{enumerate}
\renewcommand{\labelenumi}{{\textbf{[B\arabic{enumi}]}}}
\item 
\begin{enumerate}
\item
There exists a constant $C>0$ such that for any $x,y\in\mathbb R$, 
\begin{align*}
|B_{i}(x)-B_{i}(y)|\le C|x-y|\quad (i=1,\cdots p).
\end{align*}
\item   
For all $\ell\geq 0$, $\displaystyle\sup_t\E\bigl[|e_t|^{\ell}\bigr]<\infty$.

\item 
$B_{i}\in C^{4}_{\uparrow}(\mathbb R)\quad (i=1,\cdots,p)$.
\end{enumerate}

\item $\displaystyle \sigma_i>0\quad (i=1,\cdots,p)$.

\item 
The diffusion process $e_t$ is ergodic with its invariant measure $\pi_{e}$.
For any $\pi_{e}$-integrable function $g$, it holds that
\begin{align*}
\frac{1}{T}\int_{0}^{T}{g(e_t)dt}\overset{P}{\longrightarrow}\int g(x)\pi_{e}(dx)
\end{align*}
as $T\to\infty$. 
\end{enumerate}

\begin{remark}

[A1] and [B1] are the standard assumptions for ergodic diffusion processes. 
For example, see Kessler \cite{kessler(1997)}. 
[B2] implies that $\Sigma_{k}(\theta_{k})$ is non-singular. 
\end{remark}
\section{Main theorems}
\subsection{Ergodic case}
Let
\begin{align*}
    \Sigma_{k}(\theta_{k})=\Lambda_{k}\Sigma_{ff,k}\Lambda_{k}^{\top}+\Sigma_{ee},
\end{align*}
where 
\begin{align*}
    \Sigma_{ee}=\Diag(\sigma_1^2,\cdots,\sigma_p^2)^{\top}.
\end{align*}
For the estimation of $\Sigma_{k}(\theta_{k,0})$
in independent and identically distributed (i.i.d.) models, which means that
$X_1, X_2, \ldots, X_n$ are independent and each $X_i$ has an identical distribution, 
we use the sample covariance matrix 
\begin{align}
    S_n^2=\frac{1}{n}\sum_{i=1}^n(X_{i}-\bar{X}_n)(X_{i}-\bar{X}_n)^{\top}
    \label{main-1},
\end{align}
where 
\begin{align*}
    \bar{X}_n=\frac{1}{n}\sum_{i=1}^nX_{i}.
\end{align*}
Since  (\ref{main-1}) cannot be used as an estimator of $\Sigma_{k}(\theta_{k,0})$
for diffusion processes,
we consider the following realised covariance matrix
\begin{align*}
    Q_{XX}=\frac{1}{T}\sum_{i=1}^n(X_{t_{i}^n}-X_{t_{i-1}^n})(X_{t_{i}^n}-X_{t_{i-1}^n})^\top.
\end{align*}
 The realised covariance matrix is a well-known estimator in the field of covariance matrix estimation, see Bibinger et al. \cite{bibinger(2014)}, 
 Koike and Yoshida \cite{koike(2019)}, Podolskij and Vetter \cite{vetter(2009)} 
  and references therein.  
 First of all, we consider the asymptotic property of the realised covariance. Let
\begin{align*}
    W_{k}(\theta_{k})=2D_{p}^{+}(\Sigma_{k}(\theta_{k})\otimes\Sigma_{k}(\theta_{k}))D_{p}^{+\top}.
\end{align*}
Then, we have the following theorem.
\begin{theorem} Under assumptions [A1], [A2], [B1], [B2] and [B3], as $h_n\rightarrow0$ and $nh_n\rightarrow\infty$, 
\begin{align}
    Q_{XX}\stackrel{P_{\theta_{k,0}}}{\to} \Sigma_{k}(\theta_{k,0}).
    \label{the1-1}
\end{align}
Moreover, as $nh_n^2\rightarrow0$,
\begin{align}
    \sqrt{n}(\vech(Q_{XX})-\vech(\Sigma_{k}(\theta_{k,0})))\stackrel{d}{\to} N_{\bar{p}}(0,W_{k}(\theta_{k,0})).
    \label{the1-2}
\end{align}
\end{theorem}
\begin{remark}
Theorem 1 has the same asymptotic property as 
the sample covariance matrix
when each $X_i$ in i.i.d. model has a normal distribution. 
Therefore, this result justifies the use of the realized covariance instead of the sample covariance matrix in factor analysis for diffusion processes.
\end{remark}

Next, for parameter estimation of the factor model,
we consider the following contrast function. 
\begin{align}
F_{k,n}(Q_{XX},\Sigma_{k}(\theta_{k}))=\left(\vech(Q_{XX})-\vech(\Sigma_{k}(\theta_{k}))\right)^{\top}W_{k}^{-1}(\theta_{k})\left(\vech(Q_{XX})-\vech(\Sigma_{k}(\theta_{k}))\right).
\end{align}
Moreover, the minimum contrast estimator $\hat{\theta}_{k,n}$ is defined as
\begin{align}
    \hat{\theta}_{k,n}=\arginf_{\theta_{k}\in\Theta_{k}}F_{k,n}(Q_{XX},\Sigma_{k}(\theta_{k})).
\end{align}
Furthermore, the following assumptions are made.
\begin{enumerate}
    \renewcommand{\labelenumi}{{\textbf{[C\arabic{enumi}]}}}
    \item $\Sigma_k(\theta_{k})=\Sigma_k(\theta_{k,0})\Rightarrow\theta_{k}=\theta_{k,0}$.
    \item $\rank{\Delta_{k}}=q_{k}$.
\end{enumerate}
Then, we have the following theorem.
\begin{theorem} Under assumptions [A1], [A2], [B1], [B2], [B3], [C1] and [C2], 
as $h_n\rightarrow0$ and $nh_n\rightarrow\infty$, 
\begin{align}
    \hat{\theta}_{k,n} 
    \stackrel{P_{\theta_{k,0}}}{\to}\theta_{k,0}.
    \label{the2-1}
\end{align}
Moreover, as $nh_n^2\rightarrow0$, 
\begin{align}
    \sqrt{n}(\hat{\theta}_{k,n}-\theta_{k,0})
    \stackrel{d}{\to}N_{q_{k}}\left(0,(\Delta_{k}^\top W_{k}^{-1}(\theta_{k,0})\Delta_{k})^{-1}\right). 
    \label{the2-2}
\end{align}
\end{theorem}
\begin{remark}
This estimator has the same asymptotic property as the quasi-maximum likelihood estimator $\hat{\theta}_{k,n}^{ML}$ described below.
It follows from the Euler-Maruyama approximation
and (\ref{intro-2}) that
\begin{align}
    F_{k,t_{i}^n}-F_{k,t_{i-1}^n}
    =b_{k}(F_{k,t_{i-1}^n})h_n
    +S_{k}(W_{t_{i}^n}-W_{t_{i-1}^n}).
    \label{main1}
\end{align}
In the same way, we obtain from (\ref{intro-3}) that
\begin{align}
    E_{t_{i}^n}-E_{t_{i-1}^n}
    =B(E_{t_{i-1}^n})h_n
    +\bar{S}(B_{t_{i}^n}-B_{t_{i-1}^n}),
    \label{main2}
\end{align}
where \begin{align*}B(E_{t_{i-1}^n})&= (B_1(E_{t_{i-1}^n}),\cdots,B_p(E_{t_{i-1}^n}))^{\top}\\
\bar{S}&=\Diag{(\sigma_{1},\cdots,\sigma_{p})^{\top}}
\end{align*}
Since it follows from (\ref{main1}) and (\ref{main2}) that
\begin{align*}
    Z_{t_{i}^n}-Z_{t_{i-1}^n}
    &=\Lambda_{k}(F_{k,t_{i}^n}-F_{k,t_{i-1}^n})
    +E_{t_{i}^n}-E_{t_{i-1}^n}\\
    &=\Lambda_{k}b_{k}(F_{k,t_{i-1}^n})h_n
    +\Lambda_{k}S_{k}(W_{t_{i}^n}-W_{t_{i-1}^n}) +B(E_{t_{i-1}^n})h_n+\bar{S}(B_{t_{i}^n}-B_{t_{i-1}^n})\\
    &=\Lambda_{k} S_{k}(W_{t_{i}^n}-W_{t_{i-1}^n})+\bar{S}(B_{t_{i}^n}-B_{t_{i-1}^n})+R(h_n,F_{k,t_{i-1}^n})+R(h_n,E_{t_{i-1}^n}),
\end{align*}
we have
\begin{align*}
    \bar{Z}_{t_{i}^n}-\bar{Z}_{t_{i-1}^n}
    =\Lambda_{k}S_{k}(W_{t_{i}^n}-W_{t_{i-1}^n})+\bar{S}(B_{t_{i}^n}-B_{t_{i-1}^n})
\end{align*}
as an approximation to (\ref{intro-1}).
The property of the Brownian motion implies that
\begin{align*}
    \Lambda_k S_{k}(W_{t_{i}^n}-W_{t_{i-1}^n})&\sim N_{p}(0,h_n\Lambda_k^{\top}\Sigma_{ff,k}\Lambda_k),\\
    \bar{S}(B_{t_{i}^n}-B_{t_{i-1}^n})&\sim N_{p}(0,h_n\Sigma_{ee}),
\end{align*}
where for random variables $X$ and $Y$, $X \sim Y$ means that $X$ has the same distribution as $Y$.
Furthermore, the conditional distribution is as follows:
\begin{align*}
    \bar{Z}_{t_{i}^n}|\bar{Z}_{t_{i-1}^n}=\bar{z}_{t_{i-1}^n}
    \sim N_{p}(\bar{z}_{t_{i-1}^n},h_n\Sigma_{k}(\theta_{k})).
\end{align*}
Consequently, we have the following joint probability density function of 
$(\bar{Z}_{t_{i}^n})_{0\leq i\leq n}$ 
\begin{align*}
\prod_{i=1}^{n}\frac{1}{(2\pi h_n)^{\frac{p}{2}}\det{\Sigma_{k}(\theta_{k})}^{\frac{1}{2}}}\exp{\left\{-\frac{1}{2h_n}(\bar{z}_{t_{i}^n}-\bar{z}_{t_{i-1}^n})^{\top}\Sigma_{k}^{-1}(\theta_{k})(\bar{z}_{t_{i}^n}-\bar{z}_{t_{i-1}^n})\right\}}.
\end{align*}
Therefore, we define the quasi-likelihood function as
\begin{align}
    L_{k,n}(\theta_{k})=\prod_{i=1}^{n}\frac{1}{(2\pi h_n)^{\frac{p}{2}}\det{\Sigma_{k}(\theta_{k})}^{\frac{1}{2}}}\exp{\left\{-\frac{1}{2h_n}(X_{t_{i}^n}-X_{t_{i-1}^n})^{\top}\Sigma_{k}^{-1}(\theta_{k})(X_{t_{i}^n}-X_{t_{i-1}^n})\right\}}.
    \label{main3}
\end{align}
The logarithm of (\ref{main3}) is
\begin{align*}
    \ell_{k,n}(\theta_{k})&=\log L_{k,n}(\theta_{k})\\
    &=\sum_{i=1}^n\left\{-\frac{p}{2}\log(2\pi h_n)
    -\frac{1}{2}\log\det{\Sigma_{k}(\theta_{k})}-\frac{1}{2h_n}(X_{t_{i}^n}-X_{t_{i-1}^n})^{\top}\Sigma_{k}^{-1}(\theta_{k})(X_{t_{i}^n}-X_{t_{i-1}^n})\right\}\\
    &=-\frac{pn}{2}\log(2\pi h_n)-\frac{n}{2}\log\det{\Sigma_{k}(\theta_{k})}
    -\frac{1}{2h_n}\sum_{i=1}^n\tr{\left\{\Sigma_{k}^{-1}(\theta_{k})(X_{t_{i}^n}-X_{t_{i-1}^n})(X_{t_{i}^n}-X_{t_{i-1}^n})^{\top}\right\}}\\
    &=-\frac{pn}{2}\log(2\pi h_n)-\frac{n}{2}\log\det{\Sigma_{k}(\theta_{k})}
    -\frac{n}{2}\tr{\left\{\Sigma_{k}^{-1}(\theta_{k})Q_{XX}\right\}}. 
\end{align*}
If $Q_{XX}>0$, $\ell_{k,n}(\theta_{k})$ has a maximum value 
\begin{align}
    -\frac{pn}{2}\log(2\pi h_n)-\frac{n}{2}\log\det{Q_{XX}}-\frac{np}{2}
    \label{main4}
\end{align}
at $\Sigma_{k}(\theta_{k})=Q_{XX}$.
Set
\begin{align*}
    M_{k,n}(\theta_{k})
    &= -\frac{2}{n}\left\{\ell_{k,n}(\theta_{k})-\left\{-\frac{pn}{2}\log(2\pi h_n)-\frac{n}{2}\log\det{Q_{XX}}-\frac{np}{2}\right\}\right\}\\
    &=\log\det{\Sigma_{k}(\theta_{k})}-\log\det{Q_{XX}}
    +\tr{\left\{\Sigma_{k}^{-1}(\theta_{k})Q_{XX}\right\}}-p.
\end{align*}
In the same way as Theorem 1 in Shapiro \cite{shapiro(1985)}, we have
\begin{align*}
    M_{k,n}(Q_{XX},\theta_k)=F_{k,n}(Q_{XX},\theta_k)+o_p(1).
\end{align*}
Furthermore, we define the quasi-maximum likelihood estimator as
\begin{align*}
    \hat{\theta}_{k,n}^{ML}= \arg \sup_{\theta_k\in\Theta_k} (-M_{k,n}(\theta_{k})).
\end{align*}
In the same way as Theorem 2 in Shapiro \cite{shapiro(1985)}, 
$\hat{\theta}_{k,n}$ has the same asymptotic property as $\hat{\theta}_{k,n}^{ML}$.
\end{remark}
Theorem 2 has an asymptotic property
under the model of the true number of factors. 
However, we do not know the number of factors since we assume that the factor is latent. 
Hence, we have to discuss the number of factors $k$. 
In order to deal with this problem, we consider the following hypothesis testing.
\begin{align*}
    \left\{
    \begin{array}{ll}
    H_0: k=k^{*},\\
    H_1: k\neq k^{*},
    \end{array}
    \right.
\end{align*}
where $k^{*}\in\mathbb{N}$.
In order to study this hypothesis testing, 
we propose the following test statistic $T_{k^{*},n}$.
\begin{align*}
    T_{k^{*},n}= nF_{k^{*},n}(Q_{XX},\Sigma_{k^{*}}(\hat{\theta}_{k^{*},n})). 
\end{align*}
The asymptotic property of the test statistic $T_{k^{*},n}$ is as follows. 
\begin{theorem}
Under assumptions  [A1], [A2], [B1], [B2], [B3], [C1] and [C2], 
as $h_n\rightarrow0$, $nh_n\rightarrow\infty$ and $nh_n^2\rightarrow0$, 
\begin{align*}
    T_{k^{*},n} \stackrel{d}{\to}\chi^2_{\bar{p}-q_{k^{*}}}
\end{align*}
under $H_0$.
\end{theorem}
\begin{remark}
This test statistic is asymptotically equivalent to the quasi-likelihood ratio test statistic derived as follows. Let
\begin{align*}
    L_{n}(\Sigma)=\prod_{i=1}^{n}\frac{1}{(2\pi h_n)^{\frac{p}{2}}\det{(\Sigma)}^{\frac{1}{2}}}\exp{\left\{-\frac{1}{2h_n}(X_{t_{i}^n}-X_{t_{i-1}^n})^{\top}\Sigma^{-1}(X_{t_{i}^n}-X_{t_{i-1}^n})\right\}},
\end{align*}
where $\Sigma\in\mathbb{R}^{p\times p}$ is an arbitrary positive definite matrix.
Furthermore, let
\begin{align*}
    \ell_{n}(\Sigma)= \log L_{n}(\Sigma) =-\frac{pn}{2}\log(2\pi h_n)-\frac{n}{2}\log\det{(\Sigma)}
    -\frac{n}{2}\tr{(\Sigma^{-1}Q_{XX})}.
\end{align*}
We define the quasi-likelihood ratio as follows:
 \begin{align*}
     \lambda_{k^{*},n}&= \frac{\max_{\theta_{k^{*}}\in\Theta_{k^{*}}}L_{k^*,n}(\Sigma_{k^{*}}(\theta_{k^{*}}))}{\max_{\Sigma>0}L_{n}(\Sigma)}.
 \end{align*}
For $Q_{XX}>0$, we obtain
 \begin{align*}
     -2\log\lambda_{k^{*},n}&=-2\max_{\theta_{k^{*}}\in\Theta_{k^{*}}}\ell_{n}(\Sigma_{k^{*}}(\theta_{k^{*}}))+2\max_{\Sigma>0}\ell_{n}(\Sigma)\\
     &=-2\left\{-\frac{pn}{2}\log(2\pi h_n)-\frac{n}{2}\log\det{\Sigma_{k^{*}}(\hat{\theta}_{k^{*},n}^{ML})}
    -\frac{n}{2}\tr{\{\Sigma_{k^{*}}^{-1}(\hat{\theta}_{k^{*},n}^{ML})Q_{XX}\}}\right\}\\
    &\quad +2\left\{-\frac{pn}{2}\log(2\pi h_n)-\frac{n}{2}\log\det{Q_{XX}}-\frac{np}{2}\right\}\\
    &=n\left\{\log\det{\Sigma_{k^{*}}(\hat{\theta}_{k^{*},n}^{ML})}-\log\det{Q_{XX}+\tr{\{\Sigma_{k^{*}}^{-1}(\hat{\theta}_{k^{*},n}^{ML})Q_{XX}\}}-p}\right\}\\
    &=nM_{k^{*},n}(Q_{XX},\Sigma_{k^*}(\hat{\theta}_{k^{*},n}^{ML}))
 \end{align*}
 Furthermore, in the same way as Theorem 2 in Shapiro \cite{shapiro(1985)}, 
it follows that
 \begin{align*}
     nF_{k^{*},n}(Q_{XX},\Sigma_{k^{*}}(\hat{\theta}_{k^{*},n}))-nM_{k^{*},n}(Q_{XX},\Sigma_{k^{*}}(\hat{\theta}_{k^{*},n}^{ML}))=o_{p}(1).
 \end{align*}
Therefore, $T_{k^*,n}$ is asymptotically equivalent to the quasi-likelihood ratio test statistic.
 \end{remark}
From the above results, we can construct the test of asymptotically significance level $\alpha$. The rejection region is as follows:
\begin{align*}
    \left\{t_{k^{*},n}>\chi^2_{\bar{p}-q_{k^*}}(\alpha)\right\}.
\end{align*}
Let
\begin{align*}
    U_{k^*}(\theta_{k^*})=(\vech(\Sigma_{k}(\theta_{k,0}))-\vech{(\Sigma_{k^*}(\theta_{k^*}))})^{\top}W_{k^*}^{-1}(\theta_{k^*})(\vech(\Sigma_{k}(\theta_{k,0}))-\vech{(\Sigma_{k^*}(\theta_{k^*}))})
\end{align*}
and $\bar{\theta}_{k^*}$ is defined as
\begin{align}
    \bar{\theta}_{k^*}=\arginf_{\theta_{k^*}\in\Theta_{k^*}}U_{k^*}(\theta_{k^*}).
\end{align}
Furthermore, we make the following assumption.
\begin{enumerate}
    \renewcommand{\labelenumi}{{\textbf{[C3]}}}
    \item $U_{k^*}(\theta_{k^*})=U_{k^*}(\bar{\theta}_{k^*})\Rightarrow\theta_{k^*}=\bar{\theta}_{k^*}$.
    \renewcommand{\labelenumi}{{\textbf{[C4]}}}
    \item $\Sigma_{k}(\theta_{k,0})\neq\Sigma_{k^{*}}(\bar{\theta}_{k^{*}})\quad (k\neq k^*)$.     
\end{enumerate}
\begin{theorem}
Under assumptions  [A1], [A2], [B1], [B2], [B3], [C2], [C3] and [C4], as $h_n\rightarrow0$ and $nh_n\rightarrow\infty$,  
\begin{align*}
    \PP(T_{k^{*},n}>\chi^2_{\bar{p}-q_k^{*}}(\alpha))\stackrel{}{\to}1
\end{align*}
under $H_1$.
\end{theorem}
\begin{remark}
It is shown that 
$\hat{\theta}_{k^*,n}$ converges to $\bar{\theta}_{k^*}$ in probability.
See Lemma 8 for the proof. Note that $\bar{\theta}_{k}=\theta_{k,0}$ under $H_0$.
\end{remark}
\begin{remark}
First of all, we consider the test with $H_0: k=1$. If the null hypothesis is accepted, we conclude that the model is correct. If the null hypothesis is rejected, we consider the test with $H_0: k=2$. Then, if the null hypothesis is accepted, we conclude that the model is correct. Similarly, if the null hypothesis is rejected, we continue the test until the null hypothesis is accepted. If the null hypothesis is not accepted in the positive range of degrees of freedom, we conclude that there is no factor structure. 
\end{remark}
\subsection{Non-ergodic case}
Next, we study the non-ergodic model
in the case where the assumptions [A2] and [B3] are not assumed and $T$ is fix.
By analogous manners to Theorems $1$-$4$, we have the following four theorems.
\begin{theorem} 
Under assumptions [A1], [B1] and [B2], as $h_n\rightarrow0$, 
\begin{align*}
    Q_{XX}\stackrel{P_{\theta_{k,0}}}{\to} \Sigma_{k}(\theta_{k,0})
\end{align*}
and
\begin{align*}
    \sqrt{n}(\vech(Q_{XX})-\vech(\Sigma_{k}(\theta_{k,0})))\stackrel{d}{\to} 
     N_{\bar{p}}(0,W_{k}(\theta_{k,0})). 
\end{align*}
\end{theorem}
\begin{theorem} 
Under assumptions [A1], [B1], [B2], [C1] and [C2], as $h_n\rightarrow0$, 
\begin{align*}
    \hat{\theta}_{k,n}
    \stackrel{P_{\theta_{k,0}}}{\to}\theta_{k,0}
\end{align*}
and
\begin{align*}
    \sqrt{n}(\hat{\theta}_{k,n}-\theta_{k,0})
    \stackrel{d}{\to}
      N_{q_k}\left(0,(\Delta_{k}^\top W_{k}^{-1}(\theta_{k,0})\Delta_{k})^{-1}\right).
\end{align*}
\end{theorem}
\begin{theorem}
Under assumptions  [A1], [B1], [B2], [C1] and [C2], as $h_n\rightarrow0$,  
\begin{align*}
    T_{k^*,n} \stackrel{d}{\to}\chi^2_{\bar{p}-q_{k^*}}
\end{align*}
under $H_0$.
\end{theorem}
\begin{theorem}
Under assumptions  [A1], [B1], [B2], [C2], [C3] and [C4], as $h_n\rightarrow0$, 
\begin{align*}
    \PP(T_{k^*,n}>\chi^2_{\bar{p}-q_{k^*}}(\alpha))\stackrel{}{\to}1
\end{align*}
under $H_1$.
\end{theorem}

\section{SIMULATION STUDY}
\subsection{Model}
We consider the case where $p=6$ and $k=2$. 
For $t\in [0,T]$, 
\begin{align*}
    X_t^{(1)}&=\quad \ f_{2,t}^{(1)}\quad \quad \quad \quad \ +e_t^{(1)},\\
    X_t^{(2)}&=\quad \quad \quad \quad \quad \ \  f_t^{(2)}+e_t^{(2)},\\
    X_t^{(3)}&=a_{11}f_{2,t}^{(1)}+a_{12}f_{2,t}^{(2)}+e_t^{(3)}, \\  
    X_t^{(4)}&=a_{21}f_{2,t}^{(1)}+a_{22}f_{2,t}^{(2)}+e_t^{(4)},\\
    X_t^{(5)}&=a_{31}f_{2,t}^{(1)}+a_{32}f_{2,t}^{(2)}+e_t^{(5)},\\
    X_t^{(6)}&=a_{41}f_{2,t}^{(1)}+a_{42}f_{2,t}^{(2)}+e_t^{(6)},
\end{align*}
where $a_{ij}\in\mathbb{R}\quad(i=1,\cdots 4,j=1,\cdots,2)$.
We assume that $\{f_{2,t}\}_{t\geq0}$ is the following two-dimensional-OU process
\begin{align*}
\begin{cases}
\dd f_{2,t}=-(B_2f_{2,t}-\mu_2)\dd t+S_2\dd W_t\quad (t\in [0,T]),\\
f_{2,0}=c_1,
\end{cases}
\end{align*}
where $W_t$ is the two-dimensional standard Wiener process, $B_2\in\mathbb{R}^{2\times 2}$, $\mu_2\in \mathbb{R}^2$, $S_2\in \mathbb{R}^{2\times 2}$ and $c_1\in \mathbb{R}^2$. 
Furthermore, we assume that $\{e_t\}_{t\geq0}$ is the following six-dimensional-OU process
\begin{align*}
\begin{cases}
\dd e_t^{(i)}=-(\bar{B}_{i}e_t^{(i)}-\bar{\mu}_{i})\dd t
+\sigma_{i}\dd B_t^{(i)}\quad (i=1,\cdots,6)\quad (t\in [0,T]),\\
e_0=c_2,
\end{cases}
\end{align*}where 
for $i=1,\cdots,6$, 
$B_t^{(i)}$ is a one-dimensional standard Wiener process, 
$B_t^{(1)},\cdots, B_t^{(6)}$ are independent, 
$B_t=(B_t^{(1)},\cdots,B_t^{(6)})^{\top}$ and $W_t$ are independent, 
$\bar{B}_{i}\in\mathbb{R}$, $\bar{\mu}_{i}\in \mathbb{R}$, $\sigma_{i}\in\mathbb{R}\ (i=1,\cdots,6)$
and $c_2\in \mathbb{R}^6$. 
Next, we set the true parameter and the initial value. Let 
\begin{align*}
    A_{2,0}=\begin{pmatrix}
    3 & 1 \\
    1 & 5 \\
    7 &-4 \\
    -3& 2 
    \end{pmatrix}
    \ , \
    B_{2,0}&=\begin{pmatrix}
    0.5 & 0.3 \\
    0.2 & 0.4 \\
    \end{pmatrix}\ ,\ 
    \mu_{2,0}=\begin{pmatrix}
    2  \\
    4  \\
    \end{pmatrix}\ ,\ 
    S_{2,0}=\begin{pmatrix}
    2 & 3 \\
    5 & 1 \\
    \end{pmatrix}\ ,\ 
    c_1=\begin{pmatrix}
    3 \\
    5 
    \end{pmatrix},
    \end{align*}
    \begin{align*}
    \begin{pmatrix}
    \bar{B}_{1,0}\\
    \bar{B}_{2,0}\\
    \bar{B}_{3,0}\\
    \bar{B}_{4,0}\\  
    \bar{B}_{5,0}\\
    \bar{B}_{6,0}
    \end{pmatrix}=
    \begin{pmatrix}
    3\\
    2\\
    3\\
    2\\
    6\\
    2
    \end{pmatrix}\ , \ 
    \begin{pmatrix}
    \bar{\mu}_{1,0}\\
    \bar{\mu}_{2,0}\\
    \bar{\mu}_{3,0}\\
    \bar{\mu}_{4,0}\\  
    \bar{\mu}_{5,0}\\
    \bar{\mu}_{6,0}
    \end{pmatrix}=
    \begin{pmatrix}
    0\\
    0\\
    0\\
    0\\
    0\\
    0
    \end{pmatrix}\ ,\ 
    \begin{pmatrix}
    \sigma_{1,0}\\
    \sigma_{2,0}\\
    \sigma_{3,0}\\
    \sigma_{4,0}\\  
    \sigma_{5,0}\\
    \sigma_{6,0}
    \end{pmatrix}=
    \begin{pmatrix}
    2\\
    4\\
    5\\
    1\\
    3\\
    2
    \end{pmatrix}\ , \ 
    c_{2}=
    \begin{pmatrix}
    0\\
    0\\
    0\\
    0\\
    0\\
    0
    \end{pmatrix}.
\end{align*}
Since one has
\begin{align*}
    \Lambda_{2,0}=\begin{pmatrix}
    I_2\\
    A_{2,0}
    \end{pmatrix}\ , \
     \Sigma_{ff,2,0}=S_{2,0}S_{2,0}^\top=\begin{pmatrix}
    13 & 13 \\
    13 & 26 \\
    \end{pmatrix}\ ,
    \Sigma_{ee,0}=\Diag{(4,16,25,1,9,4)^{\top}},
\end{align*}
it follows that
\begin{align*}
    \theta_{k,0}=\theta_{2,0}&=(
    \vec{A_{2,0}}^\top,\vech{(\Sigma_{ff,2,0})^\top},\diag{(\Sigma_{ee,0})}^\top
    )^\top\\
    &=(
    3,1,7,-3,1,5,-4,2,13,13,26,4,16,25,1,9,4
    )^\top\in\Theta_{2},
\end{align*}
where $\Theta_{2}=[-30,30]^{17}$. 
Furthermore, 
\begin{align*}
 \Sigma_{k}({\theta_{k,0}})= \Sigma_{2}({\theta_{2,0}})=\Lambda_{2,0}\Sigma_{ff,2,0}\Lambda_{2,0}^\top+\Sigma_{ee,0}=\begin{pmatrix}
    17&13&52&78&39&-13 \\
    13&42&65&143&-13&13\\
    52&65&246&377&104&-26\\
    78&143&377&794&-26&52\\
    39&-13&104&-26&334&-143\\
    -13&13&-26&52&-143&69
    \end{pmatrix}.
\end{align*}
Finally, we confirm that the model satisfies the assumptions.  
For all $x,y\in \mathbb{R}^2$, 
\begin{align*}
    &\quad\left|-(B_{2,0} x-\mu_{2,0})-\{-(B_{2,0} y -\mu_{2,0})\}\right|\\
    &=\left|
    \begin{pmatrix}
    0.5(x^{(1)}-y^{(1)})+0.3(x^{(2)}-y^{(2)}) \\
    0.2(x^{(1)}-y^{(1)})+0.4(x^{(2)}-y^{(2)}) \\
    \end{pmatrix}
    \right|\\
    &=\sqrt{\{0.5(x^{(1)}-y^{(1)})+0.3(x^{(2)}-y^{(2)})\}^2+\{0.2(x^{(1)}-y^{(1)})+0.4(x^{(2)}-y^{(2)})\}^2}\\
    &\leq \sqrt{2\{0.5^2(x^{(1)}-y^{(1)})^2+0.3^2(x^{(2)}-y^{(2)})^2\}+2\{0.2^2(x^{(1)}-y^{(1)})^2+0.4^2(x^{(2)}-y^{(2)})^2\}}\\
    &\leq \sqrt{2(x^{(1)}-y^{(1)})^2+2(x^{(2)}-y^{(2)})^2}=\sqrt{2}|x-y|,
\end{align*}
so that [A1] (a) is satisfied. For all $x\in\mathbb{R}^{2}$, 
\begin{align*}
    x^\top\{-(B_{2,0} x-\mu_{2,0})\}
    &=x^{(1)}(-0.5x^{(1)}-0.3x^{(2)}+2)+x^{(2)}(-0.2x^{(1)}-0.4x^{(2)}+4)\\
    &=-\frac{1}{2}x^{{(1)}2}+\frac{1}{2}x^{(1)}x^{(2)}-\frac{2}{5}x^{(2)2}+2x^{(1)}+4x^{(2)}\\
    &=-\frac{1}{4}(x^{(1)}+x^{(2)})^2-\frac{1}{20}\|x\|^2-\frac{1}{5}(x^{(1)}-5)^2+5-\frac{1}{10}(x^{(2)}-20)^2+40\\
    &\leq -\frac{1}{20}\|x\|^2+45.
\end{align*}
Moreover, noting that
\begin{align*}
    \det{\Sigma_{ff,2,0}}=\begin{pmatrix}
    13 & 13 \\
    13 & 26 
    \end{pmatrix}=169>0,
\end{align*}
we have $\Sigma_{ff,2,0}>0$. 
Hence,  there exists $K>0$ such that
\begin{align*}
    \frac{1}{K}\begin{pmatrix}
    1 & 0 \\
    0 & 1 \\
    \end{pmatrix}<
    \begin{pmatrix}
    13 & 13 \\
    13 & 26 \\
    \end{pmatrix}<K
    \begin{pmatrix}
    1 & 0 \\
    0 & 1 \\
    \end{pmatrix},
\end{align*}
which meets [A1] (b) and [A2]. 
It follows from $\displaystyle-(B_{2,0}x-\mu_{2,0})\in C^{4}_{\uparrow}(\mathbb R^2)$ that
[A1] (c) is satisfied. 
For all $x,y\in \mathbb{R}$, 
\begin{align*}
    |-(\bar{B}_{i,0}x-\bar{\mu}_{i,0})-\{-(\bar{B}_{i,0}y-\bar{\mu}_{i,0})\}|\leq 6 |x-y|\quad (i=1,\cdots,6),
\end{align*}
which fulfills [B1] (a).
We have $-(\bar{B}_{i,0}x-\bar{\mu}_{i,0})\in C^{4}_{\uparrow}(\mathbb R)\ (i=1,\cdots,6)$, 
so that [B1] (c) is satisfied. 
For all $x\in\mathbb{R}$, 
\begin{align*}
    x\{-(\bar{B}_{i,0}x-\bar{\mu}_{i,0})\}\leq -2x^2\quad (i=1,\cdots,6).
\end{align*}
Furthermore, there exists $K>0$ such that
\begin{align*}
    \frac{1}{K}<\sigma_{i,0}<K\ (i=1,\cdots,6),
\end{align*}
which meets [B1] (b) and [B3].
Noting that $\sigma_{i,0}>0\ (i=1,\cdots,6)$, we have [B2]. 
Since $\Lambda_{2,0}{\Sigma_{ff,2,0}^{\frac{1}{2}}}$ meets
the identifiability condition
in Anderson and Rubin \cite{anderson(1956)}, 
[C1] is satisfied. 
It follows from $\displaystyle \rank{\Delta_{2}}=17$ that [C2] is fulfilled. 

\subsection{Result}
In the simulation, optim() is used with the BFGS method in R language. We choose the initial parameter $\theta_k=\theta_0$. The number of iteration is 10000.
\subsubsection{Ergodic case} 
We set $(n,h_n,T)=(10^6,10^{-4},10^2)$. First of all, we see the simulation result of $Q_{XX}$. Table 1 shows a sample mean and a sample standard deviation of $Q_{XX}$. We can see that $Q_{XX}$ has consistency. Figure 1 shows histogram, Q-Q plot and empirical distribution of $Q_{XX}$. We can see that $Q_{XX}$ has asymptotic normality. Therefore, these simulation results show that 
Theorem 1 holds true in this example. Next, we investigate
the simulation result of $\hat{\theta}_{2,n}$. 
Table 2 shows a sample mean and a sample standard deviation of $\hat{\theta}_{2,n}$,
which implies that $\hat{\theta}_{2,n}$ has consistency. 
Figure 2 shows histogram, Q-Q plot and empirical distribution of $\hat{\theta}_{2,n}$.
It seems that $\hat{\theta}_{2,n}$ has asymptotic normality. 
Therefore, these simulation results yield that the result of Theorem 2 is correct
in this example. Finally, we check the simulation result of goodness-of-fit test.

We consider the following hypothesis testing.
\begin{align*}
    \left\{
    \begin{array}{ll}
    H_0: k=1, \\
    H_1: k\neq 1. 
    \end{array}
    \right.
\end{align*}
The test statistic is as follows:
\begin{align*}
    T_{1,n}= nF_{1,n}(Q_{XX},\Sigma_{1}(\hat{\theta}_{1,n})).
\end{align*}
The rejection region is 
\begin{align*}
    \left\{t_{1,n}>\chi^2_{10}(0.05)=18.31\right\}.
\end{align*}
Table 3 shows the quartiles of $T_{1,n}$. Since the minimum value of $T_{1,n}$ is greater than 18.31, 
we can see from Table 5 that $H_0$ is rejected all 10000 times.
Next, we study the following hypothesis testing.

\begin{align*}
    \left\{
    \begin{array}{ll}
    H_0: k=2, \\
    H_1: k\neq 2. 
    \end{array}
    \right.
\end{align*}
The test statistic is as follows:
\begin{align*}
    T_{2,n}= nF_{2,n}(Q_{XX},\Sigma_{2}(\hat{\theta}_{2,n})).
\end{align*}
The rejection region is  
\begin{align*}
    \left\{t_{2,n}>\chi^2_{4}(0.05)=9.89\right\}.
\end{align*}
Figure 3 shows the histogram, QQ-plot and empirical distribution of $T_{2,n}$. Table 4 shows the sample mean and sample standard deviation of $T_{2,n}$. We can see that $T_{2,n}$ converges in distribution to $\chi_4^2$.
Hence, 
these simulation results show that Theorem 3 is correct in this example. 
Table 5 shows the number of rejections of the test with $H_0:k= 1,k= 2$,
which implies that
Theorem 4 holds true for this example.
\subsubsection{Non-ergodic case}
We set $(n,h_n,T)=(10^3,10^{-3},10^1)$. 
Tables 6-7 show a sample mean and a sample standard deviation of $Q_{XX}$ and  $\hat{\theta}_{2,n}$, respectively. 
Figures 4-6 show histogram, Q-Q plot and empirical distribution of $Q_{XX}$, $\hat{\theta}_{2,n}$ and $T_{2,n}$, respectively. 
Table 8 shows the quartiles of $T_{1,n}$. 
Table 9 shows the sample mean and sample standard deviation of $T_{2,n}$. 
Table 10 shows the number of rejections of the test with $H_0:k= 1,k= 2$.
All estimators and test statistics in non-ergodic case 
have the same good performances as those in the ergodic case.

\clearpage
\begin{table}[]
    \centering
    \setlength{\doublerulesep}{0.4pt}
        \caption{Sample mean and sample standard deviation (SD)  of $Q_{XX}$.}
    \begin{tabular}{ccccc}
        \hline\hline
        & $Q_{XX,11}$&$Q_{XX,12}$&$Q_{XX,13}$&\\\hline
        Mean (true value) & 17.001 (17.000) & 13.001 (13.000) & 52.003 (52.000) \\
         SD (theoretical value) & 0.024 (0.024) & 0.030 (0.030) & 0.083 (0.083)\\\hline
        $Q_{XX,14}$&$Q_{XX,15}$&$Q_{XX,16}$&$Q_{XX,22}$\\\hline
         78.003 (78.000) &39.002 (39.000) & -13.001 (-13.000) & 42.003 (42.000) \\
         0.140 (0.140) &0.085 (0.085) & 0.037 (0.037) & 0.059 (0.059)\\\hline
        $Q_{XX,23}$&$Q_{XX,24}$&$Q_{XX,25}$&$Q_{XX,26}$&\\\hline
        65.004 (65.000) & 143.006 (143.000) &-12.999 (-13.000) & 13.000 (13.000) \\
         0.121 (0.121) & 0.231 (0.232) &0.119 (0.119) & 0.055 (0.055) \\\hline
        $Q_{XX,33}$&$Q_{XX,34}$&$Q_{XX,35}$&$Q_{XX,36}$&\\\hline
         246.016 (246.000) & 377.014 (377.000) & 104.011 (104.000)& -26.004 (-26.000)\\
         0.348 (0.348) & 0.582 (0.581) & 0.302 (0.305) & 0.131 (0.133)\\\hline
        $Q_{XX,44}$&$Q_{XX,45}$&$Q_{XX,46}$&$Q_{XX,55}$\\\hline
        794.019 (794.000) & -25.989 (-26.000) & 51.996 (52.000) & 334.015 (334.000)\\
        1.122 (1.123) & 0.514 (0.516) & 0.238 (0.240) & 0.475 (0.472) \\\hline
        $Q_{XX,56}$&$Q_{XX,66}$\\\hline
        -143.005 (-143.000) & 69.002 (69.000)\\
        0.209 (0.209) & 0.098 (0.098)\\\hline\\
    \end{tabular}
\end{table}
\begin{figure}[h]
\centering
\includegraphics[width=0.15\columnwidth]{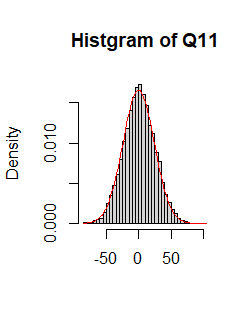}
\includegraphics[width=0.15\columnwidth]{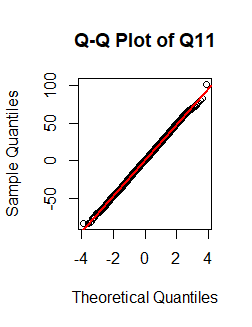}
\includegraphics[width=0.15\columnwidth]{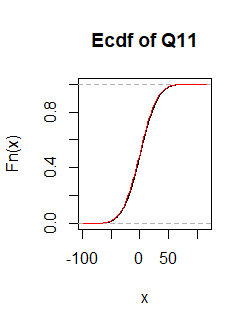}
\quad 
\includegraphics[width=0.15\columnwidth]{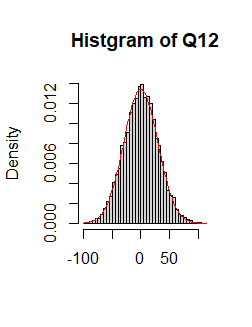}
\includegraphics[width=0.15\columnwidth]{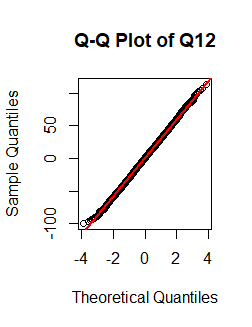}
\includegraphics[width=0.15\columnwidth]{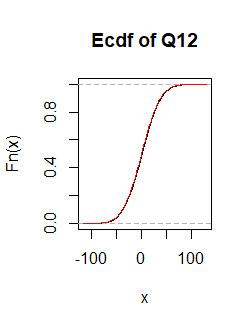}
\includegraphics[width=0.15\columnwidth]{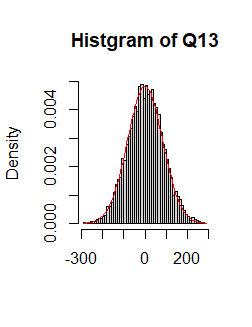}
\includegraphics[width=0.15\columnwidth]{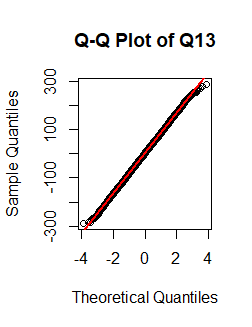}
\includegraphics[width=0.15\columnwidth]{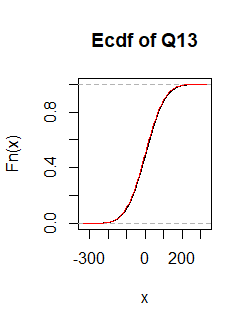}\quad
\includegraphics[width=0.15\columnwidth]{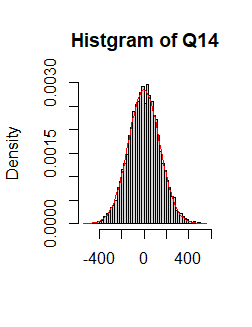}
\includegraphics[width=0.15\columnwidth]{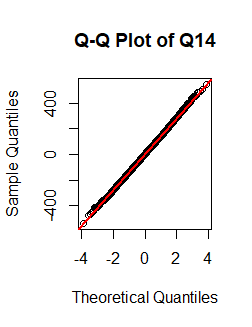}
\includegraphics[width=0.15\columnwidth]{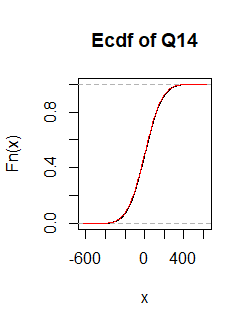}\\
\includegraphics[width=0.15\columnwidth]{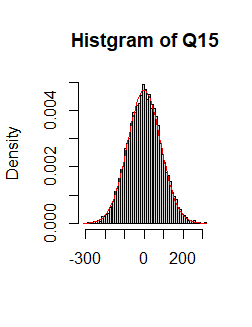}
\includegraphics[width=0.15\columnwidth]{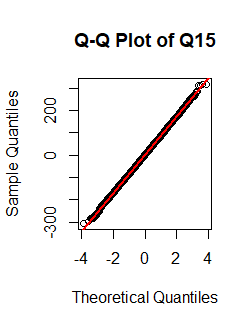}
\includegraphics[width=0.15\columnwidth]{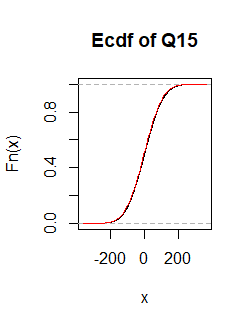}\quad
\includegraphics[width=0.15\columnwidth]{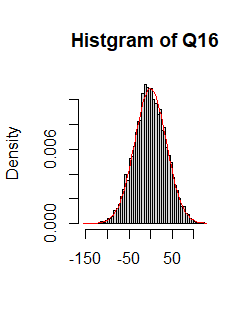}
\includegraphics[width=0.15\columnwidth]{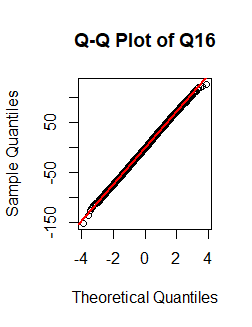}
\includegraphics[width=0.15\columnwidth]{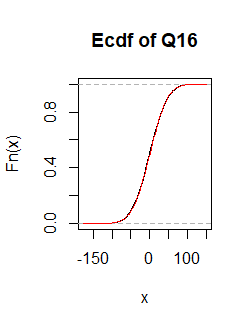}\\
\includegraphics[width=0.15\columnwidth]{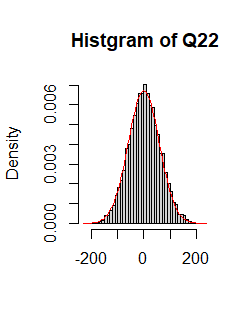}
\includegraphics[width=0.15\columnwidth]{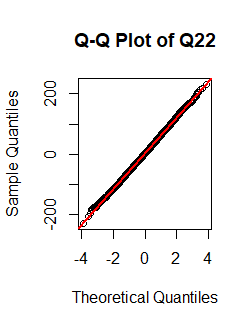}
\includegraphics[width=0.15\columnwidth]{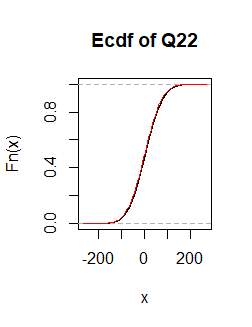}\quad
\includegraphics[width=0.15\columnwidth]{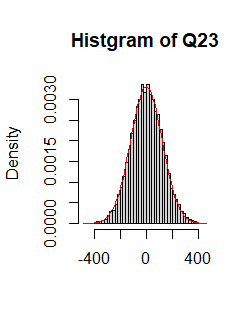}
\includegraphics[width=0.15\columnwidth]{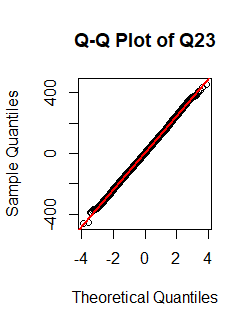}
\includegraphics[width=0.15\columnwidth]{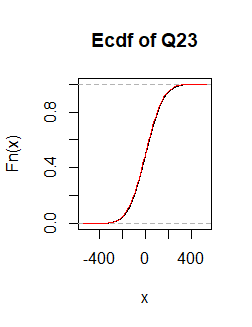}
\end{figure}
\begin{figure}
\includegraphics[width=0.15\columnwidth]{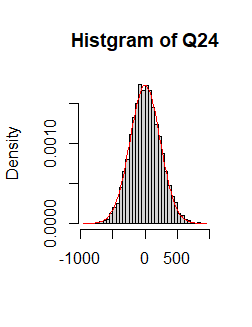}
\includegraphics[width=0.15\columnwidth]{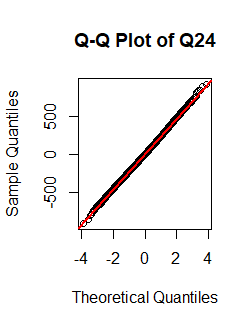}
\includegraphics[width=0.15\columnwidth]{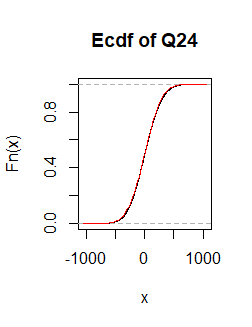}\quad
\includegraphics[width=0.15\columnwidth]{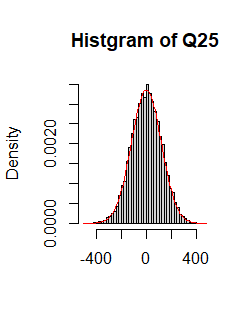}
\includegraphics[width=0.15\columnwidth]{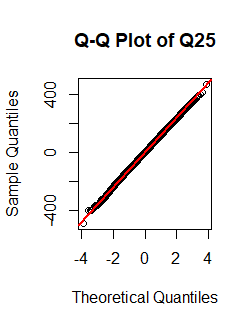}
\includegraphics[width=0.15\columnwidth]{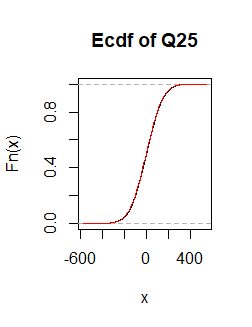}\\
\includegraphics[width=0.15\columnwidth]{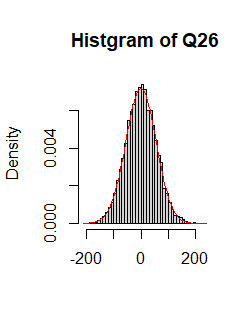}
\includegraphics[width=0.15\columnwidth]{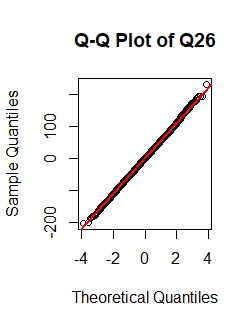}
\includegraphics[width=0.15\columnwidth]{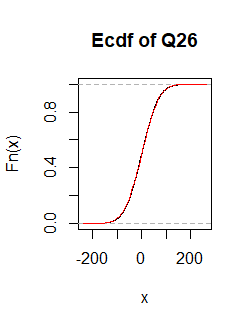}\quad
\includegraphics[width=0.15\columnwidth]{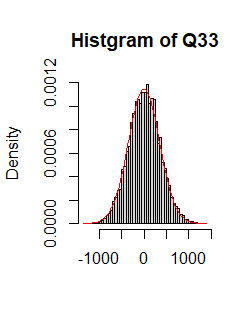}
\includegraphics[width=0.15\columnwidth]{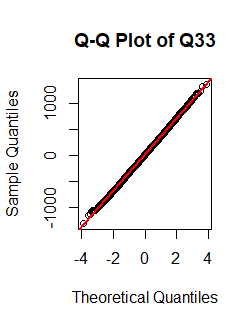}
\includegraphics[width=0.15\columnwidth]{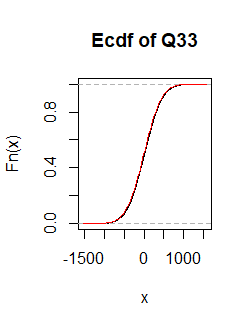}\\
\includegraphics[width=0.15\columnwidth]{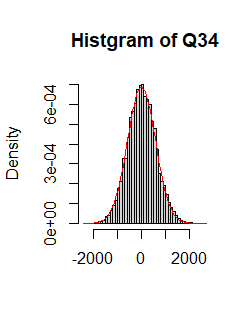}
\includegraphics[width=0.15\columnwidth]{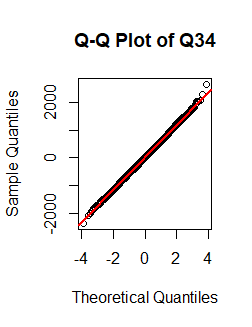}
\includegraphics[width=0.15\columnwidth]{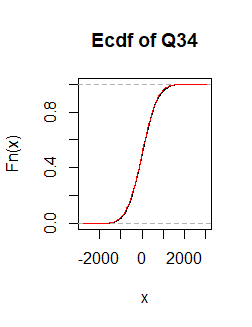}\quad
\includegraphics[width=0.15\columnwidth]{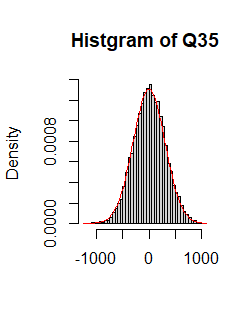}
\includegraphics[width=0.15\columnwidth]{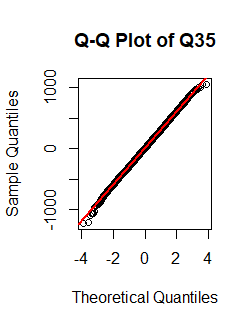}
\includegraphics[width=0.15\columnwidth]{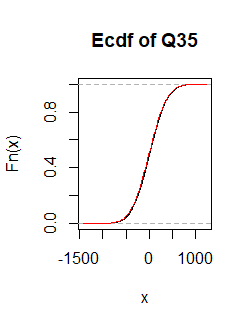}\\
\includegraphics[width=0.15\columnwidth]{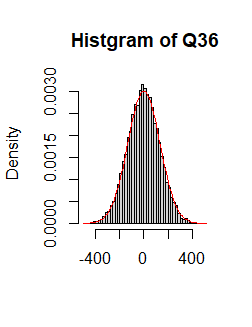}
\includegraphics[width=0.15\columnwidth]{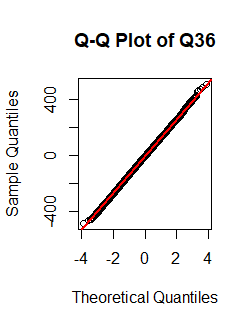}
\includegraphics[width=0.15\columnwidth]{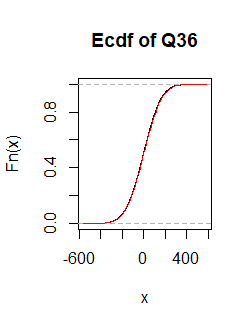}\quad
\includegraphics[width=0.15\columnwidth]{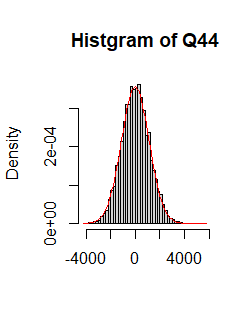}
\includegraphics[width=0.15\columnwidth]{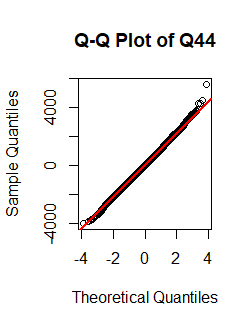}
\includegraphics[width=0.15\columnwidth]{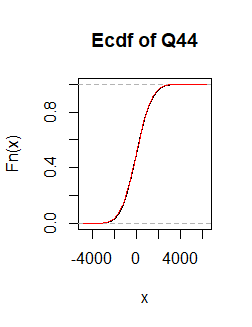}\\
\includegraphics[width=0.15\columnwidth]{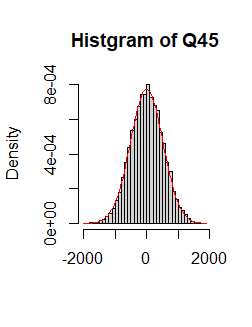}
\includegraphics[width=0.15\columnwidth]{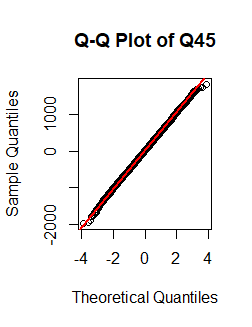}
\includegraphics[width=0.15\columnwidth]{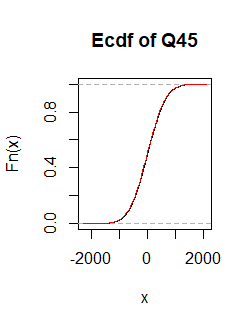}\quad
\includegraphics[width=0.15\columnwidth]{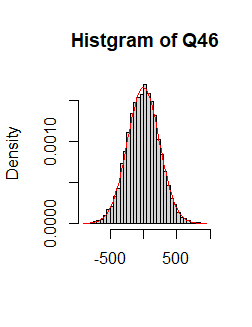}
\includegraphics[width=0.15\columnwidth]{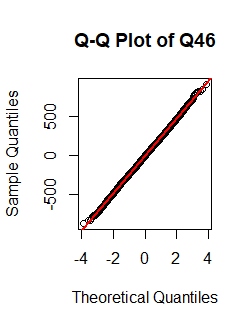}
\includegraphics[width=0.15\columnwidth]{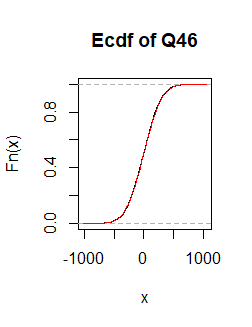}\\
\includegraphics[width=0.15\columnwidth]{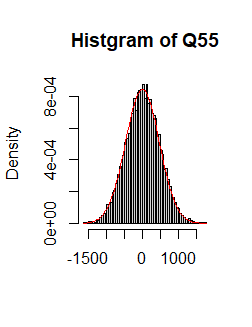}
\includegraphics[width=0.15\columnwidth]{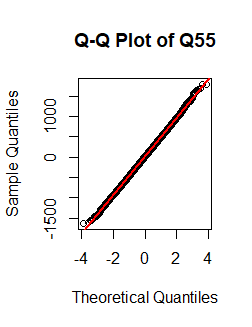}
\includegraphics[width=0.15\columnwidth]{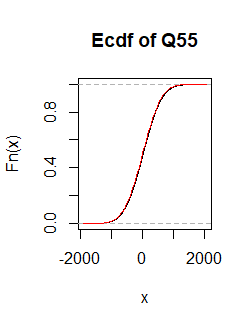}\quad
\includegraphics[width=0.15\columnwidth]{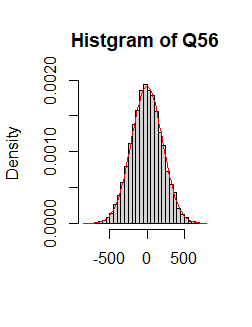}
\includegraphics[width=0.15\columnwidth]{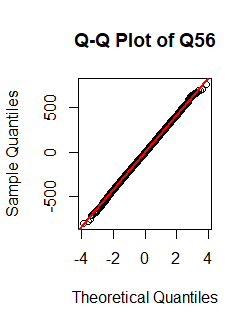}
\includegraphics[width=0.15\columnwidth]{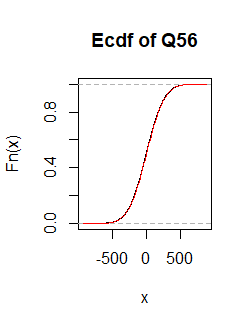}\\
\includegraphics[width=0.15\columnwidth]{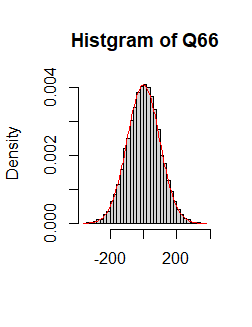}
\includegraphics[width=0.15\columnwidth]{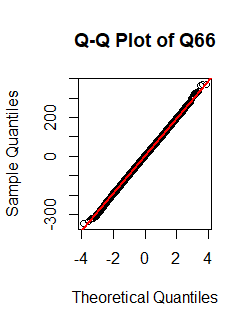}
\includegraphics[width=0.15\columnwidth]{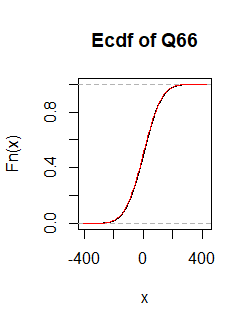}\\
\caption{Histogram (left), Q-Q plot (middle) and empirical distribution (right) of $Q_{XX}$.}
\end{figure}
\clearpage
\newpage
\begin{table}[]
    \centering
    \setlength{\doublerulesep}{0.4pt}
    \caption{Sample mean and sample standard deviation (SD)  of $\hat{\theta}_{2,n}$.}
    \begin{tabular}{cccccc}
        \hline\hline
        &$\hat{\theta}_{2,n}^{(1)}$&$\hat{\theta}_{2,n}^{(2)}$&$\hat{\theta}_{2,n}^{(3)}$&$\hat{\theta}_{2,n}^{(4)}$&\\\hline
        Mean (true value) & 3.000 (3.000) & 1.000 (1.000) & 7.000 (7.000) & -3.000 (-3.000)\\
         SD (theoretical value) & 0.003 (0.003) & 0.008 (0.007) & 0.008 (0.009) & 0.004 (0.004) \\\hline
        $\hat{\theta}_{2,n}^{(5)}$&$\hat{\theta}_{2,n}^{(6)}$&$\hat{\theta}_{2,n}^{(7)}$&$\hat{\theta}_{2,n}^{(8)}$&$\hat{\theta}_{2,n}^{(9)}$&\\\hline
         1.000 (1.000) & 5.000 (5.000) & -4.000 (-4.000) & 2.000 (2.000) & 13.001 (13.000)\\
         0.002 (0.002) & 0.006 (0.006) & 0.006 (0.006) & 0.003 (0.003) & 0.024 (0.024) \\\hline
        $\hat{\theta}_{2,n}^{(10)}$&$\hat{\theta}_{2,n}^{(11)}$&$\hat{\theta}_{2,n}^{(12)}$&$\hat{\theta}_{2,n}^{(13)}$&$\hat{\theta}_{2,n}^{(14)}$&\\\hline
        13.001 (13.000) & 26.002 (26.000) & 4.000 (4.000) & 16.002 (16.000) & 25.004 (25.000)\\
         0.029 (0.029) & 0.055 (0.055) & 0.065 (0.060) & 0.023 (0.023) & 0.049 (0.050)\\\hline
        $\hat{\theta}_{2,n}^{(15)}$&$\hat{\theta}_{2,n}^{(16)}$&$\hat{\theta}_{2,n}^{(17)}$&\\\hline
         1.001 (1.000)& 9.003 (9.000) & 4.000 (4.000) &\\
        0.134 (0.136)& 0.061 (0.061) & 0.013 (0.013) &\\\hline\\
    \end{tabular}
\end{table}
\begin{figure}[h]
\centering
\includegraphics[width=0.15\columnwidth]{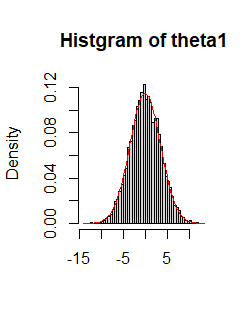}
\includegraphics[width=0.15\columnwidth]{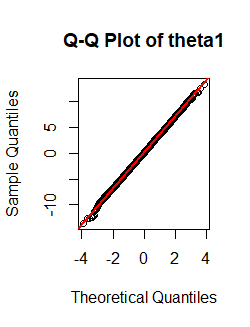}
\includegraphics[width=0.15\columnwidth]{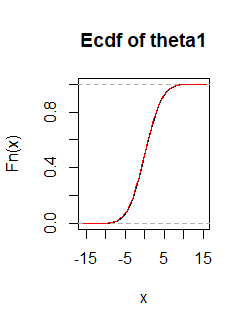}
\quad 
\includegraphics[width=0.15\columnwidth]{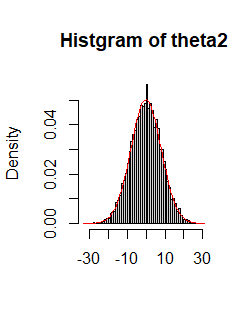}
\includegraphics[width=0.15\columnwidth]{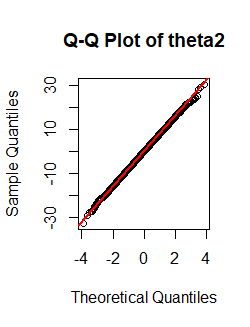}
\includegraphics[width=0.15\columnwidth]{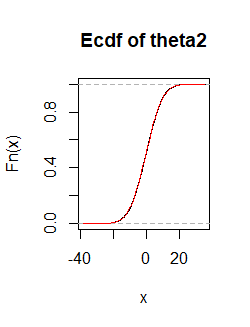}
\\
\includegraphics[width=0.15\columnwidth]{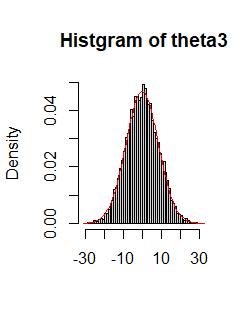}
\includegraphics[width=0.15\columnwidth]{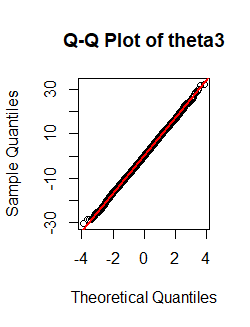}
\includegraphics[width=0.15\columnwidth]{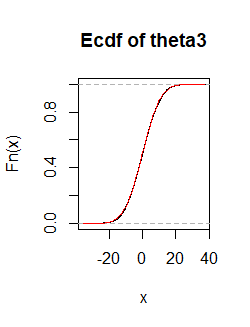}
\quad
\includegraphics[width=0.15\columnwidth]{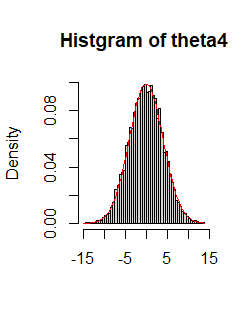}
\includegraphics[width=0.15\columnwidth]{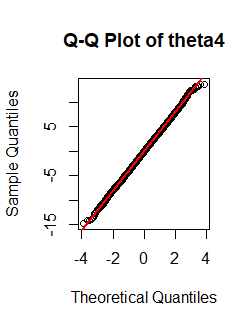}
\includegraphics[width=0.15\columnwidth]{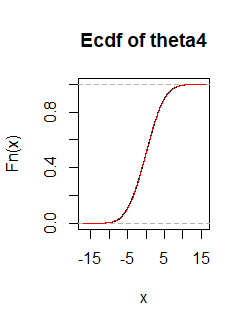}
\\
\includegraphics[width=0.15\columnwidth]{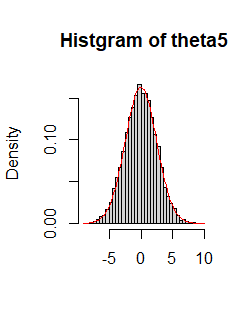}
\includegraphics[width=0.15\columnwidth]{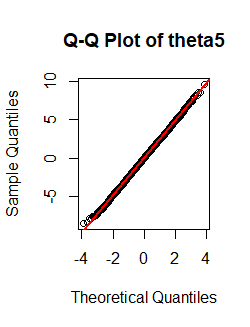}
\includegraphics[width=0.15\columnwidth]{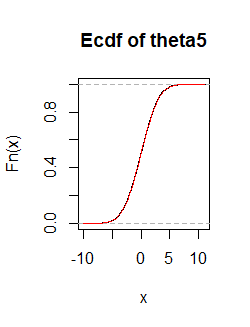}
\quad
\includegraphics[width=0.15\columnwidth]{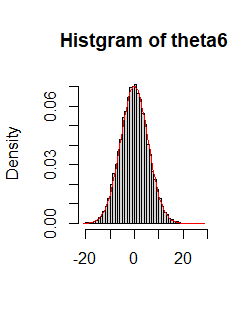}
\includegraphics[width=0.15\columnwidth]{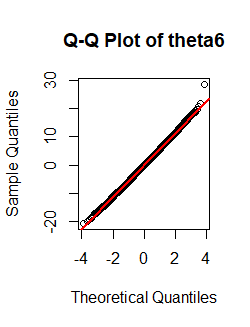}
\includegraphics[width=0.15\columnwidth]{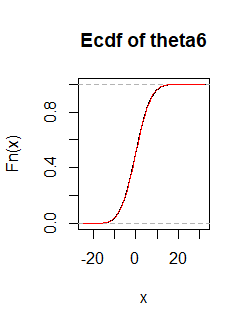}
\\
\includegraphics[width=0.15\columnwidth]{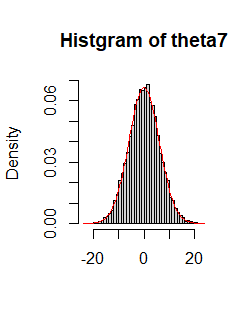}
\includegraphics[width=0.15\columnwidth]{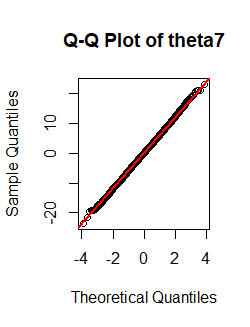}
\includegraphics[width=0.15\columnwidth]{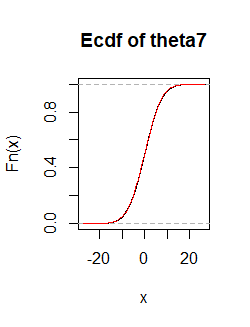}
\quad
\includegraphics[width=0.15\columnwidth]{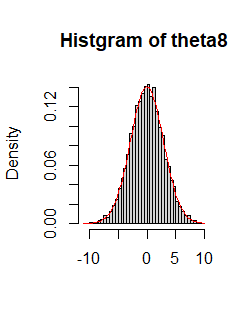}
\includegraphics[width=0.15\columnwidth]{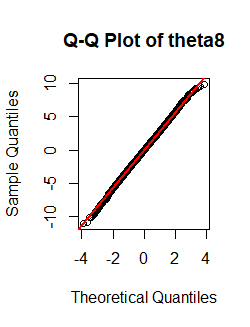}
\includegraphics[width=0.15\columnwidth]{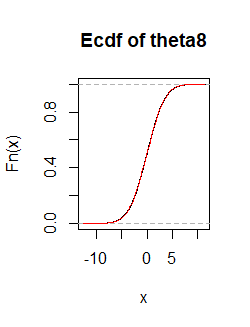}
\end{figure}
\begin{figure}
\includegraphics[width=0.15\columnwidth]{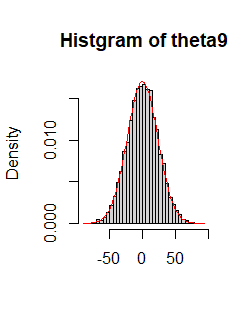}
\includegraphics[width=0.15\columnwidth]{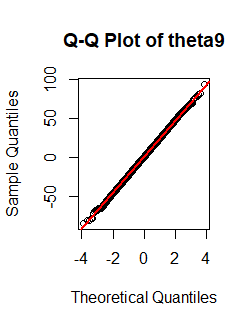}
\includegraphics[width=0.15\columnwidth]{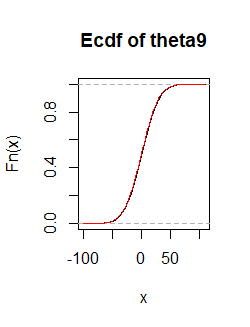}
\quad
\includegraphics[width=0.15\columnwidth]{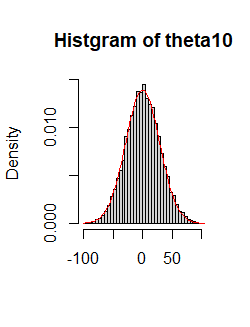}
\includegraphics[width=0.15\columnwidth]{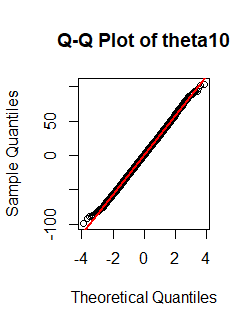}
\includegraphics[width=0.15\columnwidth]{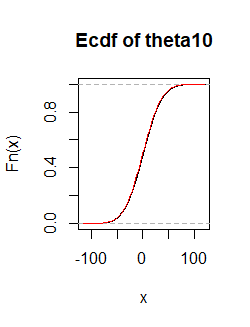}
\\
\includegraphics[width=0.15\columnwidth]{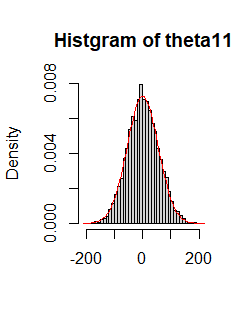}
\includegraphics[width=0.15\columnwidth]{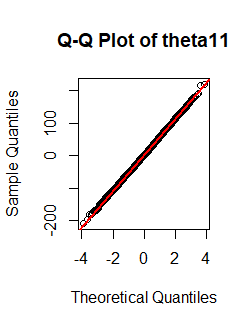}
\includegraphics[width=0.15\columnwidth]{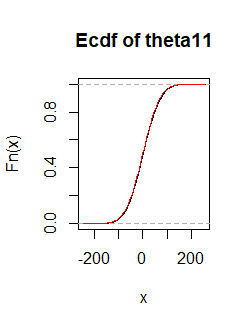}
\quad
\includegraphics[width=0.15\columnwidth]{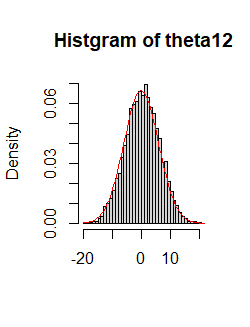}
\includegraphics[width=0.15\columnwidth]{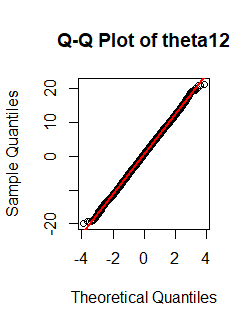}
\includegraphics[width=0.15\columnwidth]{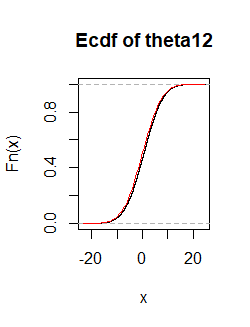}\\
\includegraphics[width=0.15\columnwidth]{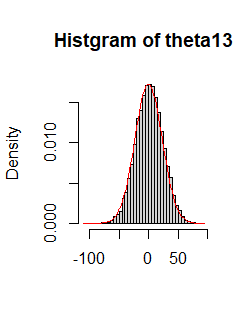}
\includegraphics[width=0.15\columnwidth]{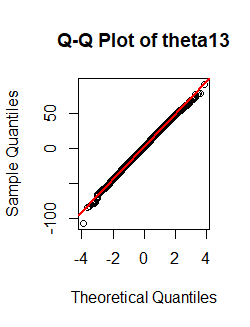}
\includegraphics[width=0.15\columnwidth]{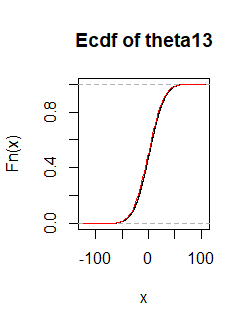}
\quad
\includegraphics[width=0.15\columnwidth]{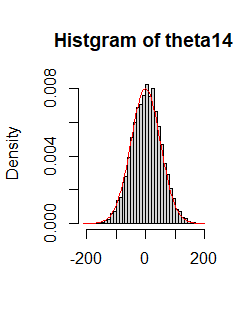}
\includegraphics[width=0.15\columnwidth]{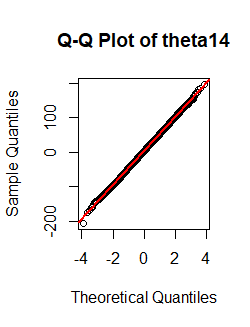}
\includegraphics[width=0.15\columnwidth]{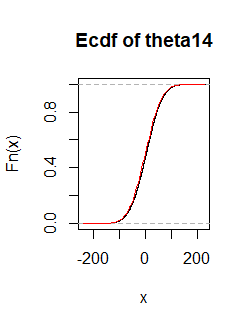}\\
\includegraphics[width=0.15\columnwidth]{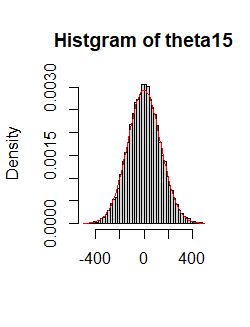}
\includegraphics[width=0.15\columnwidth]{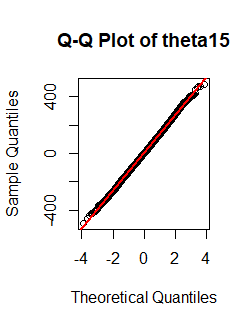}
\includegraphics[width=0.15\columnwidth]{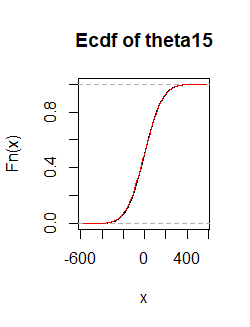}\quad
\includegraphics[width=0.15\columnwidth]{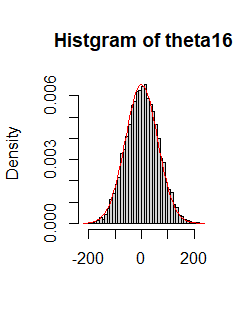}
\includegraphics[width=0.15\columnwidth]{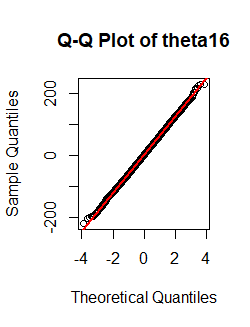}
\includegraphics[width=0.15\columnwidth]{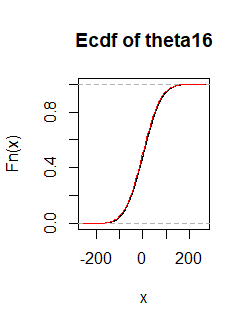}
\\
\includegraphics[width=0.15\columnwidth]{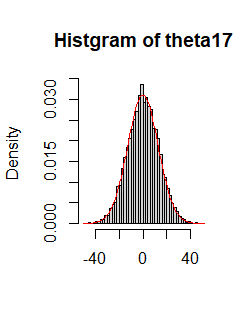}
\includegraphics[width=0.15\columnwidth]{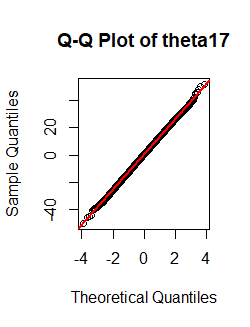}
\includegraphics[width=0.15\columnwidth]{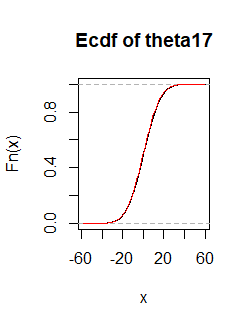}
\caption{Histogram (left), Q-Q plot (middle) and empirical distribution (right) of $\hat{\theta}_{2,n}$. }
\end{figure}
\begin{table}[h]
    \centering
    \setlength{\doublerulesep}{0.4pt}
    \caption{Quartile of $T_{1,n}$.} 
    \begin{tabular}{ccccc}
    \hline\hline
    Min&$Q_{1}$&$Q_{2}$ (Median) &$Q_{3}$&Max\\\hline
    1208126 & 1210011 & 1210451 & 1210880 & 1213047\\\hline
    \end{tabular} 
\end{table}
\begin{table}[h]
    \centering
    \setlength{\doublerulesep}{0.4pt}
    \caption{Sample mean and sample standard deviation (SD)  of $T_{2,n}$.} 
    \begin{tabular}{lcc}
    \hline\hline& sample mean (true value) & sample SD (theoretical value)\\\hline
     &  3.936 (4.000)& 2.800 (2.828)\\\hline
    \end{tabular} 
\end{table}
\begin{table}[h]
    \centering
    \setlength{\doublerulesep}{0.4pt}
    \caption{The number of rejections of the test with $H_0$: $k=1$, $k=2$.} 
    \begin{tabular}{ccc}
    \hline\hline
    $H_0$&$k=1$&$k=2$ \\\hline
    The number of rejections & 10000 & 472\\\hline
    \end{tabular} 
\end{table}
\begin{figure}[h]
\begin{center}
\includegraphics[width=0.3\columnwidth]{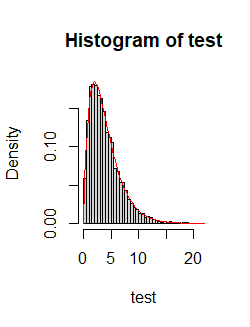}
\includegraphics[width=0.3\columnwidth]{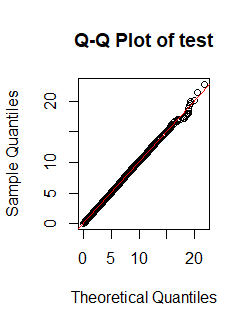}
\includegraphics[width=0.3\columnwidth]{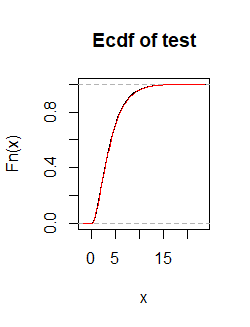}
\caption{Sample mean and sample standard deviation (SD)  of $T_{2,n}$.}
\end{center}
\end{figure}
\clearpage

\begin{table}[]
    \centering
    \setlength{\doublerulesep}{0.4pt}
    \caption{Sample mean and sample standard deviation (SD)  of $Q_{XX}$.}
    \begin{tabular}{ccccc}
        \hline\hline
        & $Q_{XX,11}$&$Q_{XX,12}$&$Q_{XX,13}$&\\\hline
        Mean (true value) & 17.027 (17.000) & 13.004 (13.000) & 52.054 (52.000) \\
         SD (theoretical value) & 0.762 (0.760) & 0.940 (0.940) & 2.634 (2.624)\\\hline
        $Q_{XX,14}$&$Q_{XX,15}$&$Q_{XX,16}$&$Q_{XX,22}$\\\hline
         78.060 (78.000) &39.078 (39.000) & -13.029 (-13.000) & 42.005 (42.000) \\
         4.419 (4.425) &2.701 (2.683) & 1.161 (1.158) & 1.847 (1.878)\\\hline
        $Q_{XX,23}$&$Q_{XX,24}$&$Q_{XX,25}$&$Q_{XX,26}$&\\\hline
        64.991 (65.000) & 142.958 (143.000) &-12.967 (-13.000) & 12.991 (13.000) \\
         3.793 (3.815) & 7.278 (7.335) &3.767 (3.768) & 1.746 (1.751) \\\hline
        $Q_{XX,33}$&$Q_{XX,34}$&$Q_{XX,35}$&$Q_{XX,36}$&\\\hline
         246.160 (246.000) & 377.058 (377.000) & 104.240 (104.000)& -26.080 (-26.000)\\
         10.996 (11.001) & 18.329 (18.370) & 9.747 (9.643) & 4.226 (4.201)\\\hline
        $Q_{XX,44}$&$Q_{XX,45}$&$Q_{XX,46}$&$Q_{XX,55}$\\\hline
        793.898 (794.000) & -25.742 (-26.000) & 51.914 (52.000) & 334.362 (334.000)\\
        35.434 (35.509) & 16.391 (16.306) & 7.594 (7.582) & 14.863 (14.937) \\\hline
        $Q_{XX,56}$&$Q_{XX,66}$\\\hline
        -143.143 (-143.000) & 69.063 (69.000)\\
        6.552 (6.600) & 3.073 (3.086)\\\hline\\
    \end{tabular}
\end{table}
\begin{figure}[h]
\centering
\includegraphics[width=0.15\columnwidth]{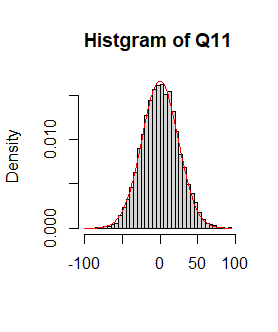}
\includegraphics[width=0.15\columnwidth]{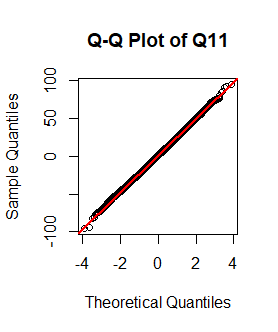}
\includegraphics[width=0.15\columnwidth]{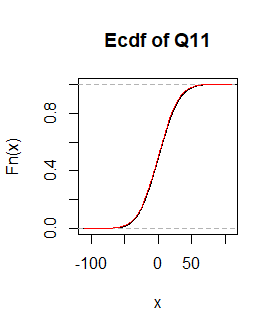}
\quad 
\includegraphics[width=0.15\columnwidth]{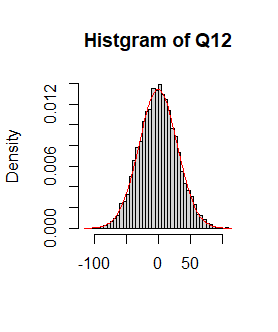}
\includegraphics[width=0.15\columnwidth]{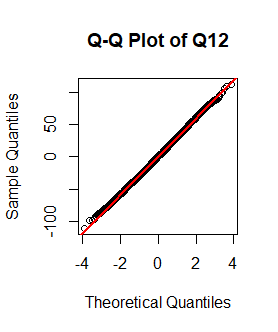}
\includegraphics[width=0.15\columnwidth]{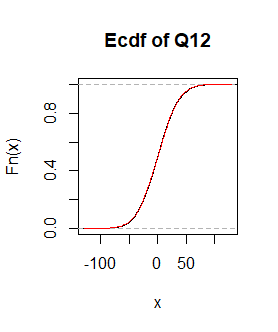}\\
\includegraphics[width=0.15\columnwidth]{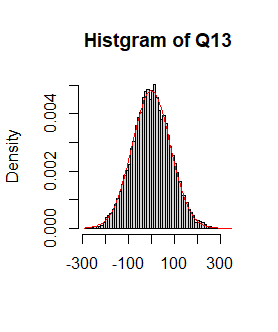}
\includegraphics[width=0.15\columnwidth]{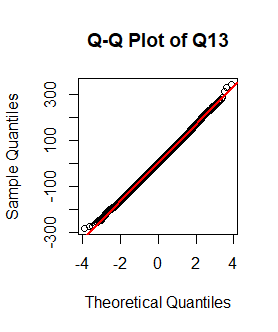}
\includegraphics[width=0.15\columnwidth]{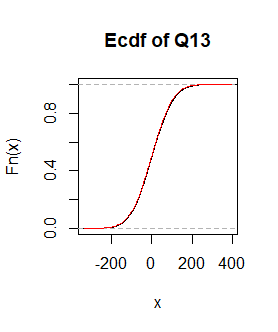}\quad
\includegraphics[width=0.15\columnwidth]{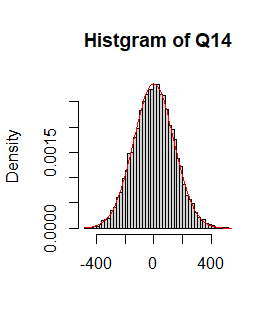}
\includegraphics[width=0.15\columnwidth]{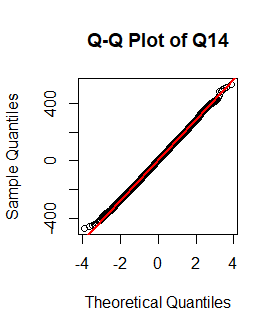}
\includegraphics[width=0.15\columnwidth]{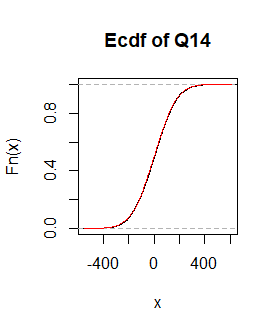}\\
\includegraphics[width=0.15\columnwidth]{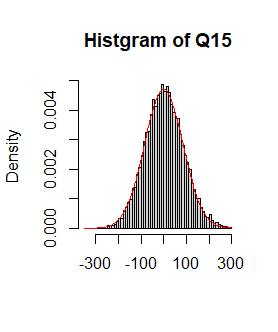}
\includegraphics[width=0.15\columnwidth]{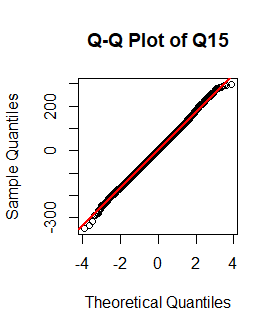}
\includegraphics[width=0.15\columnwidth]{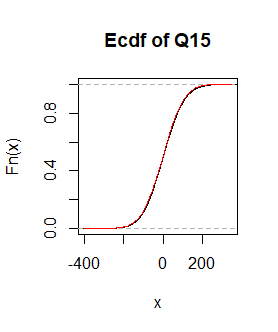}\quad
\includegraphics[width=0.15\columnwidth]{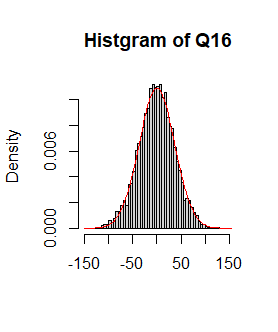}
\includegraphics[width=0.15\columnwidth]{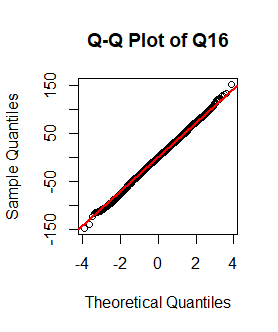}
\includegraphics[width=0.15\columnwidth]{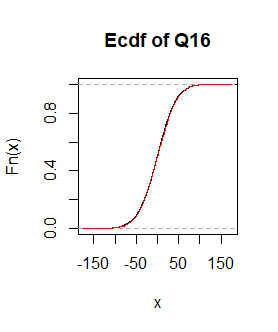}\\
\includegraphics[width=0.15\columnwidth]{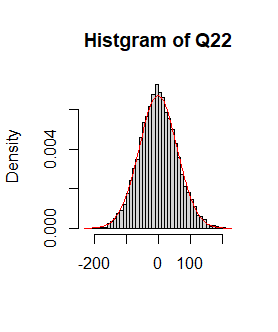}
\includegraphics[width=0.15\columnwidth]{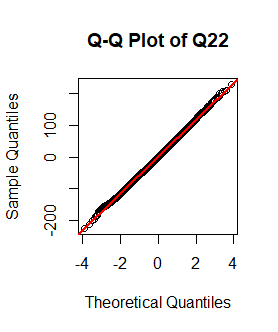}
\includegraphics[width=0.15\columnwidth]{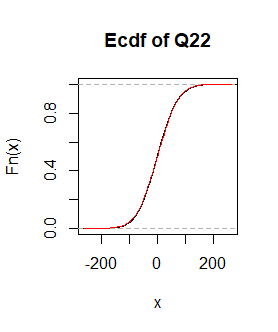}\quad
\includegraphics[width=0.15\columnwidth]{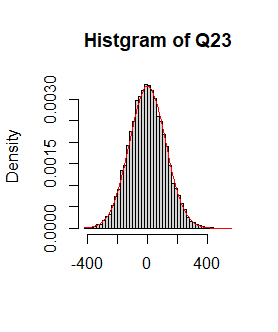}
\includegraphics[width=0.15\columnwidth]{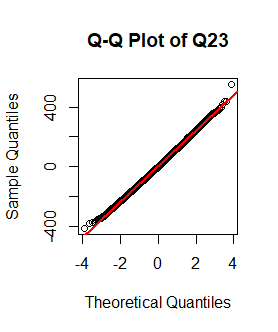}
\includegraphics[width=0.15\columnwidth]{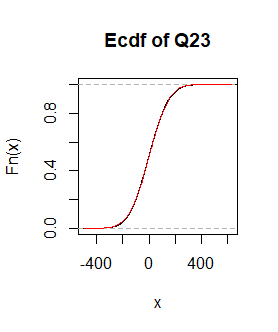}
\end{figure}
\begin{figure}
\includegraphics[width=0.15\columnwidth]{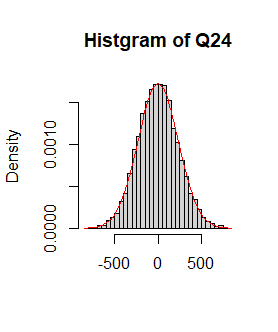}
\includegraphics[width=0.15\columnwidth]{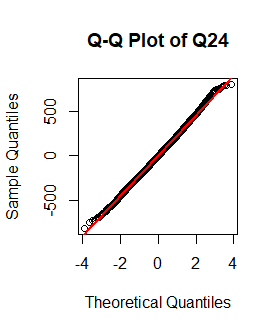}
\includegraphics[width=0.15\columnwidth]{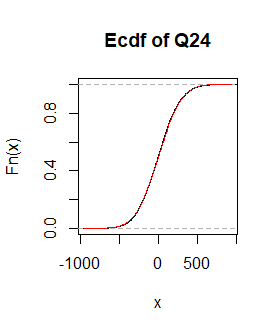}\quad
\includegraphics[width=0.15\columnwidth]{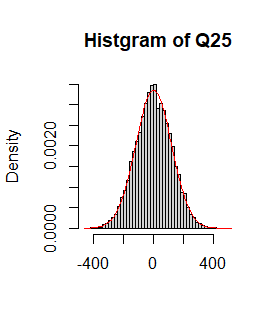}
\includegraphics[width=0.15\columnwidth]{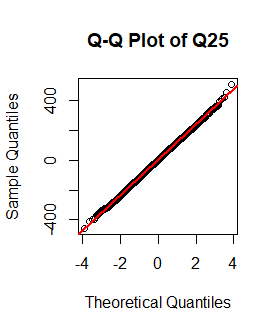}
\includegraphics[width=0.15\columnwidth]{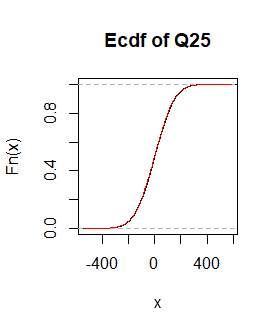}\\
\includegraphics[width=0.15\columnwidth]{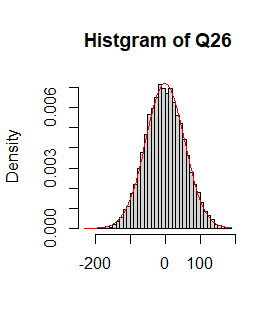}
\includegraphics[width=0.15\columnwidth]{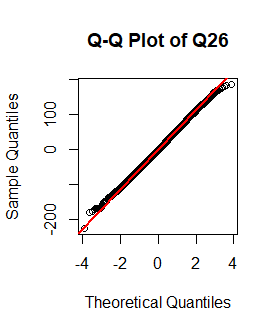}
\includegraphics[width=0.15\columnwidth]{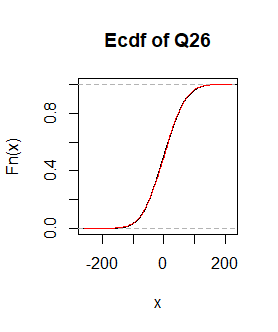}\quad
\includegraphics[width=0.15\columnwidth]{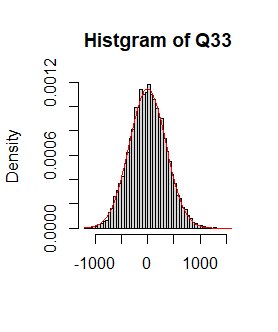}
\includegraphics[width=0.15\columnwidth]{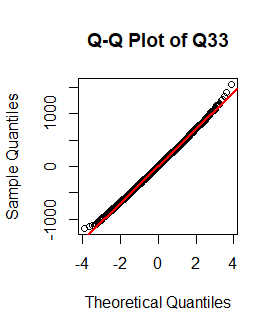}
\includegraphics[width=0.15\columnwidth]{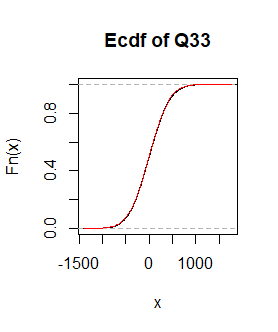}\\
\includegraphics[width=0.15\columnwidth]{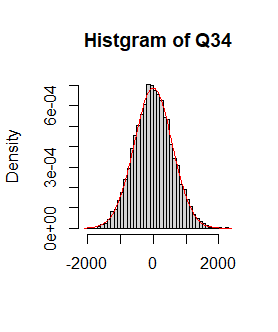}
\includegraphics[width=0.15\columnwidth]{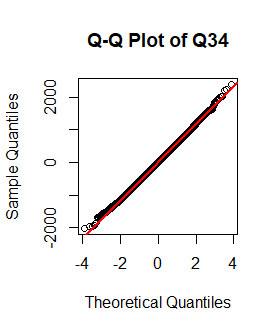}
\includegraphics[width=0.15\columnwidth]{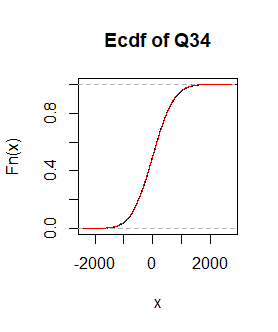}\quad
\includegraphics[width=0.15\columnwidth]{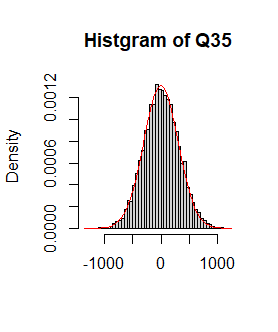}
\includegraphics[width=0.15\columnwidth]{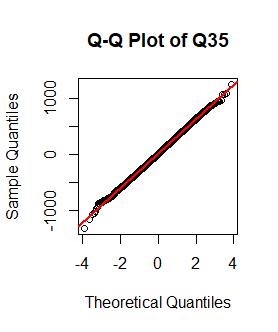}
\includegraphics[width=0.15\columnwidth]{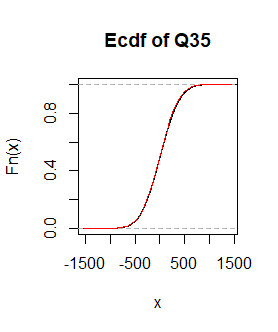}\\
\includegraphics[width=0.15\columnwidth]{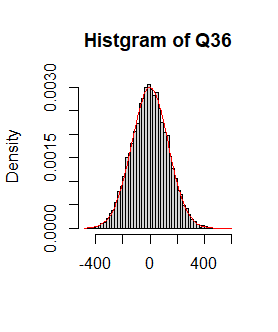}
\includegraphics[width=0.15\columnwidth]{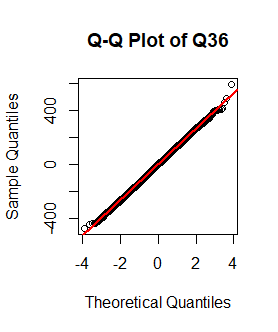}
\includegraphics[width=0.15\columnwidth]{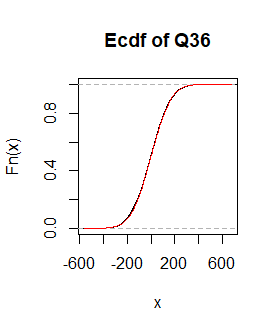}\quad
\includegraphics[width=0.15\columnwidth]{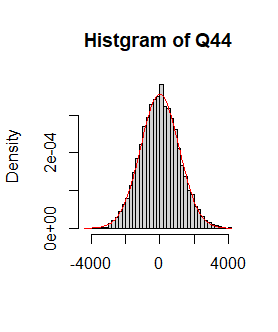}
\includegraphics[width=0.15\columnwidth]{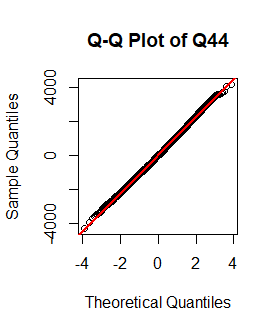}
\includegraphics[width=0.15\columnwidth]{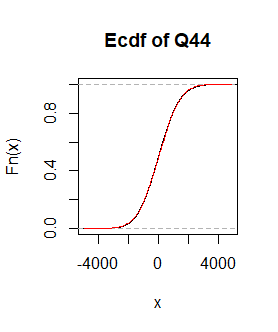}
\includegraphics[width=0.15\columnwidth]{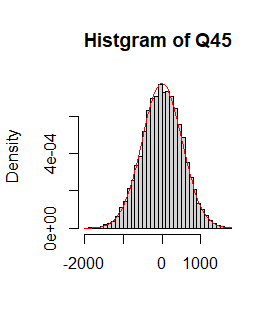}
\includegraphics[width=0.15\columnwidth]{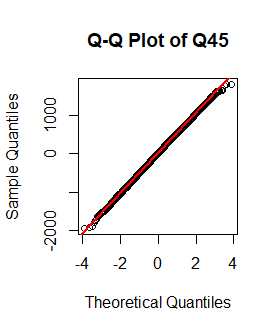}
\includegraphics[width=0.15\columnwidth]{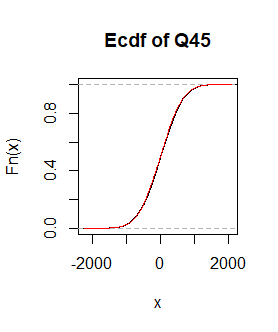}\quad
\includegraphics[width=0.15\columnwidth]{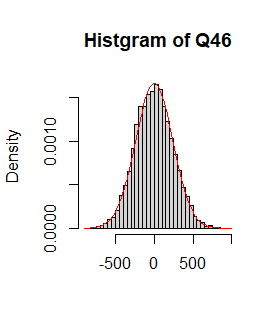}
\includegraphics[width=0.15\columnwidth]{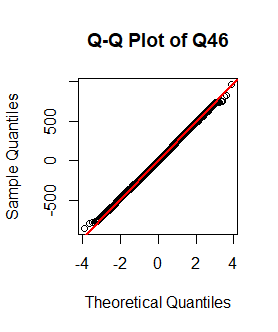}
\includegraphics[width=0.15\columnwidth]{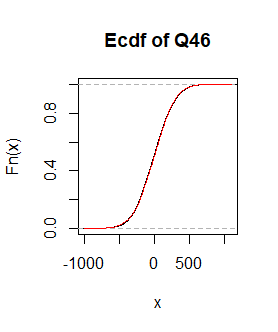}\\
\includegraphics[width=0.15\columnwidth]{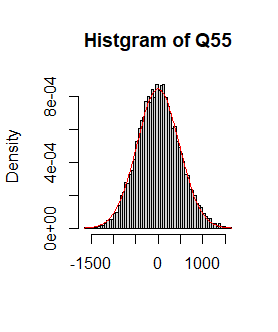}
\includegraphics[width=0.15\columnwidth]{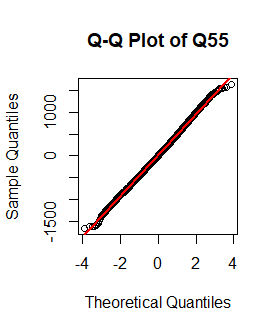}
\includegraphics[width=0.15\columnwidth]{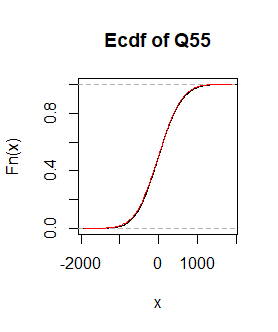}\quad
\includegraphics[width=0.15\columnwidth]{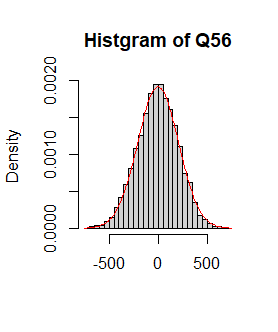}
\includegraphics[width=0.15\columnwidth]{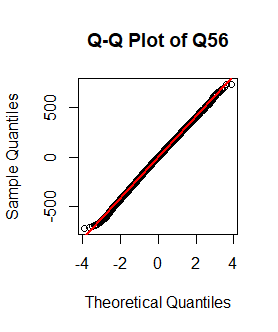}
\includegraphics[width=0.15\columnwidth]{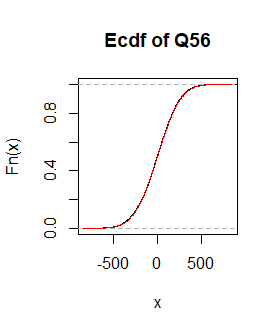}\\
\includegraphics[width=0.15\columnwidth]{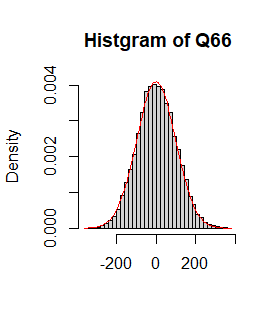}
\includegraphics[width=0.15\columnwidth]{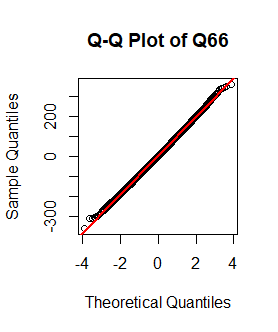}
\includegraphics[width=0.15\columnwidth]{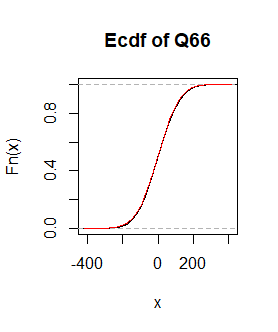}\\
\caption{Histogram (left), Q-Q plot (middle) and empirical distribution (right) of $Q_{XX}$.}
\end{figure}

\clearpage

\begin{table}[]
    \centering
    \setlength{\doublerulesep}{0.4pt}
    \caption{Sample mean and sample standard deviation (SD)  of $\hat{\theta}_{2,n}$.}
    \begin{tabular}{cccccc}
        \hline\hline
        &$\hat{\theta}_{2,n}^{(1)}$&$\hat{\theta}_{2,n}^{(2)}$&$\hat{\theta}_{2,n}^{(3)}$&$\hat{\theta}_{2,n}^{(4)}$&\\\hline
        Mean (true value) & 3.000 (3.000) & 0.995 (1.000) & 7.008 (7.000) & -3.003 (-3.000)\\
         SD (theoretical value) & 0.110 (0.109) & 0.253 (0.252) & 0.271 (0.270) & 0.128 (0.128) \\\hline
        $\hat{\theta}_{2,n}^{(5)}$&$\hat{\theta}_{2,n}^{(6)}$&$\hat{\theta}_{2,n}^{(7)}$&$\hat{\theta}_{2,n}^{(8)}$&$\hat{\theta}_{2,n}^{(9)}$&\\\hline
         1.000 (1.000) & 5.001 (5.000) & -4.009 (-4.000) & 2.004 (2.000) & 13.027 (13.000)\\
         0.078 (0.077) & 0.179 (0.179) & 0.189 (0.190) & 0.090 (0.090) & 0.743 (0.742) \\\hline
        $\hat{\theta}_{2,n}^{(10)}$&$\hat{\theta}_{2,n}^{(11)}$&$\hat{\theta}_{2,n}^{(12)}$&$\hat{\theta}_{2,n}^{(13)}$&$\hat{\theta}_{2,n}^{(14)}$&\\\hline
        13.004 (13.000) & 26.016 (26.000) & 4.007 (4.000) & 16.030 (16.000) & 25.068 (25.000)\\
         0.906 (0.909) & 1.717 (1.743) & 0.192 (0.190) & 0.730 (0.730) & 1.573 (1.574)\\\hline
        $\hat{\theta}_{2,n}^{(15)}$&$\hat{\theta}_{2,n}^{(16)}$&$\hat{\theta}_{2,n}^{(17)}$&\\\hline
         0.848 (1.000)& 8.996 (9.000) & 4.009 (4.000) &\\
        4.190 (4.308)& 1.958 (1.932) & 0.412 (0.406) &\\\hline\\
    \end{tabular}
\end{table}
\begin{figure}[h]
\centering
\includegraphics[width=0.15\columnwidth]{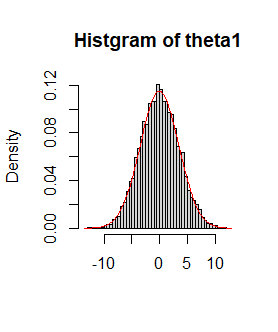}
\includegraphics[width=0.15\columnwidth]{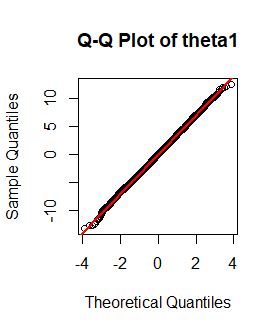}
\includegraphics[width=0.15\columnwidth]{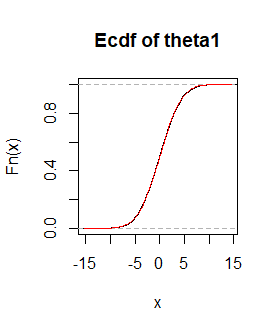}
\quad 
\includegraphics[width=0.15\columnwidth]{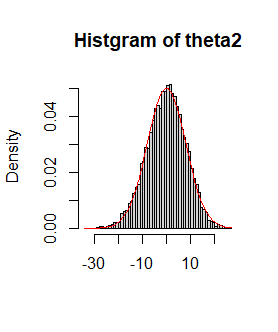}
\includegraphics[width=0.15\columnwidth]{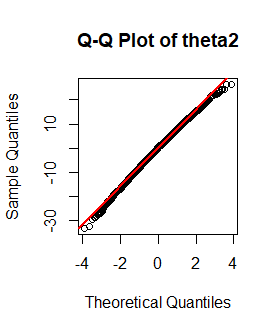}
\includegraphics[width=0.15\columnwidth]{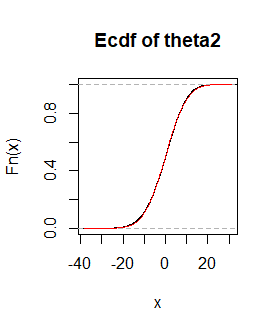}\\
\includegraphics[width=0.15\columnwidth]{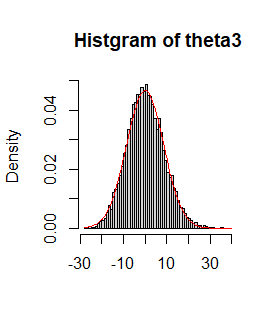}
\includegraphics[width=0.15\columnwidth]{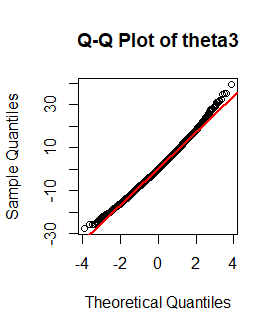}
\includegraphics[width=0.15\columnwidth]{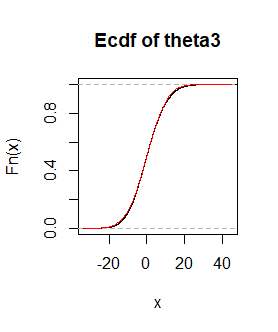}
\quad
\includegraphics[width=0.15\columnwidth]{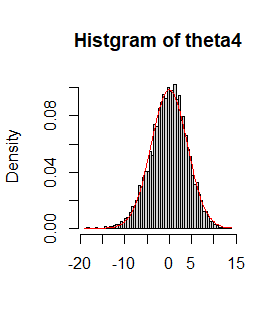}
\includegraphics[width=0.15\columnwidth]{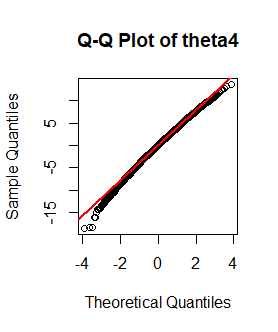}
\includegraphics[width=0.15\columnwidth]{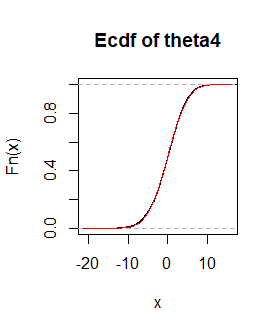}\\
\includegraphics[width=0.15\columnwidth]{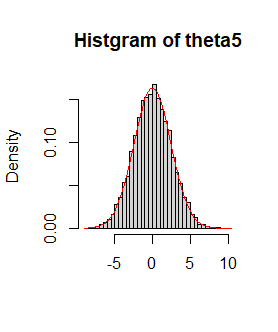}
\includegraphics[width=0.15\columnwidth]{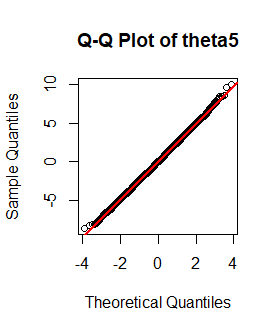}
\includegraphics[width=0.15\columnwidth]{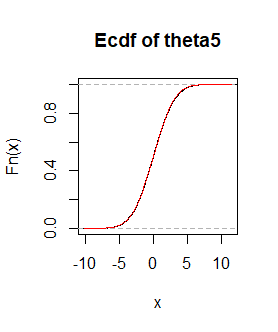}
\quad
\includegraphics[width=0.15\columnwidth]{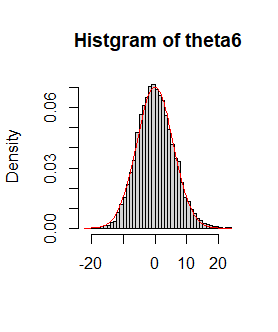}
\includegraphics[width=0.15\columnwidth]{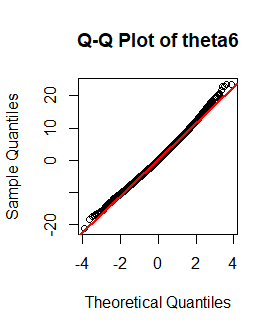}
\includegraphics[width=0.15\columnwidth]{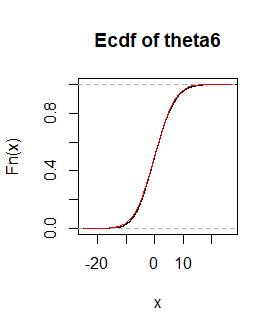}
\\
\includegraphics[width=0.15\columnwidth]{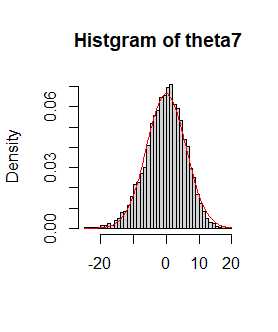}
\includegraphics[width=0.15\columnwidth]{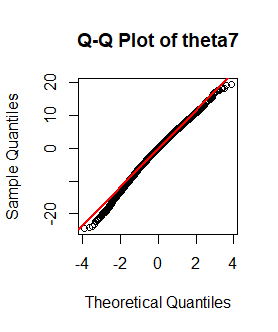}
\includegraphics[width=0.15\columnwidth]{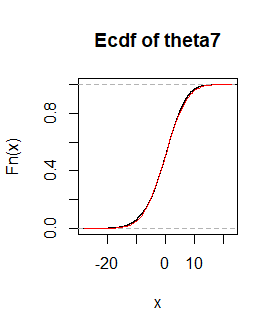}
\quad
\includegraphics[width=0.15\columnwidth]{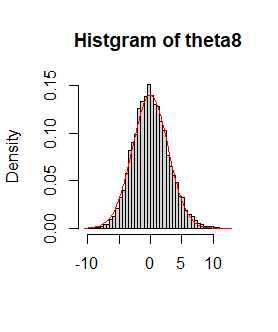}
\includegraphics[width=0.15\columnwidth]{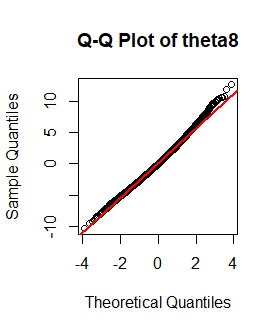}
\includegraphics[width=0.15\columnwidth]{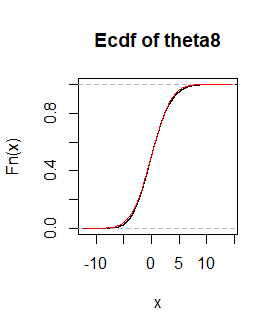}
\end{figure}
\clearpage
\begin{figure}
\includegraphics[width=0.15\columnwidth]{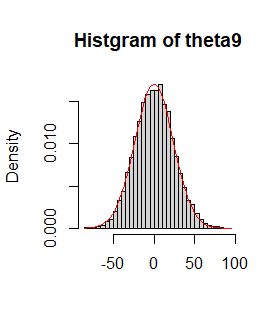}
\includegraphics[width=0.15\columnwidth]{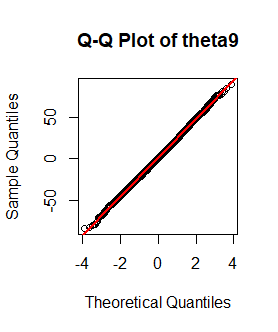}
\includegraphics[width=0.15\columnwidth]{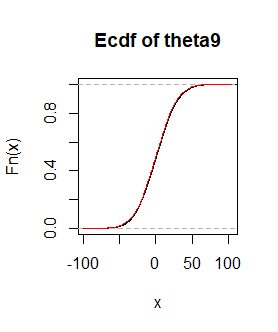}
\quad
\includegraphics[width=0.15\columnwidth]{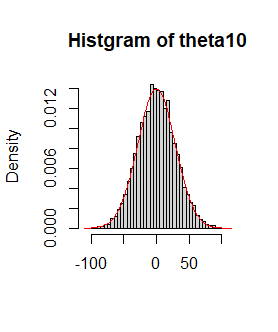}
\includegraphics[width=0.15\columnwidth]{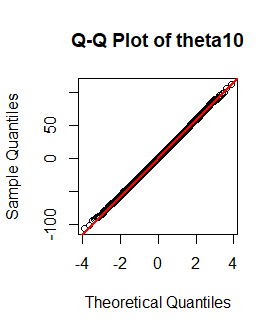}
\includegraphics[width=0.15\columnwidth]{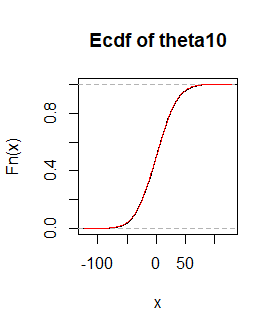}
\\
\includegraphics[width=0.15\columnwidth]{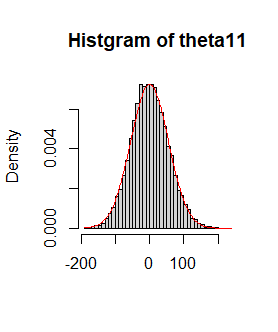}
\includegraphics[width=0.15\columnwidth]{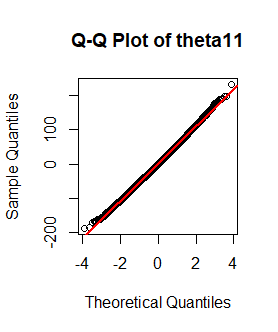}
\includegraphics[width=0.15\columnwidth]{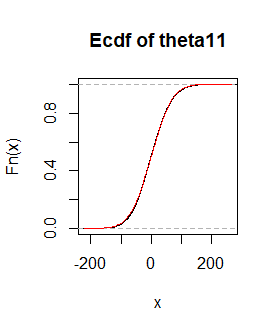}
\quad
\includegraphics[width=0.15\columnwidth]{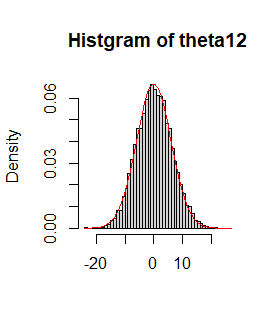}
\includegraphics[width=0.15\columnwidth]{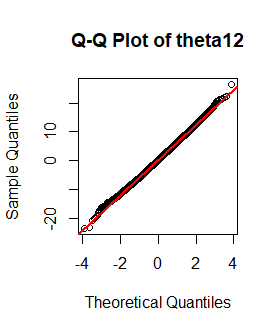}
\includegraphics[width=0.15\columnwidth]{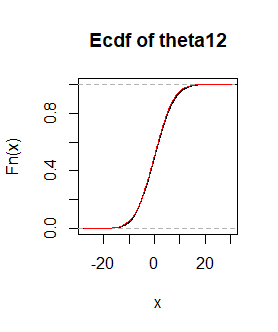}\\
\includegraphics[width=0.15\columnwidth]{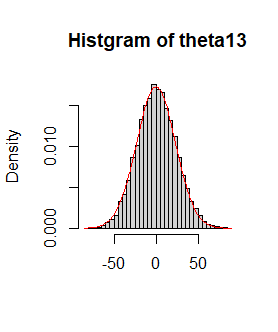}
\includegraphics[width=0.15\columnwidth]{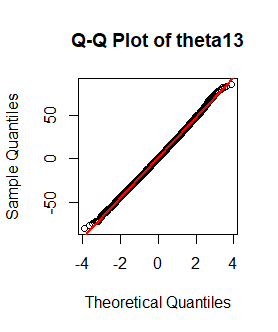}
\includegraphics[width=0.15\columnwidth]{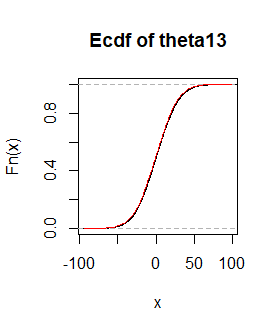}
\quad
\includegraphics[width=0.15\columnwidth]{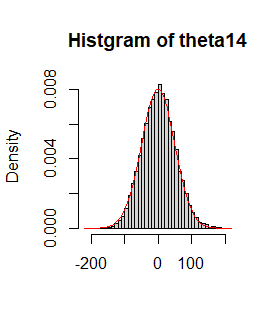}
\includegraphics[width=0.15\columnwidth]{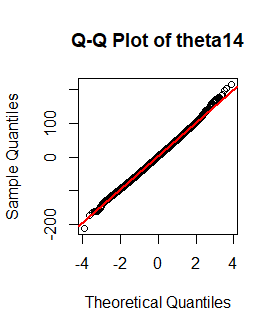}
\includegraphics[width=0.15\columnwidth]{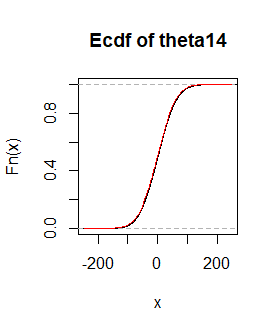}\\
\includegraphics[width=0.15\columnwidth]{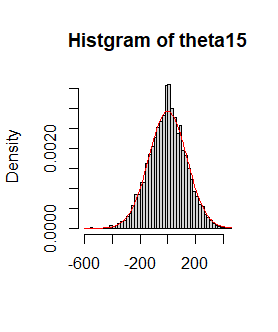}
\includegraphics[width=0.15\columnwidth]{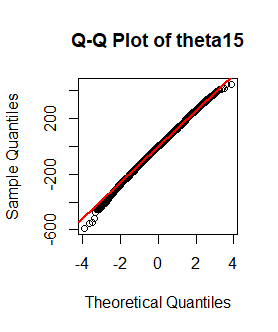}
\includegraphics[width=0.15\columnwidth]{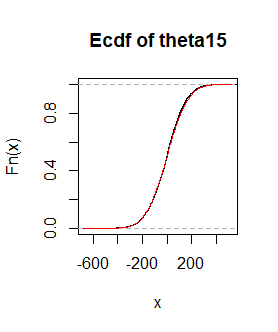}\quad
\includegraphics[width=0.15\columnwidth]{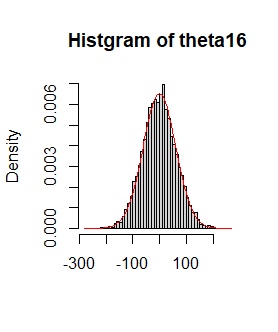}
\includegraphics[width=0.15\columnwidth]{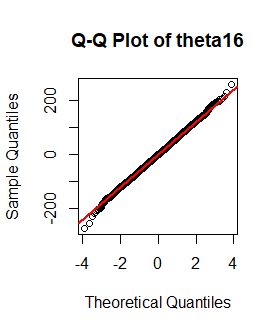}
\includegraphics[width=0.15\columnwidth]{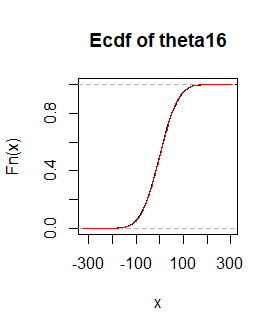}
\end{figure}
\begin{figure}
\includegraphics[width=0.15\columnwidth]{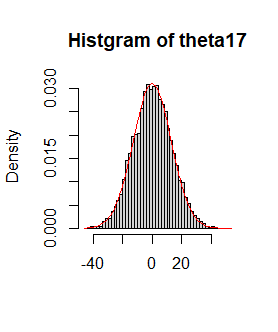}
\includegraphics[width=0.15\columnwidth]{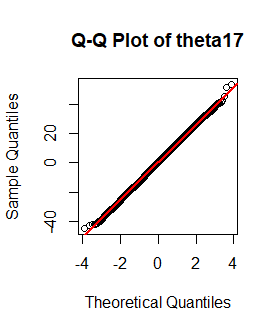}
\includegraphics[width=0.15\columnwidth]{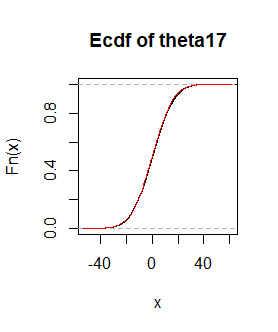}
\caption{Histogram (left), Q-Q plot (middle) and empirical distribution (right) of $\hat{\theta}_{2,n}$. }
\end{figure}
\begin{table}[h]
    \centering
    \setlength{\doublerulesep}{0.4pt}
    \caption{Quartile of $T_{1,n}$.} 
    \begin{tabular}{ccccc}
    \hline\hline
    Min&$Q_{1}$&$Q_{2}$ (Median) &$Q_{3}$&Max\\\hline
    1122 & 1198 & 1211 & 1224 & 1284\\\hline
    \end{tabular} 
\end{table}
\begin{table}[h]
    \centering
    \setlength{\doublerulesep}{0.4pt}
    \caption{Sample mean and sample standard deviation (SD)  of $T_{2,n}$.} 
    \begin{tabular}{lcc}
    \hline\hline& sample mean (true value) & sample SD (theoretical value)\\\hline
     &  4.017 (4.000)& 2.792 (2.828)\\\hline
    \end{tabular} 
\end{table}
\begin{table}[h]
    \centering
    \setlength{\doublerulesep}{0.4pt}
    \caption{The number of rejections of the test with $H_0$: $k=1$, $k=2$.} 
    \begin{tabular}{ccc}
    \hline\hline
    $H_0$&$k=1$&$k=2$ \\\hline
    The number of rejections & 10000 & 490\\\hline
    \end{tabular} 
\end{table}
\begin{figure}[h]
\begin{center}
\includegraphics[width=0.3\columnwidth]{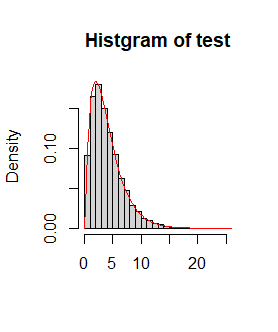}
\includegraphics[width=0.3\columnwidth]{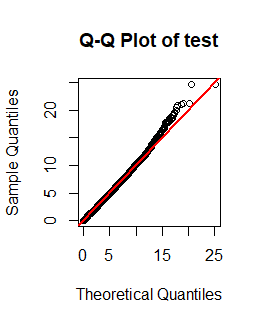}
\includegraphics[width=0.3\columnwidth]{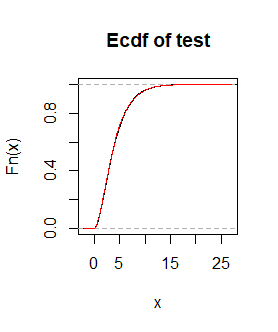}
\caption{Histogram (left), Q-Q plot (middle) and empirical distribution (right) of $T_{2,n}$.
}
\end{center}
\end{figure}


\clearpage
\section{Proofs}\label{sec5} 
For the proof, we define the following notation.
    \begin{align*}
        Q_{ff,k}&=\frac{1}{T}\sum_{i=1}^n(f_{k,t_{i}^n}-f_{k,t_{i-1}^n})(f_{k,t_{i}^n}-f_{k,t_{i-1}^n})^\top,\\
        Q_{fe,k}&=\frac{1}{T}\sum_{i=1}^n(f_{k,t_{i}^n}-f_{k,t_{i-1}^n})(e_{t_{i}^n}-e_{t_{i-1}^n})^\top,\\
        Q_{ee}&= { \frac{1}{T}\sum_{i=1}^n(e_{t_{i}^n}-e_{t_{i-1}^n})(e_{t_{i}^n}-e_{t_{i-1}^n})^\top .}
    \end{align*}
\begin{lemma}
Under assumptions [A1], [A2], [B1] and [B2],  if $h_n\rightarrow0$ and $nh_n\rightarrow\infty$, then
\begin{align}
    \label{l1-1}
    Q_{ff,k}&\stackrel{P_{\theta_{k,0}}}{\to} {\Sigma_{ff,k,0},}\\
    \label{l1-2}
    Q_{fe,k}&\stackrel{P_{\theta_{k,0}}}{\to}{ 0,}\\
    \label{l1-3}
    Q_{ee}&\stackrel{P_{\theta_{k,0}}}{\to}{ \Sigma_{ee,0}. }
\end{align}
\end{lemma}
\begin{proof}
First, we will prove (\ref{l1-1}). 
Since it follows from Lemma 7 in Kessler \cite{kessler(1997)} that
    \begin{align*}
    &\quad \frac{1}{T}\sum_{i=1}^n\E_{\theta_{k,0}}\Bigl [  (f_{k,t_i^n}^{(j_1)}-f_{k,t_{i-1}^n}^{(j_1)})(f_{k,t_i^n}^{(j_2)}-f_{k,t_{i-1}^n}^{(j_2)}) | \mathscr{F}^{n}_{i-1} \Bigl ]\\
        &=\frac{1}{T}\sum_{i=1}^n\{h_n(S_{k,0}S_{k,0}^\top)_{j_1j_2}+R(h_n^2,f_{k,t_{i-1}^n})\}\\
        &=(\Sigma_{ff,k,0})_{j_1j_2}+h_n\frac{1}{n}\sum_{i=1}^n R(1,f_{k,t_{i-1}^n})
        \stackrel{P_{\theta_{k,0}}}{\to} (\Sigma_{ff,k,0})_{j_1j_2}
        \quad (j_1,j_2=1,\cdots,k), 
    \end{align*}
we have
\begin{align*}
     \sum_{i=1}^n\E_{\theta_{k,0}}\Bigl [  \frac{1}{T}(f_{k,t_i^n}-f_{k,t_{i-1}^n})(f_{k,t_i^n}-f_{k,t_{i-1}^n})^\top | \mathscr{F}^{n}_{i-1} \Bigl ] \stackrel{P_{\theta_{k,0}}}{\to} \Sigma_{ff,k,0}.
     \end{align*}
In the same way, noting that
     \begin{align*}
    &\quad \frac{1}{T}\sum_{i=1}^n\E_{\theta_{k,0}}\Bigl [  (f_{k,t_i^n}^{(j_1)}-f_{k,t_{i-1}^n}^{(j_1)})(f_{k,t_i^n}^{(j_2)}-f_{k,t_{i-1}^n}^{(j_2)})(f_{k,t_i^n}^{(j_3)}-f_{k,t_{i-1}^n}^{(j_3)})(f_{k,t_i^n}^{(j_4)}-f_{k,t_{i-1}^n}^{(j_4)}) | \mathscr{F}^{n}_{i-1} \Bigl ]\\
    &=\frac{1}{T}\sum_{i=1}^n R(h_n^2,f_{k,t_{i-1}^n})\\
    &=h_n\frac{1}{n}\sum_{i=1}^n R(1,f_{k,t_{i-1}^n})\stackrel{P_{\theta_{k,0}}}{\to} 0
    \quad (j_1,j_2,j_3,j_4=1,\cdots,k),
    \end{align*}
we obtain
    \begin{align*}
     \sum_{i=1}^n\E_{\theta_{k,0}}\Bigl [  \frac{1}{T}(f_{k,t_i^n}-f_{k,t_{i-1}^n})(f_{k,t_i^n}-f_{k,t_{i-1}^n})^\top (f_{k,t_i^n}-f_{k,t_{i-1}^n})(f_{k,t_i^n}-f_{k,t_{i-1}^n})^\top| \mathscr{F}^{n}_{i-1} \Bigl ] \stackrel{P_{\theta_{k,0}}}{\to} 0.
     \end{align*}
Therefore, it follows from Lemma 9 in Genon-Catalot and Jacod \cite{genon(1993)} that
    \begin{align*}
        Q_{ff,k}=\frac{1}{T}\sum_{i=1}^n (f_{k,t_i^n}-f_{k,t_{i-1}^n})(f_{k,t_i^n}-f_{k,t_{i-1}^n})^\top \stackrel{P_{\theta_{k,0}}}{\to} \Sigma_{ff,k,0}.
    \end{align*}
(\ref{l1-2}) and (\ref{l1-3}) also are derived by an analogous manner. \end{proof}
\begin{lemma}
Under assumption [A1], 
\begin{align}
    \begin{split}
    \label{l2-1}
    &\quad \E_{\theta_{k,0}}\left[\Lambda_{k,0}[j_1,\cdot](f_{k,t_{i}^n}-f_{k,t_{i-1}^n})
    |\mathscr{F}^{n}_{i-1}\right]
    =R(h_n,f_{k,t_{i-1}^n})\quad (j_1=1,\cdots,p),
    \end{split}\\
    \begin{split}
    \label{l2-2}
    &\quad \E_{\theta_{k,0}}\left[\Lambda_{k,0}[j_1,\cdot](f_{k,t_{i}^n}-f_{k,t_{i-1}^n})\Lambda_{k,0}[j_2,\cdot](f_{k,t_{i}^n}-f_{k,t_{i-1}^n})
    |\mathscr{F}^{n}_{i-1}\right]\\
    &=h_n(\Lambda_{k,0}\Sigma_{ff,k,0}\Lambda_{k,0}^\top)_{j_1j_2}+R(h_n^2,f_{k,t_{i-1}^n})\quad(j_1,j_2=1,\cdots,p), 
    \end{split}\\
    \begin{split}
    \label{l2-3}
    &\quad \E_{\theta_{k,0}}\left[\Lambda_{k,0}[j_1,\cdot](f_{k,t_{i}^n}-f_{k,t_{i-1}^n})\Lambda_{k,0}[j_2,\cdot](f_{k,t_{i}^n}-f_{k,t_{i-1}^n})\Lambda_{k,0}[j_3,\cdot](f_{k,t_{i}^n}-f_{k,t_{i-1}^n})
    |\mathscr{F}^{n}_{i-1}\right]\\
    &=R(h_n^2,f_{k,t_{i-1}^n})\quad(j_1,j_2,j_3=1,\cdots,p), 
    \end{split}\\
    \begin{split}
    \label{l2-4}
    &\quad \E_{\theta_{k,0}}\left[\Lambda_{k,0}[j_1,\cdot](f_{k,t_{i}^n}-f_{k,t_{i-1}^n})\Lambda_{k,0}[j_2,\cdot](f_{k,t_{i}^n}-f_{k,t_{i-1}^n})\right.\\
    &\left.\qquad\qquad\qquad\qquad\qquad\qquad\quad\times\Lambda_{k,0}[j_3,\cdot](f_{k,t_{i}^n}-f_{k,t_{i-1}^n})\Lambda_{k,0}[j_4,\cdot](f_{k,t_{i}^n}-f_{k,t_{i-1}^n})|\mathscr{F}^{n}_{i-1}\right]\\
    &=h_n^2\{(\Lambda_{k,0}\Sigma_{ff,k,0}\Lambda_{k,0}^\top)_{j_1j_2}(\Lambda_{k,0}\Sigma_{ff,k,0}\Lambda_{k,0}^\top)_{j_3j_4}\\
    &\quad+(\Lambda_{k,0}\Sigma_{ff,k,0}\Lambda_{k,0}^\top)_{j_1j_3}(\Lambda_{k,0}\Sigma_{ff,k,0}\Lambda_{k,0}^\top)_{j_2j_4}+(\Lambda_{k,0}\Sigma_{ff,k,0}\Lambda_{k,0}^\top)_{j_1j_4}(\Lambda_{k,0}\Sigma_{ff,k,0}\Lambda_{k,0}^\top)_{j_2j_3}\}\\
    &\quad+R(h_n^3,f_{k,t_{i-1}^n})\quad (j_1,j_2,j_3,j_4=1,\cdots,p).
    \end{split}
\end{align}
\end{lemma}
\begin{proof}
First, we will prove (\ref{l2-1}). Let 
\begin{align*}
    \phi_{1,k}(x)=\Lambda_{k,0}[j_1,\cdot](x-f_{k,t_{i-1}^n})\quad (j_1=1,\cdots,p).
\end{align*}
Since
\begin{align*}
    \frac{\partial\phi_{1,k}(x)}{\partial x^{(u_1)}}
    &=\lambda_{k,0,j_1u_1}\quad(j_1=1,\cdots,p)\quad (u_1=1,\cdots,k),
    \\
    \frac{\partial^2 \phi_{1,k}(x)}{\partial x^{(u_1)}\partial x^{(u_2)}}
    &=0\quad(j_1=1,\cdots,p)\quad(u_1,u_2=1,\cdots,k),
\end{align*}
it is shown that
\begin{align*}
     \mathcal{L}(\phi_{1,k}(f_{k,t_{i-1}^n}))&=\sum_{u_1=1}^{k}b_{k,0}^{(u_1)}(f_{k,t_{i-1}^n})\left.\frac{\partial \phi_{1,k}(x)}{\partial x^{(u_1)}}\right|_{x=f_{k,t_{i-1}^n}}\\
    &\qquad\qquad+\frac{1}{2}\sum_{u_1=1}^{k}\sum_{u_2=1}^{k}(\Sigma_{ff,k,0})_{u_1u_2}\left.\frac{\partial^2 \phi_{1,k}(x)}{\partial x^{(u_1)}\partial x^{(u_2)}}\right|_{x=f_{k,t_{i-1}^n}}\\
    &=\sum_{u_1=1}^{k} b_{k,0}^{(u_1)}(f_{k,t_{i-1}^n})\lambda_{k,0,j_1u_1}\\
    &=\Lambda_{k,0}[j_1,\cdot] b_{k,0}(f_{k,t_{i-1}^n})\quad (j_1=1,\cdots,p).
\end{align*}
Therefore, 
noting that by Lemma 1 in Kessler \cite{kessler(1997)}
\begin{align*}
    \E_{\theta_{k,0}}\left[\phi_{1,k}(f_{k,t_{i}^n})
    |\mathscr{F}^{n}_{i-1}\right]
    &=\phi_{1,k}(f_{k,t_{i-1}^n})+h_n\mathcal{L}(\phi_{1,k}(f_{k,t_{i-1}^n}))
    +R(h_n^2,f_{k,t_{i-1}^n})\\
    &=h_n\Lambda_{k,0}[j_{1},\cdot] b_{k,0}(f_{k,t_{i-1}^n})+R(h_n^2,f_{k,t_{i-1}^n})\\
    &=R(h_n,f_{k,t_{i-1}^n})\quad (j_1=1,\cdots,p),
\end{align*}
we obtain $(\ref{l2-1})$. Next, we will prove (\ref{l2-2}). Let
\begin{align*}
    \phi_{2,k}(x)= \Lambda_{k,0}[j_1,\cdot](x-f_{k,t_{i-1}^n})\Lambda_{k,0}[j_2,\cdot](x-f_{k,t_{i-1}^n})\quad (j_1,j_2=1,\cdots,p).
\end{align*}
Since
\begin{align*}
    \frac{\partial \phi_{2,k}(x)}{\partial x^{(u_1)}}
    &=\lambda_{k,0,j_1u_1}\Lambda_{k,0}[j_2,\cdot](x-f_{k,t_{i-1}^n})+\lambda_{k,0,j_2u_1}\Lambda_{k,0}[j_1,\cdot](x-f_{k,t_{i-1}^n})\\
    &\qquad\qquad\qquad\qquad\qquad\qquad\qquad (j_1,j_2=1,\cdots,p)\quad (u_1=1,\cdots,k),
\end{align*}
\begin{align*}
    \frac{\partial^2 \phi_{2,k}(x)}{\partial x^{(u_1)}\partial x^{(u_2)}}
    =\lambda_{k,0,j_1u_1}\lambda_{k,0,j_2u_2}+\lambda_{k,0,j_2u_1}\lambda_{k,0,j_1u_2}\quad (j_1,j_2=1,\cdots,p)\quad  (u_1,u_2=1,\cdots,k),
\end{align*}
we have
\begin{align*}
    \mathcal{L}(\phi_{2,k}(f_{k,t_{i-1}^n}))&=\sum_{u_1=1}^{k} b_{k,0}^{(u_{1})}(f_{k,t_{i-1}^n})\left.\frac{\partial \phi_{2,k}(x)}{\partial x^{(u_1)}}\right|_{x=f_{k,t_{i-1}^n}}\\
    &\qquad\qquad
    +\frac{1}{2}\sum_{u_1=1}^{k}\sum_{u_2=1}^{k}(\Sigma_{ff,k,0})_{u_1u_2}\left.\frac{\partial^2 \phi_{2,k}(x)}{\partial x^{(u_1)}\partial x^{(u_2)}}\right|_{x=f_{k,t_{i-1}^n}}\\
    &=\frac{1}{2}\sum_{u_1=1}^{k}\sum_{u_2=1}^{k}
    (\Sigma_{ff,k,0})_{u_1u_2}(\lambda_{k,0,j_1u_1}\lambda_{k,0,j_2u_2}+\lambda_{k,0,j_2u_1}\lambda_{k,0,j_1u_2})\\
    &=\frac{1}{2}\sum_{u_1=1}^{k}\sum_{u_2=1}^{k}
    \lambda_{k,0,j_1u_1}(\Sigma_{ff,k,0})_{u_1u_2}\lambda_{k,0,j_2u_2}\\
    &\qquad\qquad+\frac{1}{2}\sum_{u_1=1}^{k}\sum_{u_2=1}^{k}\lambda_{k,0,j_2u_1}(\Sigma_{ff,k,0})_{u_1u_2}\lambda_{k,0,j_1u_2}\\
    &=\frac{1}{2}(\Lambda_{k,0} \Sigma_{ff,k,0}\Lambda_{k,0}^\top)_{j_1j_2}+\frac{1}{2}(\Lambda_{k,0} \Sigma_{ff,k,0}\Lambda_{k,0}^\top)_{j_2j_1}\\
    &=(\Lambda_{k,0}\Sigma_{ff,k,0}\Lambda_{k,0}^\top)_{j_1j_2}\quad (j_1,j_2=1,\cdots,p).
\end{align*}
By using Lemma 1 in Kessler \cite{kessler(1997)},
\begin{align*}
    \E_{\theta_{k,0}}\left[\phi_{2,k}(f_{k,t_{i}^n})
    |\mathscr{F}^{n}_{i-1}\right]
    &=\phi_{2,k}(f_{k,t_{i-1}^n})+h_n\mathcal{L}(\phi_{2,k}(f_{k,t_{i-1}^n}))
    +R(h_n^2,f_{k,t_{i-1}^n})\\
    &=h_n(\Lambda_{k,0}\Sigma_{ff,k,0}\Lambda_{k,0}^\top)_{j_1j_2}+R(h_n^2,f_{k,t_{i-1}^n})\quad (j_1,j_2=1,\cdots,p), 
\end{align*}
 and (\ref{l2-2}) are deduced. Next, we will prove (\ref{l2-3}). Let
\begin{align*}
    \phi_{3,k}(x)&=\Lambda_{k,0}[j_1,\cdot](x-f_{k,t_{i-1}^n})\Lambda_{k,0}[j_2,\cdot](x-f_{k,t_{i-1}^n})\Lambda_{k,0}[j_3,\cdot](x-f_{k,t_{i-1}^n})\quad\\ &\qquad\qquad\qquad\qquad\qquad\qquad\qquad\qquad\qquad\qquad\quad(j_1,j_2,j_3=1,\cdots,p).
\end{align*}
It is shown that
\begin{align*}
    \frac{\partial \phi_{3,k}(x)}{\partial x^{(u_1)}}
    &=\lambda_{k,0,j_1u_1}\Lambda_{k,0}[j_2,\cdot](x-f_{k,t_{i-1}^n})\Lambda_{k,0}[j_3,\cdot](x-f_{k,t_{i-1}^n})\\
    &\quad+\lambda_{k,0,j_2u_1}\Lambda_{k,0}[j_1,\cdot](x-f_{k,t_{i-1}^n})\Lambda_{k,0}[j_3,\cdot](x-f_{k,t_{i-1}^n})\\
    &\quad +\lambda_{k,0,j_3u_1}\Lambda_{k,0}[j_1,\cdot](x-f_{k,t_{i-1}^n})\Lambda_{k,0}[j_2,\cdot](x-f_{k,t_{i-1}^n})\\
    &\qquad\qquad\qquad\qquad\quad (j_1,j_2,j_3=1,\cdots,p)\quad (u_1=1,\cdots,k),
\end{align*}
\begin{align*}
    \frac{\partial^2 \phi_{3,k}(x)}{\partial x^{(u_1)}\partial x^{(u_2)}}
    &=\lambda_{k,0,j_1u_1}\lambda_{k,0,j_2u_2}\Lambda_{k,0}[j_3,\cdot](x-f_{k,t_{i-1}^n})+\lambda_{k,0,j_1u_1}\lambda_{k,0,j_3u_2}\Lambda_{k,0}[j_2,\cdot](x-f_{k,t_{i-1}^n})\\
    &\quad +\lambda_{k,0,j_2u_1}\lambda_{k,0,j_1u_2}\Lambda_{k,0}[j_3,\cdot](x-f_{k,t_{i-1}^n})+\lambda_{k,0,j_2u_1}\lambda_{k,0,j_3u_2}\Lambda_{k,0}[j_1,\cdot](x-f_{k,t_{i-1}^n})\\
    &\quad +\lambda_{k,0,j_3u_1}\lambda_{k,0,j_1u_2}\Lambda_{k,0}[j_2,\cdot](x-f_{k,t_{i-1}^n})+\lambda_{k,0,j_3u_1}\lambda_{k,0,j_2u_2}\Lambda_{k,0}[j_1,\cdot](x-f_{k,t_{i-1}^n})\\
    &\qquad\qquad\qquad\qquad\qquad\qquad\qquad\qquad\qquad (j_1,j_2,j_3=1,\cdots,p)\quad (u_1,u_2=1,\cdots,k),
\end{align*}
so that we have 
\begin{align*}
    \mathcal{L}(\phi_{3,k}(f_{k,t_{i-1}^n}))&=\sum_{u_1=1}^{k}
    b_{k,0}^{(u_1)}(f_{k,t_{i-1}^n})\left.\frac{\partial \phi_{3,k}(x)}{\partial x^{(u_1)}}\right|_{x=f_{k,t_{i-1}^n}}\\
    &\qquad\qquad+\frac{1}{2}\sum_{u_1=1}^{k}\sum_{u_2=1}^{k}(\Sigma_{ff,k,0})_{u_1u_2}\left.\frac{\partial^2 \phi_{3,k}(x)}{\partial x^{(u_1)}\partial x^{(u_2)}}\right|_{x=f_{k,t_{i-1}^n}}\\
    &=0\quad (j_1,j_2,j_3=1,\cdots,p).
\end{align*}
Since it follows from Lemma 1 in Kessler \cite{kessler(1997)} that
\begin{align*}
    \E_{\theta_{k,0}}\left[\phi_{3,k}(f_{k,t_{i}^n})
    |\mathscr{F}^{n}_{i-1}\right]
    &=\phi_{3,k}(f_{k,t_{i-1}^n})+h_n\mathcal{L}(\phi_{3,k}(f_{k,t_{i-1}^n}))
    +R(h_n^2,f_{k,t_{i-1}^n})\\
    &=R(h_n^2,f_{k,t_{i-1}^n})\quad (j_1,j_2,j_3=1,\cdots,p), 
\end{align*}
$(\ref{l2-3})$ is obtained. Next, we will prove (\ref{l2-4}). Let
\begin{align*}
    \phi_{4,k}(x)&=\Lambda_{k,0}[j_1,\cdot](x-{f_{k,t_{i-1}^n}})
    \Lambda_{k,0}[j_2,\cdot](x-{f_{k,t_{i-1}^n}})
    \Lambda_{k,0}[j_3,\cdot](x-f_{k,t_{i-1}^n})
    \Lambda_{k,0}[j_4,\cdot](x-f_{k,t_{i-1}^n})\\
    &\qquad\qquad\qquad\qquad\qquad\qquad\qquad\qquad\qquad\qquad\qquad\qquad\qquad\qquad(j_1,j_2,j_3,j_4=1,\cdots,p).
\end{align*}
Noting that
\begin{align*}
    \frac{\partial \phi_{4,k}(x)}{\partial x^{(u_1)}}
    &=\lambda_{k,0,j_1u_1}\Lambda_{k,0}[j_2,\cdot](x-{f_{k,t_{i-1}^n}})
    \Lambda_{k,0}[j_3,\cdot](x-f_{k,t_{i-1}^n})
    \Lambda_{k,0}[j_4,\cdot](x-f_{k,t_{i-1}^n})\\
    &\quad+\lambda_{k,0,j_2u_1}\Lambda_{k,0}[j_1,\cdot](x-{f_{k,t_{i-1}^n}})\Lambda_{k,0}[j_3,\cdot](x-f_{k,t_{i-1}^n})
    \Lambda_{k,0}[j_4,\cdot](x-f_{k,t_{i-1}^n})\\
    &\quad+\lambda_{k,0,j_3u_1}\Lambda_{k,0}[j_1,\cdot](x-{f_{k,t_{i-1}^n}})\Lambda_{k,0}[j_2,\cdot](x-f_{k,t_{i-1}^n})
    \Lambda_{k,0}[j_4,\cdot](x-f_{k,t_{i-1}^n})\\
    &\quad+\lambda_{k,0,j_4u_1}\Lambda_{k,0}[j_1,\cdot](x-{f_{k,t_{i-1}^n}})\Lambda_{k,0}[j_2,\cdot](x-f_{k,t_{i-1}^n})
    \Lambda_{k,0}[j_3,\cdot](x-f_{k,t_{i-1}^n})\quad\\
    &\quad \qquad\qquad\qquad\qquad\qquad\qquad\qquad\qquad(j_1,j_2,j_3,j_4=1,\cdots,p)\quad (u_1=1,\cdots,k),
    \end{align*}
    \begin{align*}
    \frac{\partial^2 \phi_{4,k}(x)}{\partial x^{(u_1)}\partial x^{(u_2)}}
    &=\lambda_{k,0,j_1u_1}\lambda_{k,0,j_2u_2}
    \Lambda_{k,0}[j_3,\cdot](x-f_{k,t_{i-1}^n})
    \Lambda_{k,0}[j_4,\cdot](x-f_{k,t_{i-1}^n})
    \\&\quad+\lambda_{k,0,j_1u_1}\lambda_{k,0,j_3u_2}
    \Lambda_{k,0}[j_2,\cdot](x-f_{k,t_{i-1}^n})
    \Lambda_{k,0}[j_4,\cdot](x-f_{k,t_{i-1}^n})\\
    &\quad+\lambda_{k,0,j_1u_1}\lambda_{k,0,j_4u_2}
    \Lambda_{k,0}[j_2,\cdot](x-f_{k,t_{i-1}^n})
    \Lambda_{k,0}[j_3,\cdot](x-f_{k,t_{i-1}^n})\\
    &\quad+\lambda_{k,0,j_2u_1}\lambda_{k,0,j_1u_2}
    \Lambda_{k,0}[j_3,\cdot](x-f_{k,t_{i-1}^n})
    \Lambda_{k,0}[j_4,\cdot](x-f_{k,t_{i-1}^n})\\
    &\quad+\lambda_{k,0,j_2u_1}\lambda_{k,0,j_3u_2}
    \Lambda_{k,0}[j_1,\cdot](x-f_{k,t_{i-1}^n})
    \Lambda_{k,0}[j_4,\cdot](x-f_{k,t_{i-1}^n})\\
    &\quad+\lambda_{k,0,j_2u_1}\lambda_{k,0,j_4u_2}
    \Lambda_{k,0}[j_1,\cdot](x-f_{k,t_{i-1}^n})
    \Lambda_{k,0}[j_3,\cdot](x-f_{k,t_{i-1}^n})\\
    &\quad+\lambda_{k,0,j_3u_1}\lambda_{k,0,j_1u_2}
    \Lambda_{k,0}[j_2,\cdot](x-f_{k,t_{i-1}^n})
    \Lambda_{k,0}[j_4,\cdot](x-f_{k,t_{i-1}^n})\\
    &\quad+\lambda_{k,0,j_3u_1}\lambda_{k,0,j_2u_2}
    \Lambda_{k,0}[j_1,\cdot](x-f_{k,t_{i-1}^n})
    \Lambda_{k,0}[j_4,\cdot](x-f_{k,t_{i-1}^n})\\
    &\quad+\lambda_{k,0,j_3u_1}\lambda_{k,0,j_4u_2}
    \Lambda_{k,0}[j_1,\cdot](x-f_{k,t_{i-1}^n})
    \Lambda_{k,0}[j_2,\cdot](x-f_{k,t_{i-1}^n})\\
    &\quad+\lambda_{k,0,j_4u_1}\lambda_{k,0,j_1u_2}
    \Lambda_{k,0}[j_2,\cdot](x-f_{k,t_{i-1}^n})
    \Lambda_{k,0}[j_3,\cdot](x-f_{k,t_{i-1}^n})\\
    &\quad+\lambda_{k,0,j_4u_1}\lambda_{k,0,j_2u_2}
    \Lambda_{k,0}[j_1,\cdot](x-f_{k,t_{i-1}^n})
    \Lambda_{k,0}[j_3,\cdot](x-f_{k,t_{i-1}^n})\\
    &\quad+\lambda_{k,0,j_4u_1}\lambda_{k,0,j_3u_2}
    \Lambda_{k,0}[j_1,\cdot](x-f_{k,t_{i-1}^n})
    \Lambda_{k,0}[j_2,\cdot](x-f_{k,t_{i-1}^n})\\
    &\qquad\qquad\qquad\qquad\qquad (j_1,j_2,j_3,j_4=1,\cdots,p)\quad (u_1,u_2=1,\cdots,k),
\end{align*}
\begin{align*}
     \frac{\partial^3 \phi_{4,k}(x)}{\partial x^{(u_1)}\partial x^{(u_2)}\partial x^{(u_3)}}
    &=\lambda_{k,0,j_1u_1}\lambda_{k,0,j_2u_2}\lambda_{k,0,j_3u_3}
    \Lambda_{k,0}[j_4,\cdot](x-f_{k,t_{i-1}^n})\\
    &\quad+\lambda_{k,0,j_1u_1}\lambda_{k,0,j_2u_2}\lambda_{k,0,j_4u_3}\Lambda_{k,0}[j_3,\cdot](x-f_{k,t_{i-1}^n})\\
    &\quad+\lambda_{k,0,j_1u_1}\lambda_{k,0,j_3u_2}
    \lambda_{k,0,j_2u_3}\Lambda_{k,0}[j_4,\cdot](x-f_{k,t_{i-1}^n})\\
    &\quad+\lambda_{k,0,j_1u_1}\lambda_{k,0,j_3u_2}\lambda_{k,0,j_4u_3}
    \Lambda_{k,0}[j_2,\cdot](x-f_{k,t_{i-1}^n})\\
    &\quad+\lambda_{k,0,j_1u_1}\lambda_{k,0,j_4u_2}
    \lambda_{k,0,j_2u_3}
    \Lambda_{k,0}
    [j_3,\cdot](x-f_{k,t_{i-1}^n})\\
    &\quad+\lambda_{k,0,j_1u_1}\lambda_{k,0,j_4u_2}\lambda_{k,0,j_3u_3}
    \Lambda_{k,0}[j_2,\cdot](x-f_{k,t_{i-1}^n})\\
    &\quad+\lambda_{k,0,j_2u_1}\lambda_{k,0,j_1u_2}
    \lambda_{k,0,j_3u_3}\Lambda_{k,0}[j_4,\cdot](x-f_{k,t_{i-1}^n})\\
    &\quad+\lambda_{k,0,j_2u_1}\lambda_{k,0,j_1u_2}\lambda_{k,0,j_4u_3}
    \Lambda_{k,0}[j_3,\cdot](x-f_{k,t_{i-1}^n})\\
    &\quad+\lambda_{k,0,j_2u_1}\lambda_{k,0,j_3u_2}
    \lambda_{k,0,j_1u_3}
    \Lambda_{k,0}[j_4,\cdot](x-f_{k,t_{i-1}^n})\\
    &\quad+\lambda_{k,0,j_2u_1}\lambda_{k,0,j_3u_2}\lambda_{k,0,j_4u_3}
    \Lambda_{k,0}[j_1,\cdot](x-f_{k,t_{i-1}^n})\\
    &\quad+\lambda_{k,0,j_2u_1}\lambda_{k,0,j_4u_2}
    \lambda_{k,0,j_1u_3}
    \Lambda_{k,0}[j_3,\cdot](x-f_{k,t_{i-1}^n})\\
    &\quad+\lambda_{k,0,j_2u_1}\lambda_{k,0,j_4u_2}
    \lambda_{k,0,j_3u_3}
    \Lambda_{k,0}[j_1,\cdot](x-f_{k,t_{i-1}^n})\\
    &\quad+\lambda_{k,0,j_3u_1}\lambda_{k,0,j_1u_2}\lambda_{k,0,j_2u_3}
    \Lambda_{k,0}[j_4,\cdot](x-f_{k,t_{i-1}^n})\\
    &\quad+\lambda_{k,0,j_3u_1}\lambda_{k,0,j_1u_2}\lambda_{k,0,j_4u_3}
    \Lambda_{k,0}[j_2,\cdot](x-f_{k,t_{i-1}^n})\\
    &\quad+\lambda_{k,0,j_3u_1}\lambda_{k,0,j_2u_2}
    \lambda_{k,0,j_1u_3}
    \Lambda_{k,0}[j_4,\cdot](x-f_{k,t_{i-1}^n})\\
    &\quad+\lambda_{k,0,j_3u_1}\lambda_{k,0,j_2u_2}
    \lambda_{k,0,j_4u_3}
    \Lambda_{k,0}[j_1,\cdot](x-f_{k,t_{i-1}^n})\\
    &\quad+\lambda_{k,0,j_3u_1}\lambda_{k,0,j_4u_2}
    \lambda_{k,0,j_1u_3}
    \Lambda_{k,0}[j_2,\cdot](x-f_{k,t_{i-1}^n})\\
    &\quad+\lambda_{k,0,j_3u_1}\lambda_{k,0,j_4u_2}
    \lambda_{k,0,j_2u_3}
    \Lambda_{k,0}[j_1,\cdot](x-f_{k,t_{i-1}^n})\\
    &\quad+\lambda_{k,0,j_4u_1}\lambda_{k,0,j_1u_2}
    \lambda_{k,0,j_2u_3}
    \Lambda_{k,0}[j_3,\cdot](x-f_{k,t_{i-1}^n})\\
    &\quad+\lambda_{k,0,j_4u_1}\lambda_{k,0,j_1u_2}
    \lambda_{k,0,j_3u_3}
    \Lambda_{k,0}[j_2,\cdot](x-f_{k,t_{i-1}^n})\\
    &\quad+\lambda_{k,0,j_4u_1}\lambda_{k,0,j_2u_2}
    \lambda_{k,0,j_1u_3}
    \Lambda_{k,0}[j_3,\cdot](x-f_{k,t_{i-1}^n})\\
    &\quad+\lambda_{k,0,j_4u_1}\lambda_{k,0,j_2u_2}
    \lambda_{k,0,j_3u_3}
    \Lambda_{k,0}[j_1,\cdot](x-f_{k,t_{i-1}^n})\\
    &\quad+\lambda_{k,0,j_4u_1}\lambda_{k,0,j_3u_2}
    \lambda_{k,0,j_1u_3}
    \Lambda_{k,0}[j_2,\cdot](x-f_{k,t_{i-1}^n})\\
    &\quad+\lambda_{k,0,j_4u_1}\lambda_{k,0,j_3u_2}
    \lambda_{k,0,j_2u_3}
    \Lambda_{k,0}[j_1,\cdot](x-f_{k,t_{i-1}^n})\\
    &\qquad(j_1,j_2,j_3,j_4=1,\cdots,p)\quad(u_1,u_2,u_3=1,\cdots,k),
\end{align*}
and 
\begin{align*}
    \frac{\partial^4 \phi_{4,k}(x)}{\partial x^{(u_1)}\partial x^{(u_2)}\partial x^{(u_3)}\partial x^{(u_4)}}
    &=\lambda_{k,0,j_1u_1}\lambda_{k,0,j_2u_2}\lambda_{k,0,j_3u_3}
    \lambda_{k,0,j_4u_4}\\
    &\quad+\lambda_{k,0,j_1u_1}\lambda_{k,0,j_2u_2}
    \lambda_{k,0,j_4u_3}\lambda_{k,0,j_3u_4}\\   
    &\quad +\lambda_{k,0,j_1u_1}\lambda_{k,0,j_3u_2}
    \lambda_{k,0,j_2u_3}\lambda_{k,0,j_4u_4}\\&\quad+\lambda_{k,0,j_1u_1}\lambda_{k,0,j_3u_2}\lambda_{k,0,j_4u_3}\lambda_{k,0,j_2u_4}\\
    &\quad+\lambda_{k,0,j_1u_1}\lambda_{k,0,j_4u_2}
    \lambda_{k,0,j_2u_3}\lambda_{k,0,j_3u_4}
    \\&\quad+\lambda_{k,0,j_1u_1}\lambda_{k,0,j_4u_2}\lambda_{k,0,j_3u_3}\lambda_{k,0,j_2u_4}\\
    &\quad+\lambda_{k,0,j_2u_1}\lambda_{k,0,j_1u_2}
    \lambda_{k,0,j_3u_3}\lambda_{k,0,j_4u_4}
    \\&\quad+\lambda_{k,0,j_2u_1}\lambda_{k,0,j_1u_2}\lambda_{k,0,j_4u_3}\lambda_{k,0,j_3u_4}\\
    &\quad+\lambda_{k,0,j_2u_1}\lambda_{k,0,j_3u_2}
    \lambda_{k,0,j_1u_3}
    \lambda_{k,0,j_4u_4}\\
   &\quad+\lambda_{k,0,j_2u_1}\lambda_{k,0,j_3u_2}\lambda_{k,0,j_4u_3}\lambda_{k,0,j_1u_4}\\
    &\quad+\lambda_{k,0,j_2u_1}\lambda_{k,0,j_4u_2}
    \lambda_{k,0,j_1u_3}
    \lambda_{k,0,j_3u_4}\\
    &\quad+\lambda_{k,0,j_2u_1}\lambda_{k,0,j_4u_2}
    \lambda_{k,0,j_3u_3}
    \lambda_{k,0,j_1u_4}\\
    &\quad+\lambda_{k,0,j_3u_1}\lambda_{k,0,j_1u_2}\lambda_{k,0,j_2u_3}\lambda_{k,0,j_4u_4}\\
    &\quad+\lambda_{k,0,j_3u_1}\lambda_{k,0,j_1u_2}\lambda_{k,0,j_4u_3}
    \lambda_{k,0,j_2u_4}\\
    &\quad+\lambda_{k,0,j_3u_1}\lambda_{k,0,j_2u_2}
    \lambda_{k,0,j_1u_3}
    \lambda_{k,0,j_4u_4}\\
    &\quad+\lambda_{k,0,j_3u_1}\lambda_{k,0,j_2u_2}
    \lambda_{k,0,j_4u_3}
    \lambda_{k,0,j_1u_4}\\
    &\quad+\lambda_{k,0,j_3u_1}\lambda_{k,0,j_4u_2}
    \lambda_{k,0,j_1u_3}
    \lambda_{k,0,j_2u_4}\\
    &\quad+\lambda_{k,0,j_3u_1}\lambda_{k,0,j_4u_2}
    \lambda_{k,0,j_2u_3}
    \lambda_{k,0,j_1u_4}\\
    &\quad+\lambda_{k,0,j_4u_1}\lambda_{k,0,j_1u_2}
    \lambda_{k,0,j_2u_3}
    \lambda_{k,0,j_3u_4}\\
    &\quad+\lambda_{k,0,j_4u_1}\lambda_{k,0,j_1u_2}
    \lambda_{k,0,j_3u_3}
    \lambda_{k,0,j_2u_4}\\
    &\quad+\lambda_{k,0,j_4u_1}\lambda_{k,0,j_2u_2}
    \lambda_{k,0,j_1u_3}
    \lambda_{k,0,j_3u_4}\\
    &\quad+\lambda_{k,0,j_4u_1}\lambda_{k,0,j_2u_2}
    \lambda_{k,0,j_3u_3}
    \lambda_{k,0,j_1u_4}\\
    &\quad+\lambda_{k,0,j_4u_1}\lambda_{k,0,j_3u_2}
    \lambda_{k,0,j_1u_3}
    \lambda_{k,0,j_2u_4}\\
    &\quad+\lambda_{k,0,j_4u_1}\lambda_{k,0,j_3u_2}
    \lambda_{k,0,j_2u_3}
    \lambda_{k,0,j_1u_4}\\
    (j_1,j_2,j_3,&j_4=1,\cdots,p)\quad (u_1,u_2,u_3,u_4=1,\cdots,k),
\end{align*}
we have
\begin{align*}
   \mathcal{L}(\phi_{4,k}(f_{k,t_{i-1}^n}))
    &=\sum_{u_1=1}^{k}b_{k,0}^{(u_1)}(f_{k,t_{i-1}^n})\left.\frac{\partial \phi_{4,k}(x)}{\partial x^{(u_1)}}\right|_{x=f_{k,t_{i-1}^n}}\\
    &\qquad\qquad +\frac{1}{2}\sum_{u_1=1}^{k}\sum_{u_2=1}^{k}{(\Sigma_{ff,k,0})}_{u_1u_2}\left.\frac{\partial^2\phi_{4,k}(x)}{\partial x^{(u_1)}\partial x^{(u_2)}}\right|_{x=f_{k,t_{i-1}^n}}\\
    &=0\quad(j_1,j_2,j_3,j_4=1,\cdots,p),
\end{align*}
and 
\begin{align*}
    &\quad\ \mathcal{L}^2(\phi_{4,k}(f_{k,t_{i-1}^n}))\\
    &=\sum_{u_1=1}^{k}b_{k,0}^{(u_1)}(f_{k,t_{i-1}^n})\left.\frac{\partial \mathcal{L}(\phi_{4,k}(x))}{\partial x^{(u_1)}}\right|_{x=f_{k,t_{i-1}^n}}+\frac{1}{2}\sum_{u_1=1}^{k}\sum_{u_2=1}^{k}{(\Sigma_{ff,k,0})}_{u_1u_2}\left.\frac{\partial^2\mathcal{L}(\phi_{4,k}(x))}{\partial x^{(u_1)}\partial x^{(u_2)}}\right|_{x=f_{k,t_{i-1}^n}}\\
    &=\sum_{u_1=1}^{k}\sum_{u_3=1}^{k}b_{k,0}^{(u_1)}(f_{k,t_{i-1}^n})b_{k,0}^{(u_3)}(f_{k,t_{i-1}^n})\left.\frac{\partial^2\phi_{4,k}(x)}{\partial x^{(u_1)}\partial x^{(u_3)}}
    \right|_{x=f_{k,t_{i-1}^n}}\\
    &\quad +\frac{1}{2}\sum_{u_1=1}^{k}\sum_{u_3=1}^{k}\sum_{u_4=1}^{k}b_{k,0}^{(u_1)}(f_{k,t_{i-1}^n})(\Sigma_{ff,k,0})_{u_3u_4}\left.\frac{\partial^3\phi_{4,k}(x)}{\partial x^{(u_1)}\partial x^{(u_3)}\partial x^{(u_4)}}
    \right|_{x=f_{k,t_{i-1}^n}}\\   
    &\quad+\frac{1}{2}\sum_{u_1=1}^{k}\sum_{u_2=1}^{k}\sum_{u_3=1}^{k}b_{k,0}^{(u_3)}(f_{k,t_{i-1}^n})(\Sigma_{ff,k,0})_{u_1u_2}\left.\frac{\partial^3\phi_{4,k}(x)}{\partial x^{(u_1)}\partial x^{(u_2)}\partial x^{(u_3)}}
    \right|_{x=f_{k,t_{i-1}^n}}\\
    &\quad+\frac{1}{4}\sum_{u_1=1}^{k}\sum_{u_2=1}^{k}\sum_{u_3=1}^{k}\sum_{u_4=1}^{k}(\Sigma_{ff,k,0})_{u_1u_2}(\Sigma_{ff,k,0})_{u_3u_4}\left.\frac{\partial^4\phi_{4,k}(x)}{\partial x^{(u_1)}\partial x^{(u_2)}\partial x^{(u_3)}\partial x^{(u_4)}}\right|_{x=f_{k,t_{i-1}^n}}\\
    &=\frac{1}{4}\sum_{u_1=1}^{k}\sum_{u_2=1}^{k}\sum_{u_3=1}^{k}\sum_{u_4=1}^{k}(\Sigma_{ff,k,0})_{u_1u_2}(\Sigma_{ff,k,0})_{u_3u_4}\\
    &\quad\times\{\lambda_{k,0,j_1u_1}\lambda_{k,0,j_2u_2}\lambda_{k,0,j_3u_3}\lambda_{k,0,j_4u_4}+\lambda_{k,0,j_1u_1}\lambda_{k,0,j_2u_2}\lambda_{k,0,j_4u_3}\lambda_{k,0,j_3u_4}\\
    &\qquad +\lambda_{k,0,j_1u_1}\lambda_{k,0,j_3u_2}\lambda_{k,0,j_2u_3}\lambda_{k,0,j_4u_4}+\lambda_{k,0,j_1u_1}\lambda_{k,0,j_3u_2}\lambda_{k,0,j_4u_3}\lambda_{k,0,j_2u_4}\\
    &\qquad+\lambda_{k,0,j_1u_1}\lambda_{k,0,j_4u_2}\lambda_{k,0,j_2u_3}\lambda_{k,0,j_3u_4}+\lambda_{k,0,j_1u_1}\lambda_{k,0,j_4u_2}\lambda_{k,0,j_3u_3}\lambda_{k,0,j_2u_4}\\
    &\qquad+\lambda_{k,0,j_2u_1}\lambda_{k,0,j_1u_2}\lambda_{k,0,j_3u_3}\lambda_{k,0,j_4u_4}+\lambda_{k,0,j_2u_1}\lambda_{k,0,j_1u_2}\lambda_{k,0,j_4u_3}\lambda_{k,0,j_3u_4}\\
    &\qquad+\lambda_{k,0,j_2u_1}\lambda_{k,0,j_3u_2}\lambda_{k,0,j_1u_3}\lambda_{k,0,j_4u_4}+\lambda_{k,0,j_2u_1}\lambda_{k,0,j_3u_2}\lambda_{k,0,j_4u_3}\lambda_{k,0,j_1u_4}\\
    &\qquad+\lambda_{k,0,j_2u_1}\lambda_{k,0,j_4u_2}\lambda_{k,0,j_1u_3}\lambda_{k,0,j_3u_4}+\lambda_{k,0,j_2u_1}\lambda_{k,0,j_4u_2}\lambda_{k,0,j_3u_3}\lambda_{k,0,j_1u_4}\\
    &\qquad+\lambda_{k,0,j_3u_1}\lambda_{k,0,j_1u_2}\lambda_{k,0,j_2u_3}
    \lambda_{k,0,j_4u_4}+\lambda_{k,0,j_3u_1}\lambda_{k,0,j_1u_2}\lambda_{k,0,j_4u_3}
    \lambda_{k,0,j_2u_4}\\
    &\qquad+\lambda_{k,0,j_3u_1}\lambda_{k,0,j_2u_2}\lambda_{k,0,j_1u_3}\lambda_{k,0,j_4u_4}+\lambda_{k,0,j_3u_1}\lambda_{k,0,j_2u_2}\lambda_{k,0,j_4u_3}\lambda_{k,0,j_1u_4}\\
    &\qquad+\lambda_{k,0,j_3u_1}\lambda_{k,0,j_4u_2}\lambda_{k,0,j_1u_3}\lambda_{k,0,j_2u_4}+\lambda_{k,0,j_3u_1}\lambda_{k,0,j_4u_2}\lambda_{k,0,j_2u_3}\lambda_{k,0,j_1u_4}\\
    &\qquad+\lambda_{k,0,j_4u_1}\lambda_{k,0,j_1u_2}\lambda_{k,0,j_2u_3}\lambda_{k,0,j_3u_4}+\lambda_{k,0,j_4u_1}\lambda_{k,0,j_1u_2}\lambda_{k,0,j_3u_3}\lambda_{k,0,j_2u_4}\\
    &\qquad+\lambda_{k,0,j_4u_1}\lambda_{k,0,j_2u_2}\lambda_{k,0,j_1u_3}\lambda_{k,0,j_3u_4}+\lambda_{k,0,j_4u_1}\lambda_{k,0,j_2u_2}\lambda_{k,0,j_3u_3}\lambda_{k,0,j_1u_4}\\
    &\qquad+\lambda_{k,0,j_4u_1}\lambda_{k,0,j_3u_2}\lambda_{k,0,j_1u_3}\lambda_{k,0,j_2u_4}+\lambda_{k,0,j_4u_1}\lambda_{k,0,j_3u_2}\lambda_{k,0,j_2u_3}\lambda_{k,0,j_1u_4}\}\\
    &=\frac{1}{4}\{(\Lambda_{k,0}\Sigma_{ff,k,0}\Lambda_{k,0}^\top)_{j_1j_2}(\Lambda_{k,0}\Sigma_{ff,k,0}\Lambda_{k,0}^\top)_{j_3j_4}
    +(\Lambda_{k,0}\Sigma_{ff,k,0}\Lambda_{k,0}^\top)_{j_1j_2}(\Lambda_{k,0}\Sigma_{ff,k,0}\Lambda_{k,0}^\top)_{j_4j_3}\\
    &\quad+(\Lambda_{k,0}\Sigma_{ff,k,0}\Lambda_{k,0}^\top)_{j_1j_3}(\Lambda_{k,0}\Sigma_{ff,k,0}\Lambda_{k,0}^\top)_{j_2j_4}
    +(\Lambda_{k,0}\Sigma_{ff,k,0}\Lambda_{k,0}^\top)_{j_1j_3}(\Lambda_{k,0}\Sigma_{ff,k,0}\Lambda_{k,0}^\top)_{j_4j_2}\\ &\quad+(\Lambda_{k,0}\Sigma_{ff,k,0}\Lambda_{k,0}^\top)_{j_1j_4}(\Lambda_{k,0}\Sigma_{ff,k,0}\Lambda_{k,0}^\top)_{j_2j_3} +(\Lambda_{k,0}\Sigma_{ff,k,0}\Lambda_{k,0}^\top)_{j_1j_4}(\Lambda_{k,0}\Sigma_{ff,k,0}\Lambda_{k,0}^\top)_{j_3j_2}\\
    &\quad+(\Lambda_{k,0}\Sigma_{ff,k,0}\Lambda_{k,0}^\top)_{j_2j_1}(\Lambda_{k,0}\Sigma_{ff,k,0}\Lambda_{k,0}^\top)_{j_3j_4}+(\Lambda_{k,0}\Sigma_{ff,k,0}\Lambda_{k,0}^\top)_{j_2j_1}(\Lambda_{k,0}\Sigma_{ff,k,0}\Lambda_{k,0}^\top)_{j_4j_3}\\
    &\quad+(\Lambda_{k,0}\Sigma_{ff,k,0}\Lambda_{k,0}^\top)_{j_2j_3}(\Lambda_{k,0}\Sigma_{ff,k,0}\Lambda_{k,0}^\top)_{j_1j_4}+(\Lambda_{k,0}\Sigma_{ff,k,0}\Lambda_{k,0}^\top)_{j_2j_3}(\Lambda_{k,0}\Sigma_{ff,k,0}\Lambda_{k,0}^\top)_{j_4j_1} \\
    &\quad+(\Lambda_{k,0}\Sigma_{ff,k,0}\Lambda_{k,0}^\top)_{j_2j_4}(\Lambda_{k,0}\Sigma_{ff,k,0}\Lambda_{k,0}^\top)_{j_1j_3} +(\Lambda_{k,0}\Sigma_{ff,k,0}\Lambda_{k,0}^\top)_{j_2j_4}(\Lambda_{k,0}\Sigma_{ff,k,0}\Lambda_{k,0}^\top)_{j_3j_1} \\
    &\quad+(\Lambda_{k,0}\Sigma_{ff,k,0}\Lambda_{k,0}^\top)_{j_3j_1}(\Lambda_{k,0}\Sigma_{ff,k,0}\Lambda_{k,0}^\top)_{j_2j_4}   +(\Lambda_{k,0}\Sigma_{ff,k,0}\Lambda_{k,0}^\top)_{j_3j_1}(\Lambda_{k,0}\Sigma_{ff,k,0}\Lambda_{k,0}^\top)_{j_4j_2}\\
    &\quad+(\Lambda_{k,0}\Sigma_{ff,k,0}\Lambda_{k,0}^\top)_{j_3j_2}(\Lambda_{k,0}\Sigma_{ff,k,0}\Lambda_{k,0}^\top)_{j_1j_4}+(\Lambda_{k,0}\Sigma_{ff,k,0}\Lambda_{k,0}^\top)_{j_3j_2}(\Lambda_{k,0}\Sigma_{ff,k,0}\Lambda_{k,0}^\top)_{j_4j_1}\\ 
    &\quad+(\Lambda_{k,0}\Sigma_{ff,k,0}\Lambda_{k,0}^\top)_{j_3j_4}(\Lambda_{k,0}\Sigma_{ff,k,0}\Lambda_{k,0}^\top)_{j_1j_2} 
    +(\Lambda_{k,0}\Sigma_{ff,k,0}\Lambda_{k,0}^\top)_{j_3j_4}(\Lambda_{k,0}\Sigma_{ff,k,0}\Lambda_{k,0}^\top)_{j_2j_1} \\
    &\quad+(\Lambda_{k,0}\Sigma_{ff,k,0}\Lambda_{k,0}^\top)_{j_4j_1}(\Lambda_{k,0}\Sigma_{ff,k,0}\Lambda_{k,0}^\top)_{j_2j_3} +(\Lambda_{k,0}\Sigma_{ff,k,0}\Lambda_{k,0}^\top)_{j_4j_1}(\Lambda_{k,0}\Sigma_{ff,k,0}\Lambda_{k,0}^\top)_{j_3j_2} \\
    &\quad+(\Lambda_{k,0}\Sigma_{ff,k,0}\Lambda_{k,0}^\top)_{j_4j_2}(\Lambda_{k,0}\Sigma_{ff,k,0}\Lambda_{k,0}^\top)_{j_1j_3}+(\Lambda_{k,0}\Sigma_{ff,k,0}\Lambda_{k,0}^\top)_{j_4j_2}(\Lambda_{k,0}\Sigma_{ff,k,0}\Lambda_{k,0}^\top)_{j_3j_1} \\
    &\quad+(\Lambda_{k,0}\Sigma_{ff,k,0}\Lambda_{k,0}^\top)_{j_4j_3}(\Lambda_{k,0}\Sigma_{ff,k,0}\Lambda_{k,0}^\top)_{j_1j_2}+(\Lambda_{k,0}\Sigma_{ff,k,0}\Lambda_{k,0}^\top)_{j_4j_3}(\Lambda_{k,0}\Sigma_{ff,k,0}\Lambda_{k,0}^\top)_{j_2j_1}\}\\
    &=2\{(\Lambda_{k,0}\Sigma_{ff,k,0}\Lambda_{k,0}^\top)_{j_1j_2}(\Lambda_{k,0}\Sigma_{ff,k,0}\Lambda_{k,0}^\top)_{j_3j_4}+(\Lambda_{k,0}\Sigma_{ff,k,0}\Lambda_{k,0}^\top)_{j_1j_3}(\Lambda_{k,0}\Sigma_{ff,k,0}\Lambda_{k,0}^\top)_{j_2j_4}\\
    &\quad +(\Lambda_{k,0}\Sigma_{ff,k,0}\Lambda_{k,0}^\top)_{j_1j_4}(\Lambda_{k,0}\Sigma_{ff,k,0}\Lambda_{k,0}^\top)_{j_2j_3}\}\quad  (j_1,j_2,j_3,j_4=1,\cdots,p).
\end{align*}
Therefore, Lemma 1 in Kessler \cite{kessler(1997)} yields that
\begin{align*}
    \E_{\theta_{k,0}}\left[\phi_{4,k}(f_{k,t_{i}^n})| \mathscr{F}^{n}_{i-1}\right]
    &=\phi_{4,k}(f_{k,t_{i-1}^n})+h_n\mathcal{L}(\phi_{4,k}(f_{k,t_{i-1}^n}))+\frac{h_n^2}{2}\mathcal{L}^2(\phi_{4,k}(f_{k,t_{i-1}^n}))+R(h_n^3,f_{k,t_{i-1}^n})\\
    &=h_n^2\{(\Lambda_{k,0}\Sigma_{ff,k,0}\Lambda_{k,0}^\top)_{j_1j_2}(\Lambda_{k,0}\Sigma_{ff,k,0}\Lambda_{k,0}^\top)_{j_3j_4}\\
    &\quad+(\Lambda_{k,0}\Sigma_{ff,k,0}\Lambda_{k,0}^\top)_{j_1j_3}(\Lambda_{k,0}\Sigma_{ff,k,0}\Lambda_{k,0}^\top)_{j_2j_4}\\
    &\quad+(\Lambda_{k,0}\Sigma_{ff,k,0}\Lambda_{k,0}^\top)_{j_1j_4}(\Lambda_{k,0}\Sigma_{ff,k,0}\Lambda_{k,0}^\top)_{j_2j_3}\}
    +R(h_n^3,f_{k,t_{i-1}^n})\\ &\qquad\qquad\qquad\qquad\qquad\qquad\qquad\qquad\qquad\qquad(j_1,j_2,j_3,j_4=1,\cdots,p),
\end{align*}
so that we obtain $(\ref{l2-4})$. 
\end{proof}
\begin{lemma}
Under assumption [B1], 
\begin{align}
    \begin{split}
    \label{l3-1}
    &\quad \E_{\theta_{k,0}}\left[e_{t_{i}^n}^{(j_1)}-e_{t_{i-1}^n}^{(j_1)}
    |\mathscr{F}^{n}_{i-1}\right]
    =R(h_n,e_{t_{i-1}^n})\quad (j_1=1,\cdots,p), 
    \end{split}\\
    \begin{split}
    \label{l3-2}
    &\quad \E_{\theta_{k,0}}\left[(e_{t_{i}^n}^{(j_1)}-e_{t_{i-1}^n}^{(j_1)})(e_{t_{i}^n}^{(j_2)}-e_{t_{i-1}^n}^{(j_2)})
    |\mathscr{F}^{n}_{i-1}\right]
    =h_n(\Sigma_{ee,0})_{j_1j_2}+R(h_n^2,e_{t_{i-1}^n})\quad(j_1,j_2=1,\cdots,p),
    \end{split}\\
    \begin{split}
    \label{l3-3}
    &\quad \E_{\theta_{k,0}}\left[(e_{t_{i}^n}^{(j_1)}-e_{t_{i-1}^n}^{(j_1)})(e_{t_{i}^n}^{(j_2)}-e_{t_{i-1}^n}^{(j_2)})(e_{t_{i}^n}^{(j_3)}-e_{t_{i-1}^n}^{(j_3)})
    |\mathscr{F}^{n}_{i-1}\right]
    =R(h_n^2,e_{t_{i-1}^n})\quad(j_1,j_2,j_3=1,\cdots,p),
    \end{split}\\
    \begin{split}
    \label{l3-4}
    &\quad \E_{\theta_{k,0}}\left[(e_{t_{i}^n}^{(j_1)}-e_{t_{i-1}^n}^{(j_1)})(e_{t_{i}^n}^{(j_2)}-e_{t_{i-1}^n}^{(j_2)})(e_{t_{i}^n}^{(j_3)}-e_{t_{i-1}^n}^{(j_3)})(e_{t_{i}^n}^{(j_4)}-e_{t_{i-1}^n}^{(j_4)})
    |\mathscr{F}^{n}_{i-1}\right]\\
    &=h_n^2\left\{(\Sigma_{ee,0})_{j_1j_2}(\Sigma_{ee,0})_{j_3j_4}
    +(\Sigma_{ee,0})_{j_1j_3}(\Sigma_{ee,0})_{j_2j_4}+(\Sigma_{ee,0})_{j_1j_4}(\Sigma_{ee,0})_{j_2j_3}\right\}+R(h_n^3,e_{t_{i-1}^n})\\
    &\qquad\qquad\qquad\qquad\qquad\qquad\qquad\qquad\qquad\qquad\qquad\qquad\qquad\qquad\qquad(j_1,j_2,j_3,j_4=1,\cdots,p).
    \end{split}
\end{align}
\end{lemma}
\begin{proof}
See  Lemma 7 in Kessler \cite{kessler(1997)}.
\end{proof}
\begin{lemma}
Under assumptions [A1] and [B1], 
\begin{align}
        \begin{split}
        \label{l4-1}
        &\quad \E_{\theta_{k,0}}\left[(X_{t_{i}^n}^{(j_1)}-X_{t_{i-1}^n}^{(j_1)})(X_{t_{i}^n}^{(j_2)}-X_{t_{i-1}^n}^{(j_2)})|\mathscr{F}^{n}_{i-1}\right]\\
        &=h_n(\Sigma_{k}(\theta_{k,0}))_{j_1j_2}
        +h_n^2\{R(1,f_{k,t_{i-1}^n})+R(1,f_{k,t_{i-1}^n})R(1,e_{t_{i-1}^n})+R(1,e_{t_{i-1}^n})\}\quad(j_1,j_2=1,\cdots,p), 
        \end{split}\\
        \begin{split}
        \label{l4-2}
        &\quad \E_{\theta_{k,0}}\left[ (X^{(j_1)}_{t_{i}^n}-X^{(j_1)}_{t_{i-1}^n})(X^{(j_2)}_{t_{i}^n}-X^{(j_2)}_{t_{i-1}^n})(X^{(j_3)}_{t_{i}^n}-X^{(j_3)}_{t_{i-1}^n})(X^{(j_4)}_{t_{i}^n}-X^{(j_4)}_{t_{i-1}^n})| \mathscr{F}^{n}_{i-1}\right]\\
        &=h_n^2\{(\Sigma_{k}(\theta_{k,0}))_{j_1j_2}(\Sigma_{k}(\theta_{k,0}))_{j_3j_4}+(\Sigma_{k}(\theta_{k,0}))_{j_1j_3}(\Sigma_{k}(\theta_{k,0}))_{j_2j_4}+(\Sigma_{k}(\theta_{k,0}))_{j_1j_4}(\Sigma_{k}(\theta_{k,0}))_{j_2j_3}\}\\
        &\quad +h_n^3\{R(1,f_{k,t_{i-1}^n})  +R(1,f_{k,t_{i-1}^n})R(1,e_{t_{i-1}^n})+R(1,e_{t_{i-1}^n})\}\quad (j_1,j_2,j_3,j_4=1,\cdots,p). 
        \end{split}
\end{align}
\end{lemma}
\begin{proof}
First, we will prove (\ref{l4-1}). From Lemma 2 (\ref{l2-1}), (\ref{l2-2}) and Lemma 3 (\ref{l3-1}), (\ref{l3-2}), 
it follows that
\begin{align*}
 &\quad \E_{\theta_{k,0}}\left[(X_{t_{i}^n}^{(j_1)}-X_{t_{i-1}^n}^{(j_1)})(X_{t_{i}^n}^{(j_2)}-X_{t_{i-1}^n}^{(j_2)})|\mathscr{F}^{n}_{i-1}\right]\\
     &=\E_{\theta_{k,0}}\left[(\Lambda_{k,0}[j_1,\cdot](f_{k,t_{i}^n}-f_{k,t_{i-1}^n})+e_{t_{i}^n}^{(j_1)}-e_{t_{i-1}^n}^{(j_1)})(\Lambda_{k,0}[j_2,\cdot](f_{k,t_{i}^n}-f_{k,t_{i-1}^n})+e_{t_{i}^n}^{(j_2)}-e_{t_{i-1}^n}^{(j_2)})
     |\mathscr{F}^{n}_{i-1}\right]\\
    &=\E_{\theta_{k,0}}\left[\Lambda_{k,0}[j_1,\cdot](f_{k,t_{i}^n}-f_{k,t_{i-1}^n})\Lambda_{k,0}[j_2,\cdot](f_{k,t_{i}^n}-f_{k,t_{i-1}^n})|\mathscr{F}^{n}_{i-1}\right]
    \\
    &\quad
    +\E_{\theta_{k,0}}\left[\Lambda_{k,0}[j_1,\cdot](f_{k,t_{i}^n}-f_{k,t_{i-1}^n})(e_{t_{i}^n}^{(j_2)}-e_{t_{i-1}^n}^{(j_2)})|\mathscr{F}^{n}_{i-1}\right]\\
    &\quad 
    +\E_{\theta_{k,0}}\left[(e_{t_{i}^n}^{(j_1)}-e_{t_{i-1}^n}^{(j_1)})\Lambda_{k,0}[j_2,\cdot](f_{k,t_{i}^n}-f_{k,t_{i-1}^n})|\mathscr{F}^{n}_{i-1}\right]\\
    &\quad+\E_{\theta_{k,0}}\left[(e_{t_{i}^n}^{(j_1)}-e_{t_{i-1}^n}^{(j_1)})(e_{t_{i}^n}^{(j_2)}-e_{t_{i-1}^n}^{(j_2)})|\mathscr{F}^{n}_{i-1}\right]\\
    &=\E_{\theta_{k,0}}\left[\Lambda_{k,0}[j_1,\cdot](f_{k,t_{i}^n}-f_{k,t_{i-1}^n})\Lambda_{k,0}[j_2,\cdot](f_{k,t_{i}^n}-f_{k,t_{i-1}^n})|\mathscr{F}^{n}_{i-1}\right]
    \\
    &\quad+\E_{\theta_{k,0}}\left[\Lambda_{k,0}[j_1,\cdot](f_{k,t_{i}^n}-f_{k,t_{i-1}^n})|\mathscr{F}^{n}_{i-1}\right]\E_{\theta_{k,0}}\left[e_{t_{i}^n}^{(j_2)}-e_{t_{i-1}^n}^{(j_2)}|\mathscr{F}^{n}_{i-1}\right]\\
    &\quad
    +\E_{\theta_{k,0}}\left[e_{t_{i}^n}^{(j_1)}-e_{t_{i-1}^n}^{(j_1)}|\mathscr{F}^{n}_{i-1}\right]\E_{\theta_{k,0}}\left[\Lambda_{k,0}[j_2,\cdot](f_{k,t_{i}^n}-f_{k,t_{i-1}^n})|\mathscr{F}^{n}_{i-1}\right]\\
    &\quad+\E_{\theta_{k,0}}\left[(e_{t_{i}^n}^{(j_1)}-e_{t_{i-1}^n}^{(j_1)})(e_{t_{i}^n}^{(j_2)}-e_{t_{i-1}^n}^{(j_2)})|\mathscr{F}^{n}_{i-1}\right]\\
    &=h_n(\Lambda_{k,0} \Sigma_{ff,k,0}\Lambda_{k,0}^\top)_{j_1j_2}+R(h_n^2,f_{k,t_{i-1}^n})+R(h_n,f_{k,t_{i-1}^n})R(h_n,e_{t_{i-1}^n})\\
    &\quad +R(h_n,e_{t_{i-1}^n})R(h_n,f_{k,t_{i-1}^n})
    +h_n{(\Sigma_{ee,0})}_{j_1j_2}+R(h_n^2,e_{t_{i-1}^n})\\
    &=h_n(\Sigma_{k}(\theta_{k,0}))_{j_1j_2}+h_n^2\{R(1,f_{k,t_{i-1}^n})+R(1,f_{k,t_{i-1}^n})R(1,e_{t_{i-1}^n})+R(1,e_{t_{i-1}^n})\}\quad(j_1,j_2=1,\cdots,p),
\end{align*}
so that we have (\ref{l4-1}). Next, we will prove (\ref{l4-2}). 
It is shown that
\begin{align}
    &\quad \E_{\theta_{k,0}}\left[ (X^{(j_1)}_{t_{i}^n}-X^{(j_1)}_{t_{i-1}^n})(X^{(j_2)}_{t_{i}^n}-X^{(j_2)}_{t_{i-1}^n})(X^{(j_3)}_{t_{i}^n}-X^{(j_3)}_{t_{i-1}^n})(X^{(j_4)}_{t_{i}^n}-X^{(j_4)}_{t_{i-1}^n})| \mathscr{F}^{n}_{i-1}\right]\nonumber \\
    &=\E_{\theta_{k,0}}\left[
    (\Lambda_{k,0}[j_1,\cdot](f_{k,t_{i}^n}-f_{k,t_{i-1}^n})+e^{(j_1)}_{t_{i}^n}-e^{(j_1)}_{t_{i-1}^n})(\Lambda_{k,0}[j_2,\cdot](f_{k,t_{i}^n}-f_{k,t_{i-1}^n})+e^{(j_2)}_{t_{i}^n}-e^{(j_2)}_{t_{i-1}^n})
    \nonumber \right.\\
    &\qquad \left.\times (\Lambda_{k,0}[j_3,\cdot](f_{k,t_{i}^n}-f_{k,t_{i-1}^n})+e^{(j_3)}_{t_{i}^n}-e^{(j_3)}_{t_{i-1}^n})(\Lambda_{k,0}[j_4,\cdot](f_{k,t_{i}^n}-f_{k,t_{i-1}^n})+e^{(j_4)}_{t_{i}^n}-e^{(j_4)}_{t_{i-1}^n})|\mathscr{F}^{n}_{i-1}\right]\nonumber \\
    &=\E_{\theta_{k,0}}\left[
    \Lambda_{k,0}[j_1,\cdot](f_{k,t_{i}^n}-f_{k,t_{i-1}^n})
    \Lambda_{k,0}[j_2,\cdot](f_{k,t_{i}^n}-f_{k,t_{i-1}^n})\right.\nonumber\\
    &\qquad\left.\qquad\qquad\qquad\qquad\qquad\qquad\qquad\times 
    \Lambda_{k,0}[j_3,\cdot](f_{k,t_{i}^n}-f_{k,t_{i-1}^n})
    \Lambda_{k,0}[j_4,\cdot](f_{k,t_{i}^n}-f_{k,t_{i-1}^n})
    |\mathscr{F}^{n}_{i-1}\right]\nonumber\\
    &\quad+\E_{\theta_{k,0}}\left[
    (e^{(j_1)}_{t_{i}^n}-e^{(j_1)}_{t_{i-1}^n})
    \Lambda_{k,0}[j_2,\cdot](f_{k,t_{i}^n}-f_{k,t_{i-1}^n})\right.\nonumber\\
    &\qquad \left.\qquad\qquad\qquad\qquad\qquad\qquad\qquad\times
    \Lambda_{k,0}[j_3,\cdot](f_{k,t_{i}^n}-f_{k,t_{i-1}^n})
    \Lambda_{k,0}[j_4,\cdot](f_{k,t_{i}^n}-f_{k,t_{i-1}^n})
    |\mathscr{F}^{n}_{i-1}\right]\nonumber\\
    &\quad+\E_{\theta_{k,0}}\left[
    \Lambda_{k,0}[j_1,\cdot](f_{k,t_{i}^n}-f_{k,t_{i-1}^n})
    (e^{(j_2)}_{t_{i}^n}-e^{(j_2)}_{t_{i-1}^n})\right.\nonumber\\
    &\qquad \left. \qquad\qquad\qquad\qquad\qquad\qquad\qquad\times
    \Lambda_{k,0}[j_3,\cdot](f_{k,t_{i}^n}-f_{k,t_{i-1}^n})
    \Lambda_{k,0}[j_4,\cdot](f_{k,t_{i}^n}-f_{k,t_{i-1}^n})
    |\mathscr{F}^{n}_{i-1}\right]\nonumber\\
    &\quad+\E_{\theta_{k,0}}\left[
    \Lambda_{k,0}[j_1,\cdot](f_{k,t_{i}^n}-f_{k,t_{i-1}^n})
    \Lambda_{k,0}[j_2,\cdot](f_{k,t_{i}^n}-f_{k,t_{i-1}^n})\right.\nonumber\\
    &\qquad \left. \qquad\qquad\qquad\qquad\qquad\qquad\qquad\qquad\qquad\times
    (e^{(j_3)}_{t_{i}^n}-e^{(j_3)}_{t_{i-1}^n})
    \Lambda_{k,0}[j_4,\cdot](f_{k,t_{i}^n}-f_{k,t_{i-1}^n})
    |\mathscr{F}^{n}_{i-1}\right]\nonumber\\
    &\quad+\E_{\theta_{k,0}}\left[
    \Lambda_{k,0}[j_1,\cdot](f_{k,t_{i}^n}-f_{k,t_{i-1}^n})
    \Lambda_{k,0}[j_2,\cdot](f_{k,t_{i}^n}-f_{k,t_{i-1}^n})\right.\nonumber\\
    &\qquad \left. \qquad\qquad\qquad\qquad\qquad\qquad\qquad\qquad\qquad\times
    \Lambda_{k,0}[j_3,\cdot](f_{k,t_{i}^n}-f_{k,t_{i-1}^n})
    (e^{(j_4)}_{t_{i}^n}-e^{(j_4)}_{t_{i-1}^n})
    |\mathscr{F}^{n}_{i-1}\right]\nonumber\\
    &\quad+\E_{\theta_{k,0}}\left[
    (e^{(j_1)}_{t_{i}^n}-e^{(j_1)}_{t_{i-1}^n})
    (e^{(j_2)}_{t_{i}^n}-e^{(j_2)}_{t_{i-1}^n})
    \Lambda_{k,0}[j_3,\cdot](f_{k,t_{i}^n}-f_{k,t_{i-1}^n})
    \Lambda_{k,0}[j_4,\cdot](f_{k,t_{i}^n}-f_{k,t_{i-1}^n})
    |\mathscr{F}^{n}_{i-1}\right]\nonumber\\
    &\quad+\E_{\theta_{k,0}}\left[
    (e^{(j_1)}_{t_{i}^n}-e^{(j_1)}_{t_{i-1}^n})
    \Lambda_{k,0}[j_2,\cdot](f_{k,t_{i}^n}-f_{k,t_{i-1}^n})
    (e^{(j_3)}_{t_{i}^n}-e^{(j_3)}_{t_{i-1}^n})
    \Lambda_{k,0}[j_4,\cdot](f_{k,t_{i}^n}-f_{k,t_{i-1}^n})
    |\mathscr{F}^{n}_{i-1}\right]\nonumber\\
    &\quad+\E_{\theta_{k,0}}\left[
    (e^{(j_1)}_{t_{i}^n}-e^{(j_1)}_{t_{i-1}^n})
    \Lambda_{k,0}[j_2,\cdot](f_{k,t_{i}^n}-f_{k,t_{i-1}^n})
    \Lambda_{k,0}[j_3,\cdot](f_{k,t_{i}^n}-f_{k,t_{i-1}^n})
    (e^{(j_4)}_{t_{i}^n}-e^{(j_4)}_{t_{i-1}^n})
    |\mathscr{F}^{n}_{i-1}\right]\nonumber\\
    &\quad+\E_{\theta_{k,0}}\left[
    \Lambda_{k,0}[j_1,\cdot](f_{k,t_{i}^n}-f_{k,t_{i-1}^n})
    (e^{(j_2)}_{t_{i}^n}-e^{(j_2)}_{t_{i-1}^n})
    (e^{(j_3)}_{t_{i}^n}-e^{(j_3)}_{t_{i-1}^n})
    \Lambda_{k,0}[j_4,\cdot](f_{k,t_{i}^n}-f_{k,t_{i-1}^n})
    |\mathscr{F}^{n}_{i-1}\right]\nonumber\\
    &\quad+\E_{\theta_{k,0}}\left[
    \Lambda_{k,0}[j_1,\cdot](f_{k,t_{i}^n}-f_{k,t_{i-1}^n})
    (e^{(j_2)}_{t_{i}^n}-e^{(j_2)}_{t_{i-1}^n})
    \Lambda_{k,0}[j_3,\cdot](f_{k,t_{i}^n}-f_{k,t_{i-1}^n})
    (e^{(j_4)}_{t_{i}^n}-e^{(j_4)}_{t_{i-1}^n})
    |\mathscr{F}^{n}_{i-1}\right]\nonumber\\
    &\quad+\E_{\theta_{k,0}}\left[
    \Lambda_{k,0}[j_1,\cdot](f_{k,t_{i}^n}-f_{k,t_{i-1}^n})
    \Lambda_{k,0}[j_2,\cdot](f_{k,t_{i}^n}-f_{k,t_{i-1}^n})
    (e^{(j_3)}_{t_{i}^n}-e^{(j_3)}_{t_{i-1}^n})
    (e^{(j_4)}_{t_{i}^n}-e^{(j_4)}_{t_{i-1}^n})
    |\mathscr{F}^{n}_{i-1}\right]\nonumber\\
    &\quad+\E_{\theta_{k,0}}\left[
    \Lambda_{k,0}[j_1,\cdot](f_{k,t_{i}^n}-f_{k,t_{i-1}^n})
    (e^{(j_2)}_{t_{i}^n}-e^{(j_2)}_{t_{i-1}^n})
    (e^{(j_3)}_{t_{i}^n}-e^{(j_3)}_{t_{i-1}^n})
    (e^{(j_4)}_{t_{i}^n}-e^{(j_4)}_{t_{i-1}^n})
    |\mathscr{F}^{n}_{i-1}\right]\nonumber\\
    &\quad+\E_{\theta_{k,0}}\left[
    (e^{(j_1)}_{t_{i}^n}-e^{(j_1)}_{t_{i-1}^n})
    \Lambda_{k,0}[j_2,\cdot](f_{k,t_{i}^n}-f_{k,t_{i-1}^n})
    (e^{(j_3)}_{t_{i}^n}-e^{(j_3)}_{t_{i-1}^n})
    (e^{(j_4)}_{t_{i}^n}-e^{(j_4)}_{t_{i-1}^n})
    |\mathscr{F}^{n}_{i-1}\right]\nonumber\\
    &\quad+\E_{\theta_{k,0}}\left[
    (e^{(j_1)}_{t_{i}^n}-e^{(j_1)}_{t_{i-1}^n})
    (e^{(j_2)}_{t_{i}^n}-e^{(j_2)}_{t_{i-1}^n})
    \Lambda_{k,0}[j_3,\cdot](f_{k,t_{i}^n}-f_{k,t_{i-1}^n})
    (e^{(j_4)}_{t_{i}^n}-e^{(j_4)}_{t_{i-1}^n})
    |\mathscr{F}^{n}_{i-1}\right]\nonumber\\
    &\quad+\E_{\theta_{k,0}}\left[
    (e^{(j_1)}_{t_{i}^n}-e^{(j_1)}_{t_{i-1}^n})
    (e^{(j_2)}_{t_{i}^n}-e^{(j_2)}_{t_{i-1}^n})
    (e^{(j_3)}_{t_{i}^n}-e^{(j_3)}_{t_{i-1}^n})
    \Lambda_{k,0}[j_4,\cdot](f_{k,t_{i}^n}-f_{k,t_{i-1}^n})
    |\mathscr{F}^{n}_{i-1}\right]\nonumber\\
    &\quad+\E_{\theta_{k,0}}\left[
    (e^{(j_1)}_{t_{i}^n}-e^{(j_1)}_{t_{i-1}^n})
    (e^{(j_2)}_{t_{i}^n}-e^{(j_2)}_{t_{i-1}^n})
    (e^{(j_3)}_{t_{i}^n}-e^{(j_3)}_{t_{i-1}^n})
    (e^{(j_4)}_{t_{i}^n}-e^{(j_4)}_{t_{i-1}^n})
    |\mathscr{F}^{n}_{i-1}\right]\nonumber\\ 
    &\quad \qquad\qquad\qquad\qquad\qquad\qquad\qquad\qquad\qquad\qquad\qquad\qquad\qquad(j_1,j_2,j_3,j_4=1,\cdots,p)\label{l4-3}.
\end{align}
From Lemma 2 (\ref{l2-4}), the first term of (\ref{l4-3}) is 
\begin{align*}
    &\quad \E_{\theta_{k,0}}\left[
    \Lambda_{k,0}[j_1,\cdot](f_{k,t_{i}^n}-f_{k,t_{i-1}^n})
    \Lambda_{k,0}[j_2,\cdot](f_{k,t_{i}^n}-f_{k,t_{i-1}^n})\right.\\
    &\qquad\qquad\qquad\qquad\qquad\qquad\qquad\qquad\left.\times
    \Lambda_{k,0}[j_3,\cdot](f_{k,t_{i}^n}-f_{k,t_{i-1}^n})
    \Lambda_{k,0}[j_4,\cdot](f_{k,t_{i}^n}-f_{k,t_{i-1}^n})
    |\mathscr{F}^{n}_{i-1}\right]\\
    &=h_n^2\{(\Lambda_{k,0}\Sigma_{ff,k,0}\Lambda_{k,0}^\top)_{j_1j_2}(\Lambda_{k,0}\Sigma_{ff,k,0}\Lambda_{k,0}^\top)_{j_3j_4}
    +(\Lambda_{k,0}\Sigma_{ff,k,0}\Lambda_{k,0}^\top)_{j_1j_3}(\Lambda_{k,0}\Sigma_{ff,k,0}\Lambda_{k,0}^\top)_{j_2j_4}\\
    &\quad +(\Lambda_{k,0}\Sigma_{ff,k,0}\Lambda_{k,0}^\top)_{j_1j_4}(\Lambda_{k,0}\Sigma_{ff,k,0}\Lambda_{k,0}^\top)_{j_2j_3}\}
    +h_n^3R(1,f_{k,t_{i-1}^n})\quad (j_1,j_2,j_3,j_4=1,\cdots,p).
\end{align*}
From Lemma 2 (\ref{l2-3}) and Lemma 3 (\ref{l3-1}), the second term of (\ref{l4-3}) is 
\begin{align*}
    &\quad \E_{\theta_{k,0}}\left[
    (e^{(j_1)}_{t_{i}^n}-e^{(j_1)}_{t_{i-1}^n})
    \Lambda_{k,0}[j_2,\cdot](f_{k,t_{i}^n}-f_{k,t_{i-1}^n})\right.\\
    &\left.\qquad\qquad\qquad\qquad\qquad\qquad\qquad\times
    \Lambda_{k,0}[j_3,\cdot](f_{k,t_{i}^n}-f_{k,t_{i-1}^n})
    \Lambda_{k,0}[j_4,\cdot](f_{k,t_{i}^n}-f_{k,t_{i-1}^n})
    |\mathscr{F}^{n}_{i-1}\right]\\
    &=\E_{\theta_{k,0}}\left[
    \Lambda_{k,0}[j_2,\cdot](f_{k,t_{i}^n}-f_{k,t_{i-1}^n})
    \Lambda_{k,0}[j_3,\cdot](f_{k,t_{i}^n}-f_{k,t_{i-1}^n})\right.\\
    &\qquad\qquad\qquad\qquad\qquad\qquad\qquad\left.\times
    \Lambda_{k,0}[j_4,\cdot](f_{k,t_{i}^n}-f_{k,t_{i-1}^n})|\mathscr{F}^{n}_{i-1}\right]\E_{\theta_{k,0}}\left[
    e^{(j_1)}_{t_{i}^n}-e^{(j_1)}_{t_{i-1}^n}
    |\mathscr{F}^{n}_{i-1}\right]\\
    &=R(h_n^2,f_{k,t_{i-1}^n})R(h_n,e_{t_{i-1}^n})\\
    &=h_n^3R(1,f_{k,t_{i-1}^n})R(1,e_{t_{i-1}^n})\quad (j_1,j_2,j_3,j_4=1,\cdots,p).
\end{align*}
From Lemma 2 (\ref{l2-3}) and Lemma 3 (\ref{l3-1}), the third term of (\ref{l4-3}) is 
\begin{align*}
    &\quad \E_{\theta_{k,0}}\left[
    \Lambda_{k,0}[j_1,\cdot](f_{k,t_{i}^n}-f_{k,t_{i-1}^n})
    (e^{(j_2)}_{t_{i}^n}-e^{(j_2)}_{t_{i-1}^n})\right.\\
    &\left.\qquad\qquad\qquad\qquad\qquad\qquad\qquad\times\Lambda_{k,0}[j_3,\cdot](f_{k,t_{i}^n}-f_{k,t_{i-1}^n})
    \Lambda_{k,0}[j_4,\cdot](f_{k,t_{i}^n}-f_{k,t_{i-1}^n})
    |\mathscr{F}^{n}_{i-1}\right]\\
    &=\E_{\theta_{k,0}}\left[
    \Lambda_{k,0}[j_1,\cdot](f_{k,t_{i}^n}-f_{k,t_{i-1}^n})
    \Lambda_{k,0}[j_3,\cdot](f_{k,t_{i}^n}-f_{k,t_{i-1}^n})\right.\\
    &\qquad\qquad\qquad\qquad\qquad\qquad\qquad\left.\times
    \Lambda_{k,0}[j_4,\cdot](f_{k,t_{i}^n}-f_{k,t_{i-1}^n})|\mathscr{F}^{n}_{i-1}\right]\E_{\theta_{k,0}}\left[
    e^{(j_2)}_{t_{i}^n}-e^{(j_2)}_{t_{i-1}^n}
    |\mathscr{F}^{n}_{i-1}\right]\\
    &=R(h_n^2,f_{k,t_{i-1}^n})
    R(h_n,e_{t_{i-1}^n})\\
    &=h_n^3R(1,f_{k,t_{i-1}^n})
    R(1,e_{t_{i-1}^n})\quad (j_1,j_2,j_3,j_4=1,\cdots,p).
\end{align*}
From Lemma 2 (\ref{l2-3}) and Lemma 3 (\ref{l3-1}), the fourth term of (\ref{l4-3}) is 
\begin{align*}
    &\quad \E_{\theta_{k,0}}\left[
    \Lambda_{k,0}[j_1,\cdot](f_{k,t_{i}^n}-f_{k,t_{i-1}^n})
    \Lambda_{k,0}[j_2,\cdot](f_{k,t_{i}^n}-f_{k,t_{i-1}^n})\right.\\
    &\qquad\qquad\qquad\qquad\qquad\qquad\qquad\qquad\qquad\qquad\left.\times (e^{(j_3)}_{t_{i}^n}-e^{(j_3)}_{t_{i-1}^n})
    \Lambda_{k,0}[j_4,\cdot](f_{k,t_{i}^n}-f_{k,t_{i-1}^n})
    |\mathscr{F}^{n}_{i-1}\right]\\
    &=\E_{\theta_{k,0}}\left[
    \Lambda_{k,0}[j_1,\cdot](f_{k,t_{i}^n}-f_{k,t_{i-1}^n})
    \Lambda_{k,0}[j_2,\cdot](f_{k,t_{i}^n}-f_{k,t_{i-1}^n})\right.\\
    &\qquad\qquad\qquad\qquad\qquad\qquad\qquad\qquad\left.\times
    \Lambda_{k,0}[j_4,\cdot](f_{k,t_{i}^n}-f_{k,t_{i-1}^n})|\mathscr{F}^{n}_{i-1}\right]\E_{\theta_{k,0}}\left[
    e^{(j_3)}_{t_{i}^n}-e^{(j_3)}_{t_{i-1}^n}
    |\mathscr{F}^{n}_{i-1}\right]\\
    &=R(h_n^2,f_{k,t_{i-1}^n})
    R(h_n,e_{t_{i-1}^n})\\
    &=h_n^3R(1,f_{k,t_{i-1}^n})R(1,e_{t_{i-1}^n})\quad (j_1,j_2,j_3,j_4=1,\cdots,p).
\end{align*}
From Lemma 2 (\ref{l2-3}) and Lemma 3 (\ref{l3-1}), the fifth term of (\ref{l4-3}) is 
\begin{align*}
    &\quad \E_{\theta_{k,0}}\left[
    \Lambda_{k,0}[j_1,\cdot](f_{k,t_{i}^n}-f_{k,t_{i-1}^n})
    \Lambda_{k,0}[j_2,\cdot](f_{k,t_{i}^n}-f_{k,t_{i-1}^n})\right.\\
    &\left.\qquad\qquad\qquad\qquad\qquad\qquad\qquad\qquad\qquad\qquad\qquad\times
    \Lambda_{k,0}[j_3,\cdot](f_{k,t_{i}^n}-f_{k,t_{i-1}^n})
    (e^{(j_4)}_{t_{i}^n}-e^{(j_4)}_{t_{i-1}^n})
    |\mathscr{F}^{n}_{i-1}\right] \\
    &=\E_{\theta_{k,0}}\left[
    \Lambda_{k,0}[j_1,\cdot](f_{k,t_{i}^n}-f_{k,t_{i-1}^n})
    \Lambda_{k,0}[j_2,\cdot](f_{k,t_{i}^n}-f_{k,t_{i-1}^n})\right.\\
    &\qquad\qquad\qquad\qquad\qquad\qquad\qquad\qquad\left.\times
    \Lambda_{k,0}[j_3,\cdot](f_{k,t_{i}^n}-f_{k,t_{i-1}^n})|\mathscr{F}^{n}_{i-1}\right]\E_{\theta_{k,0}}\left[
    e^{(j_4)}_{t_{i}^n}-e^{(j_4)}_{t_{i-1}^n}
    |\mathscr{F}^{n}_{i-1}\right]\\
    &=R(h_n^2,f_{k,t_{i-1}^n})R(h_n,e_{t_{i-1}^n})\\
    &=h_n^3R(1,f_{k,t_{i-1}^n})R(1,e_{t_{i-1}^n})\quad (j_1,j_2,j_3,j_4=1,\cdots,p).
\end{align*}
From Lemma 2 (\ref{l2-2}) and Lemma 3 (\ref{l3-2}), the sixth term of (\ref{l4-3}) is 
\begin{align*}
    &\quad \E_{\theta_{k,0}}\left[
    (e^{(j_1)}_{t_{i}^n}-e^{(j_1)}_{t_{i-1}^n})
    (e^{(j_2)}_{t_{i}^n}-e^{(j_2)}_{t_{i-1}^n})
    \Lambda_{k,0}[j_3,\cdot](f_{k,t_{i}^n}-f_{k,t_{i-1}^n})
    \Lambda_{k,0}[j_4,\cdot](f_{k,t_{i}^n}-f_{k,t_{i-1}^n})
    |\mathscr{F}^{n}_{i-1}\right]\\
    &=\E_{\theta_{k,0}}\left[
    \Lambda_{k,0}[j_3,\cdot](f_{k,t_{i}^n}-f_{k,t_{i-1}^n})
    \Lambda_{k,0}[j_4,\cdot](f_{k,t_{i}^n}-f_{k,t_{i-1}^n})
    |\mathscr{F}^{n}_{i-1}\right]\\
    &\qquad\qquad\qquad\qquad\qquad\qquad\qquad\qquad\qquad\qquad\qquad\times\E_{\theta_{k,0}}\left[
    (e^{(j_1)}_{t_{i}^n}-e^{(j_1)}_{t_{i-1}^n})
    (e^{(j_2)}_{t_{i}^n}-e^{(j_2)}_{t_{i-1}^n})
    |\mathscr{F}^{n}_{i-1}\right]\\
    &=\{h_n(\Lambda_{k,0}\Sigma_{ff,k,0}\Lambda_{k,0}^\top)_{j_3j_4}+R(h_n^2,f_{k,t_{i-1}^n})\}\{h_n(\Sigma_{ee,0})_{j_1j_2}+R(h_n^2,e_{t_{i-1}^n})\}\\
    &=h_n^2(\Lambda_{k,0}\Sigma_{ff,k,0}\Lambda_{k,0}^\top)_{j_3j_4}(\Sigma_{ee,0})_{j_1j_2}\\
    &\quad+h_n^3 R(1,f_{k,t_{i-1}^n})+h_n^4R(1,f_{k,t_{i-1}^n})R(1,e_{t_{i-1}^n})+h_n^3 R(1,e_{t_{i-1}^n})\quad (j_1,j_2,j_3,j_4=1,\cdots,p).
\end{align*}
From Lemma 2 (\ref{l2-2}) and Lemma 3 (\ref{l3-2}), the seventh term of (\ref{l4-3}) is 
\begin{align*}
    &\quad \E_{\theta_{k,0}}\left[
    (e^{(j_1)}_{t_{i}^n}-e^{(j_1)}_{t_{i-1}^n})
    \Lambda_{k,0}[j_2,\cdot](f_{k,t_{i}^n}-f_{k,t_{i-1}^n})
    (e^{(j_3)}_{t_{i}^n}-e^{(j_3)}_{t_{i-1}^n})
    \Lambda_{k,0}[j_4,\cdot](f_{k,t_{i}^n}-f_{k,t_{i-1}^n})
    |\mathscr{F}^{n}_{i-1}\right]\\
    &=\E_{\theta_{k,0}}\left[
    \Lambda_{k,0}[j_2,\cdot](f_{k,t_{i}^n}-f_{k,t_{i-1}^n})
    \Lambda_{k,0}[j_4,\cdot](f_{k,t_{i}^n}-f_{k,t_{i-1}^n})
    |\mathscr{F}^{n}_{i-1}\right]\\
    &\qquad\qquad\qquad\qquad\qquad\qquad\qquad\qquad\qquad\qquad\qquad\times\E_{\theta_{k,0}}\left[
    (e^{(j_1)}_{t_{i}^n}-e^{(j_1)}_{t_{i-1}^n})
    (e^{(j_3)}_{t_{i}^n}-e^{(j_3)}_{t_{i-1}^n})
    |\mathscr{F}^{n}_{i-1}\right]\\
    &=\{h_n(\Lambda_{k,0}\Sigma_{ff,k,0}\Lambda_{k,0}^\top)_{j_2j_4}+R(h_n^2,f_{k,t_{i-1}^n})\}\{h_n(\Sigma_{ee,0})_{j_1j_3}+R(h_n^2,e_{t_{i-1}^n})\}\\
    &=h_n^2(\Lambda_{k,0}\Sigma_{ff,k,0}\Lambda_{k,0}^\top)_{j_2j_4}(\Sigma_{ee,0})_{j_1j_3}\\
    &\quad+h_n^3 R(1,f_{k,t_{i-1}^n})+h_n^4R(1,f_{k,t_{i-1}^n})R(1,e_{t_{i-1}^n})+h_n^3 R(1,e_{t_{i-1}^n})\quad (j_1,j_2,j_3,j_4=1,\cdots,p).
\end{align*}
From Lemma 2 (\ref{l2-2}) and Lemma 3 (\ref{l3-2}), the eighth term of (\ref{l4-3}) is 
\begin{align*}
    &\quad \E_{\theta_{k,0}}\left[
    (e^{(j_1)}_{t_{i}^n}-e^{(j_1)}_{t_{i-1}^n})
    \Lambda_{k,0}[j_2,\cdot](f_{k,t_{i}^n}-f_{k,t_{i-1}^n})
    \Lambda_{k,0}[j_3,\cdot](f_{k,t_{i}^n}-f_{k,t_{i-1}^n})
    (e^{(j_4)}_{t_{i}^n}-e^{(j_4)}_{t_{i-1}^n})
    |\mathscr{F}^{n}_{i-1}\right]\\
    &=\E_{\theta_{k,0}}\left[
    \Lambda_{k,0}[j_2,\cdot](f_{k,t_{i}^n}-f_{k,t_{i-1}^n})
    \Lambda_{k,0}[j_3,\cdot](f_{k,t_{i}^n}-f_{k,t_{i-1}^n})
    |\mathscr{F}^{n}_{i-1}\right]\\
    &\qquad\qquad\qquad\qquad\qquad\qquad\qquad\qquad\qquad\qquad\times\E_{\theta_{k,0}}\left[
    (e^{(j_1)}_{t_{i}^n}-e^{(j_1)}_{t_{i-1}^n})
    (e^{(j_4)}_{t_{i}^n}-e^{(j_4)}_{t_{i-1}^n})
    |\mathscr{F}^{n}_{i-1}\right]\\
    &=\{h_n(\Lambda_{k,0}\Sigma_{ff,k,0}\Lambda_{k,0}^\top)_{j_2j_3}+R(h_n^2,f_{k,t_{i-1}^n})\}\{h_n(\Sigma_{ee,0})_{j_1j_4}+R(h_n^2,e_{t_{i-1}^n})\}\\
    &=h_n^2(\Lambda_{k,0}\Sigma_{ff,k,0}\Lambda_{k,0}^\top)_{j_2j_3}(\Sigma_{ee,0})_{j_1j_4}\\
    &\quad+h_n^3 R(1,f_{k,t_{i-1}^n})+h_n^4R(1,f_{k,t_{i-1}^n})R(1,e_{t_{i-1}^n})+h_n^3 R(1,e_{t_{i-1}^n})\quad (j_1,j_2,j_3,j_4=1,\cdots,p).
\end{align*}
From Lemma 2 (\ref{l2-2}) and Lemma 3 (\ref{l3-2}), the ninth term of (\ref{l4-3}) is 
\begin{align*}
    &\quad \E_{\theta_0}\left[
    \Lambda_{k,0}[j_1,\cdot](f_{k,t_{i}^n}-f_{k,t_{i-1}^n})
    (e^{(j_2)}_{t_{i}^n}-e^{(j_2)}_{t_{i-1}^n})
    (e^{(j_3)}_{t_{i}^n}-e^{(j_3)}_{t_{i-1}^n})
    \Lambda_{k,0}[j_4,\cdot](f_{k,t_{i}^n}-f_{k,t_{i-1}^n})
    |\mathscr{F}^{n}_{i-1}\right]\\
    &=\E_{\theta_{k,0}}\left[
    \Lambda_{k,0}[j_1,\cdot](f_{k,t_{i}^n}-f_{k,t_{i-1}^n})
    \Lambda_{k,0}[j_4,\cdot](f_{k,t_{i}^n}-f_{k,t_{i-1}^n})
    |\mathscr{F}^{n}_{i-1}\right]\\
    &\qquad\qquad\qquad\qquad\qquad\qquad\qquad\qquad\qquad\qquad\times
    \E_{\theta_{k,0}}\left[
    (e^{(j_2)}_{t_{i}^n}-e^{(j_2)}_{t_{i-1}^n})
    (e^{(j_3)}_{t_{i}^n}-e^{(j_3)}_{t_{i-1}^n})
    |\mathscr{F}^{n}_{i-1}\right]\\
    &=\{h_n(\Lambda_{k,0}\Sigma_{ff,k,0}\Lambda_{k,0}^\top)_{j_1j_4}+R(h_n^2,f_{k,t_{i-1}^n})\}\{h_n(\Sigma_{ee,0})_{j_2j_3}+R(h_n^2,e_{t_{i-1}^n})\}\\
    &=h_n^2(\Lambda_{k,0}\Sigma_{ff,k,0}\Lambda_{k,0}^\top)_{j_1j_4}(\Sigma_{ee,0})_{j_2j_3}\\
    &\quad+h_n^3 R(1,f_{k,t_{i-1}^n})+h_n^4R(1,f_{k,t_{i-1}^n})R(1,e_{t_{i-1}^n})+h_n^3 R(1,e_{t_{i-1}^n})\quad (j_1,j_2,j_3,j_4=1,\cdots,p).
\end{align*}
From Lemma 2 (\ref{l2-2}) and Lemma 3 (\ref{l3-2}), the tenth term of (\ref{l4-3}) is 
\begin{align*}
    &\quad \E_{\theta_{k,0}}\left[
    \Lambda_{k,0}[j_1,\cdot](f_{k,t_{i}^n}-f_{k,t_{i-1}^n})
    (e^{(j_2)}_{t_{i}^n}-e^{(j_2)}_{t_{i-1}^n})
    \Lambda_{k,0}[j_3,\cdot](f_{k,t_{i}^n}-f_{k,t_{i-1}^n})
    (e^{(j_4)}_{t_{i}^n}-e^{(j_4)}_{t_{i-1}^n})
    |\mathscr{F}^{n}_{i-1}\right]\\
    &=\E_{\theta_{k,0}}\left[
    \Lambda_{k,0}[j_1,\cdot](f_{k,t_{i}^n}-f_{k,t_{i-1}^n})
    \Lambda_{k,0}[j_3,\cdot](f_{k,t_{i}^n}-f_{k,t_{i-1}^n})
    |\mathscr{F}^{n}_{i-1}\right]\\
    &\qquad\qquad\qquad\qquad\qquad\qquad\qquad\qquad\qquad\qquad\times\E_{\theta_{k,0}}\left[
    (e^{(j_2)}_{t_{i}^n}-e^{(j_2)}_{t_{i-1}^n})
    (e^{(j_4)}_{t_{i}^n}-e^{(j_4)}_{t_{i-1}^n})
    |\mathscr{F}^{n}_{i-1}\right]\\
    &=\{h_n(\Lambda_{k,0}\Sigma_{ff,k,0}\Lambda_{k,0}^\top)_{j_1j_3}+R(h_n^2,f_{k,t_{i-1}^n})\}\{h_n(\Sigma_{ee,0})_{j_2j_4}+R(h_n^2,e_{t_{i-1}^n})\}\\
    &=h_n^2(\Lambda_{k,0}\Sigma_{ff,k,0}\Lambda_{k,0}^\top)_{j_1j_3}(\Sigma_{ee,0})_{j_2j_4}\\
    &\quad+h_n^3 R(1,f_{k,t_{i-1}^n})+h_n^4R(1,f_{k,t_{i-1}^n})R(1,e_{t_{i-1}^n})+h_n^3 R(1,e_{t_{i-1}^n})\quad (j_1,j_2,j_3,j_4=1,\cdots,p).
\end{align*}
From Lemma 2 (\ref{l2-2}) and Lemma 3 (\ref{l3-2}), the eleventh term of (\ref{l4-3}) is 
\begin{align*}
    &\quad \E_{\theta_{k,0}}\left[
    \Lambda_{k,0}[j_1,\cdot](f_{k,t_{i}^n}-f_{k,t_{i-1}^n})
    \Lambda_{k,0}[j_2,\cdot](f_{k,t_{i}^n}-f_{k,t_{i-1}^n})
    (e^{(j_3)}_{t_{i}^n}-e^{(j_3)}_{t_{i-1}^n})
    (e^{(j_4)}_{t_{i}^n}-e^{(j_4)}_{t_{i-1}^n})
    |\mathscr{F}^{n}_{i-1}\right]\\
    &=\E_{\theta_{k,0}}\left[
    \Lambda_{k,0}[j_1,\cdot](f_{k,t_{i}^n}-f_{k,t_{i-1}^n})
    \Lambda_{k,0}[j_2,\cdot](f_{k,t_{i}^n}-f_{k,t_{i-1}^n})
    |\mathscr{F}^{n}_{i-1}\right]\\
    &\qquad\qquad\qquad\qquad\qquad\qquad\qquad\qquad\qquad\qquad\times\E_{\theta_{k,0}}\left[
    (e^{(j_3)}_{t_{i}^n}-e^{(_3)}_{t_{i-1}^n})
    (e^{(j_4)}_{t_{i}^n}-e^{(j_4)}_{t_{i-1}^n})
    |\mathscr{F}^{n}_{i-1}\right]\\
    &=\{h_n(\Lambda_{k,0}\Sigma_{ff,k,0}\Lambda_{k,0}^\top)_{j_1j_2}+R(h_n^2,f_{k,t_{i-1}^n})\}\{h_n(\Sigma_{ee,0})_{j_3j_4}+R(h_n^2,e_{t_{i-1}^n})\}\\
    &=h_n^2(\Lambda_{k,0}\Sigma_{ff,k,0}\Lambda_{k,0}^\top)_{j_1j_2}(\Sigma_{ee,0})_{j_3j_4}\\
    &\quad+h_n^3 R(1,f_{k,t_{i-1}^n})+h_n^4R(1,f_{k,t_{i-1}^n})R(1,e_{t_{i-1}^n})+h_n^3 R(1,e_{t_{i-1}^n})\quad (j_1,j_2,j_3,j_4=1,\cdots,p).
\end{align*}
From Lemma 2 (\ref{l2-1}) and Lemma 3 (\ref{l3-3}), the twelfth term of (\ref{l4-3}) is 
\begin{align*}
    &\quad \E_{\theta_{k,0}}\left[
    \Lambda_{k,0}[j_1,\cdot](f_{k,t_{i}^n}-f_{k,t_{i-1}^n})
    (e^{(j_2)}_{t_{i}^n}-e^{(j_2)}_{t_{i-1}^n})
    (e^{(j_3)}_{t_{i}^n}-e^{(j_3)}_{t_{i-1}^n})
    (e^{(j_4)}_{t_{i}^n}-e^{(j_4)}_{t_{i-1}^n})
    |\mathscr{F}^{n}_{i-1}\right]\\
    &=\E_{\theta_{k,0}}\left[\Lambda_{k,0}[j_1,\cdot](f_{k,t_{i}^n}-f_{k,t_{i-1}^n})|\mathscr{F}^{n}_{i-1}\right]\E_{\theta_{k,0}}\left[
    (e^{(j_2)}_{t_{i}^n}-e^{(j_2)}_{t_{i-1}^n})
    (e^{(j_3)}_{t_{i}^n}-e^{(j_3)}_{t_{i-1}^n})
    (e^{(j_4)}_{t_{i}^n}-e^{(j_4)}_{t_{i-1}^n})
    |\mathscr{F}^{n}_{i-1}\right]\\
    &=R(h_n,f_{k,t_{i-1}^n})R(h_n^2,e_{t_{i-1}^n})\\
    &=h_n^3R(1,f_{k,t_{i-1}^n})R(1,e_{t_{i-1}^n})\quad (j_1,j_2,j_3,j_4=1,\cdots,p).
\end{align*}
From Lemma 2 (\ref{l2-1}) and Lemma 3 (\ref{l3-3}), the thirteenth term of (\ref{l4-3}) is 
\begin{align*}
    &\quad \E_{\theta_{k,0}}\left[
    (e^{(j_1)}_{t_{i}^n}-e^{(j_1)}_{t_{i-1}^n})
    \Lambda_{k,0}[j_2,\cdot](f_{k,t_{i}^n}-f_{k,t_{i-1}^n})
    (e^{(j_3)}_{t_{i}^n}-e^{(j_3)}_{t_{i-1}^n})
    (e^{(j_4)}_{t_{i}^n}-e^{(j_4)}_{t_{i-1}^n})
    |\mathscr{F}^{n}_{i-1}\right]\\
    &=\E_{\theta_{k,0}}\left[\Lambda_{k,0}[j_2,\cdot](f_{k,t_{i}^n}-f_{k,t_{i-1}^n})|\mathscr{F}^{n}_{i-1}\right]\E_{\theta_{k,0}}\left[
    (e^{(j_1)}_{t_{i}^n}-e^{(j_1)}_{t_{i-1}^n})
    (e^{(j_3)}_{t_{i}^n}-e^{(j_3)}_{t_{i-1}^n})
    (e^{(j_4)}_{t_{i}^n}-e^{(j_4)}_{t_{i-1}^n})
    |\mathscr{F}^{n}_{i-1}\right]\\
    &=R(h_n,f_{k,t_{i-1}^n})R(h_n^2,e_{t_{i-1}^n})\\
    &=h_n^3 R(1,f_{k,t_{i-1}^n})R(1,e_{t_{i-1}^n})\quad (j_1,j_2,j_3,j_4=1,\cdots,p).
\end{align*}
From Lemma 2 (\ref{l2-1}) and Lemma 3 (\ref{l3-3}), the fourteenth term of (\ref{l4-3}) is 
\begin{align*}
    &\quad \E_{\theta_{k,0}}\left[
    (e^{(j_1)}_{t_{i}^n}-e^{(j_1)}_{t_{i-1}^n})
    (e^{(j_2)}_{t_{i}^n}-e^{(j_2)}_{t_{i-1}^n})
    \Lambda_{k,0}[j_3,\cdot](f_{k,t_{i}^n}-f_{k,t_{i-1}^n})
    (e^{(j_4)}_{t_{i}^n}-e^{(j_4)}_{t_{i-1}^n})
    |\mathscr{F}^{n}_{i-1}\right]\\
    &=\E_{\theta_{k,0}}\left[\Lambda_{k,0}[j_3,\cdot](f_{k,t_{i}^n}-f_{k,t_{i-1}^n})|\mathscr{F}^{n}_{i-1}\right]\E_{\theta_{k,0}}\left[
    (e^{(j_1)}_{t_{i}^n}-e^{(j_1)}_{t_{i-1}^n})
    (e^{(j_2)}_{t_{i}^n}-e^{(j_2)}_{t_{i-1}^n})
    (e^{(j_4)}_{t_{i}^n}-e^{(j_4)}_{t_{i-1}^n})
    |\mathscr{F}^{n}_{i-1}\right]\\
    &=R(h_n,f_{k,t_{i-1}^n})R(h_n^2,e_{t_{i-1}^n})\\
    &=h_n^3 R(1,f_{k,t_{i-1}^n})R(1,e_{t_{i-1}^n})\quad (j_1,j_2,j_3,j_4=1,\cdots,p).
\end{align*}
From Lemma 2 (\ref{l2-1}) and Lemma 3 (\ref{l3-3}), the fifteenth term of (\ref{l4-3}) is 
\begin{align*}
    &\quad \E_{\theta_{k,0}}\left[
    (e^{(j_1)}_{t_{i}^n}-e^{(j_1)}_{t_{i-1}^n})
    (e^{(j_2)}_{t_{i}^n}-e^{(j_2)}_{t_{i-1}^n})
    (e^{(j_3)}_{t_{i}^n}-e^{(j_3)}_{t_{i-1}^n})
    \Lambda_{k,0}[j_4,\cdot](f_{k,t_{i}^n}-f_{k,t_{i-1}^n})
    |\mathscr{F}^{n}_{i-1}\right]\\
    &=\E_{\theta_{k,0}}\left[\Lambda_{k,0}[j_4,\cdot](f_{k,t_{i}^n}-f_{k,t_{i-1}^n})|\mathscr{F}^{n}_{i-1}\right]\E_{\theta_{k,0}}\left[
    (e^{(j_1)}_{t_{i}^n}-e^{(j_1)}_{t_{i-1}^n})
    (e^{(j_2)}_{t_{i}^n}-e^{(j_2)}_{t_{i-1}^n})
    (e^{(j_3)}_{t_{i}^n}-e^{(j_3)}_{t_{i-1}^n})
    |\mathscr{F}^{n}_{i-1}\right]\\
    &=R(h_n,f_{k,t_{i-1}^n})R(h_n^2,e_{t_{i-1}^n})\\
    &=h_n^3R(1,f_{k,t_{i-1}^n})R(1,e_{t_{i-1}^n})\quad (j_1,j_2,j_3,j_4=1,\cdots,p).
\end{align*}
From Lemma 3 (\ref{l3-4}), the sixteenth term of (\ref{l4-3}) is 
\begin{align*}
    &\quad \E_{\theta_{k,0}}\left[
    (e^{(j_1)}_{t_{i}^n}-e^{(j_1)}_{t_{i-1}^n})
    (e^{(j_2)}_{t_{i}^n}-e^{(j_2)}_{t_{i-1}^n})
    (e^{(j_3)}_{t_{i}^n}-e^{(j_3)}_{t_{i-1}^n})
    (e^{(j_4)}_{t_{i}^n}-e^{(j_4)}_{t_{i-1}^n})
    |\mathscr{F}^{n}_{i-1}\right]\\
    &=h_n^2\{(\Sigma_{ee,0})_{j_1j_2}(\Sigma_{ee,0})_{j_3j_4}
    +(\Sigma_{ee,0})_{j_1j_3}(\Sigma_{ee,0})_{j_2j_4}
    +(\Sigma_{ee,0})_{j_1j_4}(\Sigma_{ee,0})_{j_2j_3}\}
    +R(h_n^3,e_{t_{i-1}^n})\\
    &=h_n^2\left\{(\Sigma_{ee,0})_{j_1j_2}(\Sigma_{ee,0})_{j_3j_4}
    +(\Sigma_{ee,0})_{j_1j_3}(\Sigma_{ee,0})_{j_2j_4}
    +(\Sigma_{ee,0})_{j_1j_4}(\Sigma_{ee,0})_{j_2j_3}\right\}
    +h_n^3R(1,e_{t_{i-1}^n})\\ 
    &\qquad\qquad\qquad\qquad\qquad\qquad\qquad\qquad\qquad\qquad\qquad\qquad\qquad\qquad\quad (j_1,j_2,j_3,j_4=1,\cdots,p).
\end{align*}
Since
\begin{align*}
    &\quad \E_{\theta_{k,0}}\left[ (X^{(j_1)}_{t_{i}^n}-X^{(j_1)}_{t_{i-1}^n})(X^{(j_2)}_{t_{i}^n}-X^{(j_2)}_{t_{i-1}^n})(X^{(j_3)}_{t_{i}^n}-X^{(j_3)}_{t_{i-1}^n})(X^{(j_4)}_{t_{i}^n}-X^{(j_4)}_{t_{i-1}^n})| \mathscr{F}^{n}_{i-1}\right]\\
    &=h_n^2\{(\Lambda_{k,0}\Sigma_{ff,k,0}\Lambda_{k,0}^\top)_{j_1j_2}(\Lambda_{k,0}\Sigma_{ff,k,0}\Lambda_{k,0}^\top)_{j_3j_4}
    +(\Lambda_{k,0}\Sigma_{ff,k,0}\Lambda_{k,0}^\top)_{j_1j_3}(\Lambda_{k,0}\Sigma_{ff,k,0}\Lambda_{k,0}^\top)_{j_2j_4}\\
    &\quad
    +(\Lambda_{k,0}\Sigma_{ff,k,0}\Lambda_{k,0}^\top)_{j_1j_4}(\Lambda_{k,0}\Sigma_{ff,k,0}\Lambda_{k,0}^\top)_{j_2j_3}\}+h_n^2(\Lambda_{k,0}\Sigma_{ff,k,0}\Lambda_{k,0}^\top)_{j_3j_4}(\Sigma_{ee,0})_{j_1j_2}\\
    &\quad+h_n^2(\Lambda_{k,0}\Sigma_{ff,k,0}\Lambda_{k,0}^\top)_{j_2j_4}(\Sigma_{ee,0})_{j_1j_3}
    +h_n^2(\Lambda_{k,0}\Sigma_{ff,k,0}\Lambda_{k,0}^\top)_{j_2j_3}(\Sigma_{ee,0})_{j_1j_4}\\
    &\quad+h_n^2(\Lambda_{k,0}\Sigma_{ff,k,0}\Lambda_{k,0}^\top)_{j_1j_4}(\Sigma_{ee,0})_{j_2j_3}+h_n^2(\Lambda_{k,0}\Sigma_{ff,k,0}\Lambda_{k,0}^\top)_{j_1j_3}(\Sigma_{ee,0})_{j_2j_4}\\
    &\quad+h_n^2(\Lambda_{k,0}\Sigma_{ff,k,0}\Lambda_{k,0}^\top)_{j_1j_2}(\Sigma_{ee,0})_{j_3j_4} +h_n^2\{(\Sigma_{ee,0})_{j_1j_2}(\Sigma_{ee,0})_{j_3j_4}+(\Sigma_{ee,0})_{j_1j_3}(\Sigma_{ee,0})_{j_2j_4}\\
    &\quad+(\Sigma_{ee,0})_{j_1j_4}(\Sigma_{ee,0})_{j_2j_3}\}+h_n^3\{R(1,f_{k,t_{i-1}^n})+R(1,f_{k,t_{i-1}^n})R(1,e_{t_{i-1}^n})+R(1,e_{t_{i-1}^n})\}\\
    &=h_n^2\left\{(\Sigma_{k}(\theta_{k,0}))_{j_1j_2}(\Sigma_{k}(\theta_{k,0}))_{j_3j_4}+(\Sigma_{k}(\theta_{k,0}))_{j_1j_3}(\Sigma_{k}(\theta_{k,0}))_{j_2j_4}+(\Sigma_{k}(\theta_{k,0}))_{j_1j_4}(\Sigma_{k}(\theta_{k,0}))_{j_2j_3}\right\}\\
    &\quad +h_n^3\{R(1,f_{k,t_{i-1}^n})+R(1,f_{k,t_{i-1}^n})R(1,e_{t_{i-1}^n})+R(1,e_{t_{i-1}^n})\}\quad (j_1,j_2,j_3,j_4=1,\cdots,p),
\end{align*}
(\ref{l4-2}) is deduced. 
\end{proof}
\begin{prf}
First, we will prove (\ref{the1-1}). 
It follows from Lemma 1 and  Slutsky's theorem that
\begin{align*}
    Q_{XX}&=\frac{1}{T}\sum_{i=1}^n(X_{t_{i}^n}-X_{t_{i-1}^n})(X_{t_{i}^n}-X_{t_{i-1}^n})^\top\\
    &=\frac{1}{T}\sum_{i=1}^n\{\Lambda_{k,0}(f_{k,t_{i}^n}-f_{k,t_{i-1}^n})+(e_{t_{i}^n}-e_{t_{i-1}^n})\}\{\Lambda_{k,0}(f_{k,t_{i}^n}-f_{k,t_{i-1}^n})+(e_{t_{i}^n}-e_{t_{i-1}^n})\}^\top\\
    &=\Lambda_{k,0}\left\{\frac{1}{T}\sum_{i=1}^n(f_{k,t_{i}^n}-f_{k,t_{i-1}^n})(f_{k,t_{i}^n}-f_{k,t_{i-1}^n})^\top\right\}\Lambda_{k,0}^\top\\ &\quad +\Lambda_{k,0}\left\{\frac{1}{T}\sum_{i=1}^n(f_{k,t_{i}^n}-f_{k,t_{i-1}^n})(e_{t_{i}^n}-e_{t_{i-1}^n})^\top\right\}\\
    &\quad +\left\{\frac{1}{T}\sum_{i=1}^n(e_{t_{i}^n}-e_{t_{i-1}^n})(f_{k,t_{i}^n}-f_{k,t_{i-1}^n})^\top\right\}\Lambda_{k,0}^\top\\
    &\quad+\frac{1}{T}\sum_{i=1}^n(e_{t_{i}^n}-e_{t_{i-1}^n})(e_{t_{i}^n}-e_{t_{i-1}^n})^\top\\
    &=\Lambda_{k,0} Q_{ff,k}\Lambda_{k,0}^\top+\Lambda_{k,0} Q_{fe,k}+Q_{fe,k}^\top\Lambda_{k,0}^\top
    +Q_{ee}\\
    &\stackrel{P_{\theta_{k,0}}}{\to} 
    \Lambda_{k,0}\Sigma_{ff,k,0}\Lambda_{k,0}^\top+\Sigma_{ee,0}=\Sigma_{k}(\theta_{k,0}), 
\end{align*}
so that we obtain (\ref{the1-1}). 
Next, in order to show  (\ref{the1-2}), we prove that
\begin{align}
    \label{th1-1}
    \sqrt{n}(\vec(Q_{XX})-\vec(\Sigma_{k}(\theta_{k,0})))\stackrel{d}{\to} N_{p^2}(0,\Gamma_{k}(\theta_{k,0})),
\end{align}
where 
\begin{align*}
    \Gamma_{k}(\theta_{k,0})_{(j_1,j_2),(j_3,j_4)}&=(\Sigma_{k}(\theta_{k,0}))_{j_{1}j_{3}}(\Sigma_{k}(\theta_{k,0}))_{j_{2}j_{4}}+(\Sigma_{k}(\theta_{k,0}))_{j_{1}j_{4}}(\Sigma_{k}(\theta_{k,0}))_{j_{2}j_{3}}\\ &\qquad\qquad\qquad\qquad\qquad\qquad\qquad\qquad(j_1,j_2,j_3,j_4=1,\cdots,p).
\end{align*}
Setting
\begin{align*}
L_{i,n}=
\begin{pmatrix}
\frac{1}{\sqrt{n}h_n}(X_{t_{i}^n}^{(1)}-X_{t_{i-1}^n}^{(1)})^2-\frac{1}{\sqrt{n}}(\Sigma_{k}(\theta_{k,0}))_{11}\\
\vdots\\
\frac{1}{\sqrt{n}h_n}(X_{t_{i}^n}^{(p)}-X_{t_{i-1}^n}^{(p)})(X_{t_{i}^n}^{(p)}-X_{t_{i-1}^n}^{(1)})
-\frac{1}{\sqrt{n}}(\Sigma_{k}(\theta_{k,0}))_{p1}\\
\vdots\\
\vdots\\
\frac{1}{\sqrt{n}h_n}(X_{t_{i}^n}^{(1)}-X_{t_{i-1}^n}^{(1)})(X_{t_{i}^n}^{(p)}-X_{t_{i-1}^n}^{(p)})
-\frac{1}{\sqrt{n}}(\Sigma_{k}(\theta_{k,0}))_{1p}\\
\vdots\\
\frac{1}{\sqrt{n}h_n}(X_{t_{i}^n}^{(p)}-X_{t_{i-1}^n}^{(p)})^2-\frac{1}{\sqrt{n}}(\Sigma_{k}(\theta_{k,0}))_{pp}
\end{pmatrix}, 
\end{align*}
we have
\begin{align*}
\sqrt{n}(\vec(Q_{XX})-\vec(\Sigma_{k}(\theta_{k,0})))
=\sum_{i=1}^n L_{i,n}.
\end{align*}
Let 
\begin{align*}
    M_{i,n}= L_{i,n}-\E_{\theta_{k,0}}\left[L_{i,n}|\mathscr{F}^{n}_{i-1}\right]\quad(i=1,\cdots,p^2).
\end{align*}
Note that $M_{i,n}$ is zero-mean martingale.
If it is shown that
\begin{align}
    \label{th1-2}
    \sum_{i=1}^n\E_{\theta_{k,0}}\left[M_{i,n}M_{i,n}^\top| \mathscr{F}^{n}_{i-1}\right]&\stackrel{P_{\theta_{k,0}}}{\to}
    \Gamma_{k}(\theta_{k,0}),
\end{align}
and
\begin{align}
    \label{th1-3}
    \sum_{i=1}^n\E_{\theta_{k,0}}\left[|M_{i,n}|^4| \mathscr{F}^{n}_{i-1}\right]&\stackrel{P_{\theta_{k,0}}}{\to}0,
\end{align}
then, we see from Theorem 3.2. in Hall and Heyde \cite{hall(1981)} that
\begin{align*}
    \sum_{i=1}^nM_{i,n}\stackrel{d}{\to} N_{p^2}(0,\Gamma_{k}(\theta_{k,0})).
\end{align*}
Furthermore, if it is proved that
\begin{align}
    \label{th1-4}
    \sum_{i=1}^n\E_{\theta_{k,0}}\left[L_{i,n}| \mathscr{F}^{n}_{i-1}\right]\stackrel{P_{\theta_{k,0}}}{\to}0,
\end{align}
we have 
\begin{align*}
    \sum_{i=1}^n M_{i,n}-\sum_{i=1}^n L_{i,n}\stackrel{P_{\theta_{k,0}}}{\to}0.
\end{align*}
Hence, 
(\ref{th1-2}), (\ref{th1-3}) and (\ref{th1-4}) imply that
\begin{align*}
     \sum_{i=1}^n L_{i,n}\stackrel{d}{\to} N_{p^2}(0,\Gamma_{k}(\theta_{k,0}))
\end{align*}
from Slutsky's theorem, and (\ref{th1-1}) is obtained. 
Since
\begin{align*}
    &\quad\sum_{i=1}^n\E_{\theta_{k,0}}\left[M_{i,n}M_{i,n}^\top| \mathscr{F}^{n}_{i-1}\right]\\
    &=\sum_{i=1}^n\E_{\theta_{k,0}}\left[(L_{i,n}-\E_{\theta_{k,0}}\left[L_{i,n}| \mathscr{F}^{n}_{i-1}\right])(L_{i,n}-\E_{\theta_{k,0}}\left[L_{i,n}| \mathscr{F}^{n}_{i-1}\right])^\top| \mathscr{F}^{n}_{i-1}\right]\\
    &=\sum_{i=1}^n\E_{\theta_{k,0}}\left[L_{i,n}L_{i,n}^\top| \mathscr{F}^{n}_{i-1}\right]
    -\sum_{i=1}^n\E_{\theta_{k,0}}\left[L_{i,n}| \mathscr{F}^{n}_{i-1}\right]
    \E_{\theta_{k,0}}\left[L_{i,n}| \mathscr{F}^{n}_{i-1}\right]^\top,
\end{align*}
and 
\begin{align*}
    0\leq\sum_{i=1}^n\E_{\theta_{k,0}}\left[|M_{i,n}|^4| \mathscr{F}^{n}_{i-1}\right]
    &=\sum_{i=1}^n\E_{\theta_{k,0}}\left[|L_{i,n}-\E_{\theta_{k,0}}\left[L_{i,n}| \mathscr{F}^{n}_{i-1}\right]|^4| \mathscr{F}^{n}_{i-1}\right]\\
    &\leq8\sum_{i=1}^n\left\{\E_{\theta_{k,0}}\left[|L_{i,n}|^4| \mathscr{F}^{n}_{i-1}\right]+|\E_{\theta_{k,0}}\left[L_{i,n}| \mathscr{F}^{n}_{i-1}\right]|^4\right\}\\
    &\leq8\sum_{i=1}^n\left\{\E_{\theta_{k,0}}\left[|L_{i,n}|^4| \mathscr{F}^{n}_{i-1}\right]+\E_{\theta_{k,0}}\left[|L_{i,n}|^4| \mathscr{F}^{n}_{i-1}\right]\right\}\\
    &=16\sum_{i=1}^n\E_{\theta_{k,0}}\left[|L_{i,n}|^4| \mathscr{F}^{n}_{i-1}\right],
\end{align*}
it is sufficient to prove the following three convergences in order to show (\ref{th1-1}).
\begin{align}
    \label{th1-5}
    \sum_{i=1}^n\E_{\theta_{k,0}}\left[L_{i,n}| \mathscr{F}^{n}_{i-1}\right]&\stackrel{P_{\theta_{k,0}}}{\to}0, \\
    \label{th1-6}
    \sum_{i=1}^n\E_{\theta_{k,0}}\left[L_{i,n}L_{i,n}^\top| \mathscr{F}^{n}_{i-1}\right]
    -\sum_{i=1}^n\E_{\theta_{k,0}}\left[L_{i,n}| \mathscr{F}^{n}_{i-1}\right]
    \E_{\theta_{k,0}}\left[L_{i,n}| \mathscr{F}^{n}_{i-1}\right]^\top
    &\stackrel{P_{\theta_{k,0}}}{\to} \Gamma_{k}(\theta_{k,0}),
    \\
    \label{th1-7}
    \sum_{i=1}^n\E_{\theta_{k,0}}\left[|L_{i,n}|^4
    |\mathscr{F}^{n}_{i-1}\right]&\stackrel{P_{\theta_{k,0}}}{\to}0. 
\end{align}
In order to show (\ref{th1-5}), it is sufficient to prove that
\begin{align*}
    &\sum_{i=1}^n\left\{\frac{1}{\sqrt{n}h_n}\E_{\theta_{k,0}}\left[(X_{t_{i}^n}^{(j_1)}-X_{t_{i-1}^n}^{(j_1)})(X_{t_{i}^n}^{(j_2)}-X_{t_{i-1}^n}^{(j_2)})|\mathscr{F}^{n}_{i-1}\right]-\frac{1}{\sqrt{n}}(\Sigma_{k}(\theta_{k,0}))_{j_1j_2}\right\}\stackrel{P_{\theta_{k,0}}}{\to}0\\
    &\qquad\qquad\qquad\qquad\qquad\qquad\qquad\qquad\qquad\qquad\qquad\qquad\qquad\qquad\qquad\quad(j_1,j_2=1,\cdots p).
\end{align*}
Since it follows from Lemma 4 that
\begin{align*}
    &\quad \sum_{i=1}^n\left\{\frac{1}{\sqrt{n}h_n}\E_{\theta_{k,0}}\left[(X_{t_{i}^n}^{(j_1)}-X_{t_{i-1}^n}^{(j_1)})(X_{t_{i}^n}^{(j_2)}-X_{t_{i-1}^n}^{(j_2)})|\mathscr{F}^{n}_{i-1}\right]-\frac{1}{\sqrt{n}}(\Sigma_{k}(\theta_{k,0}))_{j_1j_2}\right\}\\
    &=\frac{1}{\sqrt{n}}\sum_{i=1}^n\left\{
    \frac{1}{h_n}\left\{h_n(\Sigma_{k}(\theta_{k,0}))_{j_1j_2}+h_n^2\{R(1,f_{k,t_{i-1}^n})\right.\right.\\
    &\qquad\qquad\qquad\qquad\qquad\qquad\qquad\left.\left.+R(1,f_{k,t_{i-1}^n})R(1,e_{t_{i-1}^n})
    +R(1,e_{t_{i-1}^n})\}\right\}-(\Sigma_{k}(\theta_{k,0}))_{j_1j_2}\right\}\\
    &=\frac{h_n}{\sqrt{n}}\sum_{i=1}^nR(1,f_{k,t_{i-1}^n})
    +\frac{h_n}{\sqrt{n}}\sum_{i=1}^nR(1,f_{k,t_{i-1}^n})R(1,e_{t_{i-1}^n})
    +\frac{h_n}{\sqrt{n}}\sum_{i=1}^nR(1,e_{t_{i-1}^n})
    \\
    &=\sqrt{nh_n^2}\frac{1}{n}\sum_{i=1}^nR(1,f_{k,t_{i-1}^n})+\sqrt{nh_n^2}\frac{1}{n}\sum_{i=1}^nR(1,f_{k,t_{i-1}^n})R(1,e_{t_{i-1}^n})+\sqrt{nh_n^2}\frac{1}{n}\sum_{i=1}^nR(1,e_{t_{i-1}^n})\\
    &\stackrel{P_{\theta_{k,0}}}{\to}0   \quad (j_1,j_2=1,\cdots p),
\end{align*}
we obtain (\ref{th1-5}). In order to prove (\ref{th1-6}), it is sufficient to show that
\begin{align}
    &\quad \sum_{i=1}^n\E_{\theta_{k,0}}\left[\left\{\frac{1}{\sqrt{n}h_n}(X_{t_{i}^n}^{(j_1)}-X_{t_{i-1}^n}^{(j_1)})(X_{t_{i}^n}^{(j_2)}-X_{t_{i-1}^n}^{(j_2)})-\frac{1}{\sqrt{n}}(\Sigma_{k}(\theta_{k,0}))_{j_1j_2}\right\}\nonumber \right.\\
    &\left.\qquad \qquad \qquad \times \left\{\frac{1}{\sqrt{n}h_n} (X_{t_{i}^n}^{(j_3)}-X_{t_{i-1}^n}^{(j_3)})(X_{t_{i}^n}^{(j_4)}-X_{t_{i-1}^n}^{(j_4)})-\frac{1}{\sqrt{n}}(\Sigma_{k}(\theta_{k,0}))_{j_3j_4}\right\}
    | \mathscr{F}^{n}_{i-1}\right]\nonumber \\
    &\quad -\sum_{i=1}^n\E_{\theta_{k,0}}\left[\left\{\frac{1}{\sqrt{n}h_n}(X_{t_{i}^n}^{(j_1)}-X_{t_{i-1}^n}^{(j_1)})(X_{t_{i}^n}^{(j_2)}-X_{t_{i-1}^n}^{(j_2)})-\frac{1}{\sqrt{n}}(\Sigma_{k}(\theta_{k,0}))_{j_1j_2}\right\}| \mathscr{F}^{n}_{i-1}\right]\label{th1-8}\\
    &\qquad \qquad \qquad \times \E_{\theta_{k,0}}\left[\left\{\frac{1}{\sqrt{n}h_n} (X_{t_{i}^n}^{(j_3)}-X_{t_{i-1}^n}^{(j_3)})(X_{t_{i}^n}^{(j_4)}-X_{t_{i-1}^n}^{(j_4)})-\frac{1}{\sqrt{n}}(\Sigma_{k}(\theta_{k,0}))_{j_3j_4}\right\}| \mathscr{F}^{n}_{i-1}\right]\nonumber \\
    &\stackrel{P_{\theta_{k,0}}}{\to}(\Gamma_{k}(\theta_{k,0}))_{(j_1,j_2),(j_3,j_4)}\quad(j_1,j_2,j_3,j_4=1,\cdots p).\nonumber 
\end{align}
From Lemma 4, the first term of (\ref{th1-8}) is 
\begin{align*}
    &\quad \sum_{i=1}^n\E_{\theta_{k,0}}\left[\left\{\frac{1}{\sqrt{n}h_n}(X_{t_{i}^n}^{(j_1)}-X_{t_{i-1}^n}^{(j_1)})(X_{t_{i}^n}^{(j_2)}-X_{t_{i-1}^n}^{(j_2)})-\frac{1}{\sqrt{n}}(\Sigma_{k}(\theta_{k,0}))_{j_1j_2}\right\}\right.\\
    &\left.\qquad \qquad \qquad \times \left\{\frac{1}{\sqrt{n}h_n} (X_{t_{i}^n}^{(j_3)}-X_{t_{i-1}^n}^{(j_3)})(X_{t_{i}^n}^{(j_4)}-X_{t_{i-1}^n}^{(j_4)})-\frac{1}{\sqrt{n}}(\Sigma_{k}(\theta_{k,0}))_{j_3j_4}\right\}
    | \mathscr{F}^{n}_{i-1}\right]
    \\
    &=\frac{1}{nh_n^2}\sum_{i=1}^n\E_{\theta_{k,0}}\left[(X_{t_{i}^n}^{(j_1)}-X_{t_{i-1}^n}^{(j_1)})(X_{t_{i}^n}^{(j_2)}-X_{t_{i-1}^n}^{(j_2)})(X_{t_{i}^n}^{(j_3)}-X_{t_{i-1}^n}^{(j_3)})(X_{t_{i}^n}^{(j_4)}-X_{t_{i-1}^n}^{(j_4)})| \mathscr{F}^{n}_{i-1}\right]\\
    &\quad -\frac{1}{nh_n}\sum_{i=1}^n\E_{\theta_{k,0}}\left[(X_{t_{i}^n}^{(j_1)}-X_{t_{i-1}^n}^{(j_1)})(X_{t_{i}^n}^{(j_2)}-X_{t_{i-1}^n}^{(j_2)})| \mathscr{F}^{n}_{i-1}\right](\Sigma_{k}(\theta_{k,0}))_{j_3j_4}\\
    &\quad -\frac{1}{nh_n}\sum_{i=1}^n\E_{\theta_{k,0}}\left[(X_{t_{i}^n}^{(j_3)}-X_{t_{i-1}^n}^{(j_3)})(X_{t_{i}^n}^{(j_4)}-X_{t_{i-1}^n}^{(j_4)})| \mathscr{F}^{n}_{i-1}\right](\Sigma_{k}(\theta_{k,0}))_{j_1j_2}\\
    &\quad+(\Sigma_{k}(\theta_{k,0}))_{j_1j_2}(\Sigma_{k}(\theta_{k,0}))_{j_3j_4}\\
     &=\frac{1}{nh_n^2}\sum_{i=1}^nh_n^2\{(\Sigma_{k}(\theta_{k,0}))_{j_1j_2}(\Sigma_{k}(\theta_{k,0}))_{j_3j_4}+(\Sigma_{k}(\theta_{k,0}))_{j_1j_3}(\Sigma_{k}(\theta_{k,0}))_{j_2j_4}\\
     &\qquad\qquad\qquad\qquad\qquad\qquad\qquad\qquad\qquad\qquad\qquad+(\Sigma_{k}(\theta_{k,0}))_{j_1j_4}(\Sigma_{k}(\theta_{k,0}))_{j_2j_3}\}\\
    &\quad +\frac{1}{nh_n^2}\sum_{i=1}^nh_n^3\{R(1,f_{k,t_{i-1}^n})+R(1,f_{k,t_{i-1}^n})R(1,e_{t_{i-1}^n})+R(1,e_{t_{i-1}^n})  \}
    \\
    &\quad-\frac{1}{nh_n}\sum_{i=1}^n\left\{h_n(\Sigma_k(\theta_{k,0}))_{j_1j_2}+h_n^2\{R(1,f_{k,t_{i-1}^n})\right.\\
    &\qquad\qquad\qquad\qquad\qquad\left.+R(1,f_{k,t_{i-1}^n})R(1,e_{t_{i-1}^n})+R(1,e_{t_{i-1}^n})\}\right\}(\Sigma_{k}(\theta_{k,0}))_{j_3j_4}\\
    &\quad-\frac{1}{nh_n}\sum_{i=1}^n\left\{h_n(\Sigma_{k}(\theta_{k,0}))_{j_3j_4}+h_n^2\{R(1,f_{k,t_{i-1}^n})\right.\\
    &\quad \left.\qquad \qquad \qquad \qquad +R(1,f_{k,t_{i-1}^n})R(1,e_{t_{i-1}^n})+R(1,e_{t_{i-1}^n})\}\right\}(\Sigma_{k}(\theta_{k,0}))_{j_1j_2}\\
    &\quad+(\Sigma_{k}(\theta_{k,0}))_{j_1j_2}(\Sigma_{k}(\theta_{k,0}))_{j_3j_4}\\
    &=(\Sigma_{k}(\theta_{k,0}))_{j_1j_3}(\Sigma_{k}(\theta_{k,0}))_{j_2j_4}+(\Sigma_{k}(\theta_{k,0}))_{j_1j_4}(\Sigma_{k}(\theta_{k,0}))_{j_2j_3}\\
    &\quad+h_n\frac{1}{n}\sum_{i=1}^nR(1,f_{k,t_{i-1}^n})
    +h_n\frac{1}{n}\sum_{i=1}^nR(1,f_{k,t_{i-1}^n})R(1,e_{t_{i-1}^n})
    +h_n\frac{1}{n}\sum_{i=1}^nR(1,e_{t_{i-1}^n})\\
    &\stackrel{P_{\theta_{k,0}}}{\to}(\Sigma_{k}(\theta_{k,0}))_{j_1j_3}(\Sigma_{k}(\theta_{k,0}))_{j_2j_4}+(\Sigma_{k}(\theta_{k,0}))_{j_1j_4}(\Sigma_{k}(\theta_{k,0}))_{j_2j_3}(=\Gamma_{k}(\theta_{k,0})_{(j_1,j_2),(j_3,j_4)})\\
    &\qquad\qquad\qquad\qquad\qquad\qquad\qquad\qquad\qquad\qquad\qquad\qquad\qquad\qquad\quad(j_1,j_2,j_3,j_4=1,\cdots p).
\end{align*}
The second term of (\ref{th1-8}) is 
\begin{align*}
    &\quad \sum_{i=1}^n\E_{\theta_{k,0}}\left[\left\{\frac{1}{\sqrt{n}h_n}(X_{t_{i}^n}^{(j_1)}-X_{t_{i-1}^n}^{(j_1)})(X_{t_{i}^n}^{(j_2)}-X_{t_{i-1}^n}^{(j_2)})-\frac{1}{\sqrt{n}}(\Sigma_{k}(\theta_{k,0}))_{j_1j_2}\right\}| \mathscr{F}^{n}_{i-1}\right]\\
    &\qquad \qquad \qquad \times \E_{\theta_{k,0}}\left[\left\{\frac{1}{\sqrt{n}h_n} (X_{t_{i}^n}^{(j_3)}-X_{t_{i-1}^n}^{(j_3)})(X_{t_{i}^n}^{(j_4)}-X_{t_{i-1}^n}^{(j_4)})-\frac{1}{\sqrt{n}}(\Sigma_{k}(\theta_{k,0}))_{j_3j_4}\right\}| \mathscr{F}^{n}_{i-1}\right]\\
    &=\sum_{i=1}^n\left\{\frac{1}{\sqrt{n}}R(h_n,f_{k,t_{i-1}^n})+\frac{1}{\sqrt{n}}R(1,f_{k,t_{i-1}^n})R(h_n,e_{t_{i-1}^n})
    +\frac{1}{\sqrt{n}}R(h_n,e_{t_{i-1}^n})
    \right\}\\
    &\qquad \qquad \qquad \times\left\{\frac{1}{\sqrt{n}}R(h_n,f_{k,t_{i-1}^n})+\frac{1}{\sqrt{n}}R(1,f_{k,t_{i-1}^n})R(h_n,e_{t_{i-1}^n})
    +\frac{1}{\sqrt{n}}R(h_n,e_{t_{i-1}^n})
    \right\}\\
    &=h_n^2\frac{1}{n}\sum_{i=1}^nR(1,f_{k,t_{i-1}^n})
    +h_n^2\frac{1}{n}\sum_{i=1}^nR(1,f_{k,t_{i-1}^n})R(1,e_{t_{i-1}^n})+h_n^2\frac{1}{n}\sum_{i=1}^nR(1,e_{t_{i-1}^n})\stackrel{P_{\theta_{k,0}}}{\to}0\\
    &\qquad\qquad\qquad\qquad\qquad\qquad\qquad\qquad\qquad\qquad\qquad\qquad\qquad\qquad\qquad\quad(j_1,j_2,j_3,j_4=1,\cdots p),
\end{align*}
so that (\ref{th1-8}) is deduced from Slutsky's theorem, and we obtain (\ref{th1-6}). 
Finally, for the proof of (\ref{th1-7}), we note that
\begin{align*}
    0 &\leq \sum_{i=1}^n\E_{\theta_{k,0}}\left[|L_{i,n}|^4| \mathscr{F}^{n}_{i-1}\right]\\
    &=\sum_{i=1}^n\E_{\theta_{k,0}}\left[\left|\sum_{u=1}^{p^2}L_{i,n}^{(u)2}\right|^2| \mathscr{F}^{n}_{i-1}\right]
    \leq C_{1}\sum_{u=1}^{p^2}\sum_{i=1}^n\E_{\theta_{k,0}}\left[|L_{i,n}^{(u)}|^4|\mathscr{F}^{n}_{i-1}\right].
\end{align*}
Hence, if it is proved that
\begin{align}
    \sum_{i=1}^n\E_{\theta_{k,0}}\left[|L_{i,n}^{(u)}|^4|\mathscr{F}^{n}_{i-1}\right]\stackrel{P_{\theta_{k,0}}}{\to}0\quad (u=1,\cdots,p^2), 
    \label{th1-9}
\end{align}
then, (\ref{th1-7}) is obtained. In order to prove (\ref{th1-9}), it is sufficient to show that
\begin{align*}
    &\sum_{i=1}^n\E_{\theta_{k,0}}\left[\left|\frac{1}{\sqrt{n}h_n}(X^{(j_1)}_{t_{i}^n}-X^{(j_1)}_{t_{i-1}^n})(X^{(j_2)}_{t_{i}^n}-X^{(j_2)}_{t_{i-1}^n})-\frac{1}{\sqrt{n}}(\Sigma_{k}(\theta_{k,0}))_{j_1j_2}
    \right|^4|\mathscr{F}^{n}_{i-1}\right]\stackrel{P_{\theta_{k,0}}}{\to}0\\\
    &\qquad\qquad\qquad\qquad\qquad\qquad\qquad\qquad\qquad\qquad\qquad\qquad\qquad\qquad\qquad(j_1,j_2=1,\cdots p).
\end{align*}
Since it follows from Lemma 6 in Kessler \cite{kessler(1997)} that
\begin{align*}
    &0\leq \sum_{i=1}^n\E_{\theta_{k,0}}\left[\left|\frac{1}{\sqrt{n}h_n}(X^{(j_1)}_{t_{i}^n}-X^{(j_1)}_{t_{i-1}^n})(X^{(j_2)}_{t_{i}^n}-X^{(j_2)}_{t_{i-1}^n})-\frac{1}{\sqrt{n}}(\Sigma_{k}(\theta_{k,0}))_{j_1j_2}
    \right|^4|\mathscr{F}^{n}_{i-1}\right]\\
    &\leq C_2\left\{\frac{1}{n^2h_n^4}\sum_{i=1}^n\E_{\theta_{k,0}}\left[\left|(X^{(j_1)}_{t_{i}^n}-X^{(j_1)}_{t_{i-1}^n})(X^{(j_2)}_{t_{i}^n}-X^{(j_2)}_{t_{i-1}^n})
    \right|^4|\mathscr{F}^{n}_{i-1}\right]
    +\frac{1}{n^2}(\Sigma_{k}(\theta_{k,0}))_{j_1j_2}^4\right\}\\
    &=\frac{C_2}{n^2h_n^4}\sum_{i=1}^n\E_{\theta_{k,0}}\left[\left|\{(\Lambda_{k,0}[j_1,\cdot](f_{k,t_{i}^n}-f_{k,t_{i-1}^n})
    +e^{(j_1)}_{t_{i}^n}-e^{(j_1)}_{t_{i-1}^n}\}\right.\right.\\
    &\left.\left.\qquad\qquad\qquad\qquad\qquad\times\{\Lambda_{k,0}[j_2,\cdot](f_{k,t_{i}^n}-f_{k,t_{i-1}^n})+e^{(j_2)}_{t_{i}^n}-e^{(j_2)}_{t_{i-1}^n}\}\right|^4|\mathscr{F}^{n}_{i-1}\right]+\frac{C_2}{n^2}(\Sigma_{k}(\theta_{k,0}))_{j_1j_2}^4\\
    &\leq \frac{C_3}{n^2h_n^4}\sum_{i=1}^n\E_{\theta_{k,0}}\left[\left|\Lambda_{k,0}[j_1,\cdot](f_{k,t_{i}^n}-f_{k,t_{i-1}^n})\Lambda_{k,0}[j_2,\cdot](f_{k,t_{i}^n}-f_{k,t_{i-1}^n})\right|^4|\mathscr{F}^{n}_{i-1}\right]\\
    &\quad +\frac{C_3}{n^2h_n^4}\sum_{i=1}^n\E_{\theta_{k,0}}\left[\left|\Lambda_{k,0}[j_1,\cdot](f_{k,t_{i}^n}-f_{k,t_{i-1}^n})(e^{(j_2)}_{t_{i}^n}-e^{(j_2)}_{t_{i-1}^n})\right|^4|\mathscr{F}^{n}_{i-1}\right]\\
    &\quad +\frac{C_3}{n^2h_n^4}\sum_{i=1}^n\E_{\theta_{k,0}}\left[\left|\Lambda_{k,0}[j_2,\cdot](f_{k,t_{i}^n}-f_{k,t_{i-1}^n})(e^{(j_1)}_{t_{i}^n}-e^{(j_1)}_{t_{i-1}^n})\right|^4|\mathscr{F}^{n}_{i-1}\right]\\
    &\quad +\frac{C_3}{n^2h_n^4}\sum_{i=1}^n\E_{\theta_{k,0}}\left[\left|(e^{(j_1)}_{t_{i}^n}-e^{(j_1)}_{t_{i-1}^n})(e^{(j_2)}_{t_{i}^n}-e^{(j_2)}_{t_{i-1}^n})\right|^4|\mathscr{F}^{n}_{i-1}\right]+\frac{C_2}{n^2}(\Sigma_{k}(\theta_{k,0}))_{j_1j_2}^4\\
    &\leq \frac{C_3}{n^2h_n^4}\sum_{i=1}^n\E_{\theta_{k,0}}\left[\left|\Lambda_{k,0}[j_1,\cdot](f_{k,t_{i}^n}-f_{k,t_{i-1}^n})\right|^8|\mathscr{F}^{n}_{i-1}\right]^\frac{1}{2}\E_{\theta_{k,0}}\left[\left|\Lambda_{k,0}[j_2,\cdot](f_{k,t_{i}^n}-f_{k,t_{i-1}^n})\right|^8|\mathscr{F}^{n}_{i-1}\right]^{\frac{1}{2}}\\
    &\quad +\frac{C_3}{n^2h_n^4}\sum_{i=1}^n\E_{\theta_{k,0}}\left[\left|\Lambda_{k,0}[j_1,\cdot](f_{k,t_{i}^n}-f_{k,t_{i-1}^n})\right|^8|\mathscr{F}^{n}_{i-1}\right]^{\frac{1}{2}}\E_{\theta_{k,0}}\left[\left|e^{(j_2)}_{t_{i}^n}-e^{(j_2)}_{t_{i-1}^n}\right|^8|\mathscr{F}^{n}_{i-1}\right]^{\frac{1}{2}}\\
    &\quad +\frac{C_3}{n^2h_n^4}\sum_{i=1}^n\E_{\theta_{k,0}}\left[\left|\Lambda_{k,0}[j_2,\cdot](f_{k,t_{i}^n}-f_{k,t_{i-1}^n})\right|^8|\mathscr{F}^{n}_{i-1}\right]^{\frac{1}{2}}\E_{\theta_{k,0}}\left[\left|e^{(j_1)}_{t_{i}^n}-e^{(j_1)}_{t_{i-1}^n}\right|^8|\mathscr{F}^{n}_{i-1}\right]^{\frac{1}{2}}\\
    &\quad +\frac{C_3}{n^2h_n^4}\sum_{i=1}^n\E_{\theta_{k,0}}\left[\left|e^{(j_1)}_{t_{i}^n}-e^{(j_1)}_{t_{i-1}^n}\right|^8|\mathscr{F}^{n}_{i-1}\right]^{\frac{1}{2}}\E_{\theta_{k,0}}\left[\left|e^{(j_2)}_{t_{i}^n}-e^{(j_2)}_{t_{i-1}^n}\right|^8|\mathscr{F}^{n}_{i-1}\right]^{\frac{1}{2}}+\frac{C_2}{n^2}(\Sigma_{k}(\theta_{k,0}))_{j_1j_2}^4\\
    &\leq \frac{C_3}{n^2h_n^4}\sum_{i=1}^n\E_{\theta_{k,0}}\left[\left|\Lambda_{k,0}[j_1,\cdot]\right|^8\left|f_{k,t_{i}^n}-f_{k,t_{i-1}^n}\right|^8|\mathscr{F}^{n}_{i-1}\right]^{\frac{1}{2}}\\
    &\qquad\qquad\qquad\qquad\qquad\qquad\qquad\qquad\qquad\qquad\times\E_{\theta_{k,0}}\left[\left|\Lambda_{k,0}[j_2,\cdot]\right|^8\left|f_{k,t_{i}^n}-f_{k,t_{i-1}^n}\right|^8|\mathscr{F}^{n}_{i-1}\right]^{\frac{1}{2}}\\
    &\quad +\frac{C_3}{n^2h_n^4}\sum_{i=1}^n\E_{\theta_{k,0}}\left[\left|\Lambda_{k,0}[j_1,\cdot]\right|^8\left|f_{k,t_{i}^n}-f_{k,t_{i-1}^n}\right|^8|\mathscr{F}^{n}_{i-1}\right]^{\frac{1}{2}}\E_{\theta_{k,0}}\left[\left|e^{(j_2)}_{t_{i}^n}-e^{(j_2)}_{t_{i-1}^n}\right|^8|\mathscr{F}^{n}_{i-1}\right]^{\frac{1}{2}}\\
    &\quad +\frac{C_3}{n^2h_n^4}\sum_{i=1}^n\E_{\theta_{k,0}}\left[\left|\Lambda_{k,0}[j_2,\cdot]\right|^8\left|f_{k,t_{i}^n}-f_{k,t_{i-1}^n}\right|^8|\mathscr{F}^{n}_{i-1}\right]^{\frac{1}{2}}\E_{\theta_{k,0}}\left[\left|e^{(j_1)}_{t_{i}^n}-e^{(j_1)}_{t_{i-1}^n}\right|^8|\mathscr{F}^{n}_{i-1}\right]^{\frac{1}{2}}\\
    &\quad +\frac{C_3}{n^2h_n^4}\sum_{i=1}^n\E_{\theta_{k,0}}\left[\left|e^{(j_1)}_{t_{i}^n}-e^{(j_1)}_{t_{i-1}^n}\right|^8|\mathscr{F}^{n}_{i-1}\right]^{\frac{1}{2}}\E_{\theta_{k,0}}\left[\left|e^{(j_2)}_{t_{i}^n}-e^{(j_2)}_{t_{i-1}^n}\right|^8|\mathscr{F}^{n}_{i-1}\right]^{\frac{1}{2}}+\frac{C_2}{n^2}(\Sigma_{k}(\theta_{k,0}))_{j_1j_2}^4\\    
    &\leq \frac{C_3}{n^2h_n^4}\sum_{i=1}^nR(h_n^4,f_{k,t_{i-1}^n})^{\frac{1}{2}}R(h_n^4,f_{k,t_{i-1}^n})^{\frac{1}{2}}+\frac{C_3}{n^2h_n^4}\sum_{i=1}^nR(h_n^4,f_{k,t_{i-1}^n})^{\frac{1}{2}}R(h_n^4,e_{t_{i-1}^n})^{\frac{1}{2}}\\
    &\quad +\frac{C_3}{n^2h_n^4}\sum_{i=1}^nR(h_n^4,e_{t_{i-1}^n})^{\frac{1}{2}}R(h_n^4,e_{t_{i-1}^n})^{\frac{1}{2}}+\frac{C_2}{n^2}(\Sigma_{k}(\theta_{k,0}))_{j_1j_2}^4\\
    &\leq \frac{C_3}{n}\frac{1}{n}\sum_{i=1}^nR(1,f_{k,t_{i-1}^n})+\frac{C_3}{n}\frac{1}{n}\sum_{i=1}^nR(1,f_{k,t_{i-1}^n})R(1,e_{t_{i-1}^n})\\
    &\quad+\frac{C_3}{n}\frac{1}{n}\sum_{i=1}^nR(1,e_{t_{i-1}^n})+\frac{C_2}{n^2}(\Sigma_{k}(\theta_{k,0}))_{j_1j_2}^4
    \stackrel{P_{\theta_{k,0}}}{\to}0\quad(j_1,j_2=1,\cdots,p),
\end{align*}
(\ref{th1-7}) is deduced, and we have (\ref{th1-1}). 
Let 
$f(x)=D_{p}^{+}x$. 
The continuous mapping theorem yields that
\begin{align*}
    f(\sqrt{n}(\vec(Q_{XX})-\vec(\Sigma_{k}(\theta_{k,0}))))
    &=\sqrt{n}(\vech(Q_{XX})-\vech(\Sigma_{k}(\theta_{k,0})))\\
    &\stackrel{d}{\to}f( N_{p^2}(0,\Gamma_{k}(\theta_{k,0})))
    =N_{\bar{p}}(0,D_{p}^{+}\Gamma_{k}(\theta_{k,0})D_{p}^{+\top}).
\end{align*}
Furthermore, since
\begin{align*}
    \left(D_{p}^{+}\Gamma_{k}(\theta_{k,0})D_{p}^{+\top}\right)_{ij}
    &=\sum_{\ell=1}^{p^{2}}\sum_{m=1}^{p^{2}}
    (D_{p}^{+})_{i\ell}(\Gamma_{k}(\theta_{k,0}))_{\ell m}(D_{p}^{+})_{jm}\\
    &=\sum_{j_1=1}^{p}\sum_{j_2=1}^{p}\sum_{j_3=1}^{p}\sum_{j_4=1}^{p}(D_{p}^{+})_{i,(j_1,j_2)}(\Gamma_{k}(\theta_{k,0}))_{(j_1,j_2),(j_3,j_4)}(D_{p}^{+})_{j,(j_3,j_4)}\\
    &=\sum_{j_1=1}^{p}\sum_{j_2=1}^{p}\sum_{j_3=1}^{p}\sum_{j_4=1}^{p}(D_{p}^{+})_{i,(j_1,j_2)}(\Sigma_{k}(\theta_{k,0}))_{j_1j_3}(\Sigma_{k}(\theta_{k,0}))_{j_2j_4}
    (D_{p}^{+})_{j,(j_3,j_4)}\\
    &\quad +\sum_{j_1=1}^{p}\sum_{j_2=1}^{p}\sum_{j_3=1}^{p}\sum_{j_4=1}^{p}(D_{p}^{+})_{i,(j_1,j_2)}(\Sigma_{k}(\theta_{k,0}))_{j_1j_4}(\Sigma_{k}(\theta_{k,0}))_{j_2j_3}
    (D_{p}^{+})_{j,(j_3,j_4)}\\
    &=\sum_{j_1=1}^{p}\sum_{j_2=1}^{p}\sum_{j_3=1}^{p}\sum_{j_4=1}^{p}(D_{p}^{+})_{i,(j_1,j_2)}(\Sigma_{k}(\theta_{k,0}))_{j_1j_3}(\Sigma_{k}(\theta_{k,0}))_{j_2j_4}
    (D_{p}^{+})_{j,(j_3,j_4)}\\
    &\quad +\sum_{j_1=1}^{p}\sum_{j_2=1}^{p}\sum_{j_4=1}^{p}\sum_{j_3=1}^{p}(D_{p}^{+})_{i,(j_1,j_2)}(\Sigma_{k}(\theta_{k,0}))_{j_1j_4}(\Sigma_{k}(\theta_{k,0}))_{j_2j_3}
    (D_{p}^{+})_{j,(j_4,j_3)}\\
    &=2\sum_{j_1=1}^{p}\sum_{j_2=1}^{p}\sum_{j_3=1}^{p}\sum_{j_4=1}^{p}(D_{p}^{+})_{i,(j_1,j_2)}(\Sigma_{k}(\theta_{k,0}))_{j_1j_3}(\Sigma_{k}(\theta_{k,0}))_{j_2j_4}
    (D_{p}^{+})_{j,(j_3,j_4)}\\
    &=2\sum_{j_1=1}^{p}\sum_{j_2=1}^{p}\sum_{j_3=1}^{p}\sum_{j_4=1}^{p}(D_{p}^{+})_{i,(j_1,j_2)}(\Sigma_{k}(\theta_{k,0})\otimes\Sigma_{k}(\theta_{k,0}))_{(j_1,j_2),(j_3,j_4)}(D_{p}^{+})_{j,(j_3,j_4)}\\
    &=2\sum_{\ell=1}^{p^2}\sum_{m=1}^{p^2}(D_{p}^{+})_{i\ell}(\Sigma_{k}(\theta_{k,0})\otimes\Sigma_{k}(\theta_{k,0}))_{\ell m}(D_{p}^{+})_{jm}\\
    &=(2D_{p}^{+}(\Sigma_{k}(\theta_{k,0})\otimes\Sigma_{k}(\theta_{k,0}))D_{p}^{+\top})_{ij}
    \quad (i,j=1,\cdots \bar{p}),
\end{align*}
it is proved that
\begin{align*}
    D_{p}^{+}\Gamma_{k}(\theta_{k,0})D_{p}^{+\top}=2D_{p}^{+}(\Sigma_{k}(\theta_{k,0})\otimes\Sigma_{k}(\theta_{k,0}))D_{p}^{+\top}(=W_{k}(\theta_{k,0})).
\end{align*} 
Therefore, (\ref{the1-2}) is obtained. 
\qed
\end{prf}
\begin{lemma} Under assumption [B2], 
\begin{align*}
W_{k}(\theta_{k})=D_{p}^{+}(\Sigma_{k}(\theta_{k})\otimes\Sigma_{k}(\theta_{k}))D_{p}^{+\top}>0.
\end{align*}
\end{lemma}
\begin{proof}
From $\Sigma_{ff,k}=S_{k}S_{k}^\top\geq0$, 
\begin{align*}
    \Lambda_{k}^\top\Sigma_{ff,k}\Lambda_{k}=(\Sigma_{ff,k}^{\frac{1}{2}}\Lambda_{k})^\top(\Sigma_{ff,k}^{\frac{1}{2}}\Lambda_{k})\geq0.
\end{align*} Furthermore, 
\begin{align*}
    \Sigma_{ee}=\Diag(\sigma_1^2,\cdots,\sigma_p^2)^{\top}>0.
\end{align*}
Hence, 
\begin{align*}
    \Sigma_{k}(\theta_{k})=\Lambda_{k}^\top\Sigma_{ff,k}\Lambda_{k}+\Sigma_{ee}>0.
\end{align*}
Since
$\Sigma_{k}(\theta_{k})\otimes\Sigma_{k}(\theta_{k})>0$, 
it follows that for all $x\in\mathbb{R}^{p^{2}}(\neq0)$, 
\begin{align*}
    x^{\top}(\Sigma_{k}(\theta_{k})\otimes\Sigma_{k}(\theta_{k}))x>0.
\end{align*} 
Furthermore, for all $x\in\mathbb{R}^{\bar{p}}(\neq0)$, 
\begin{align*}
     D_{p}^{+\top}x=0 
    &\Leftrightarrow D_{p}(D_{p}^{\top}D_{p})^{-1}x=0\\
    &\Leftrightarrow
    D_{p}^{\top}D_{p}(D_{p}^{\top}D_{p})^{-1}x=D_{p}^{\top}0\\
    &\Leftrightarrow x=0.
\end{align*}
Therefore, for all $x\in\mathbb{R}^{\bar{p}}(\neq0)$, 
\begin{align*}
    x^\top D_{p}^{+}(\Sigma_{k}(\theta_{k})\otimes\Sigma_{k}(\theta_{k}))D_{p}^{+\top}x>0,
\end{align*}
and we obtain $W_{k}(\theta_{k})>0$. 
\end{proof}
\begin{lemma}Under assumptions [B2], [C2], 
\begin{align*}
    \det{\Delta_{k}^{\top}W_{k}(\theta_{k,0}) \Delta_{k}}\neq 0
\end{align*}
\end{lemma}
\begin{proof}
From Lemma 5, 
\begin{align*}
    \rank{\{\Delta_{k}^{\top}W_{k}(\theta_{k,0})\Delta_{k}\}}
    &=\rank{\{(W_{k}(\theta_{k,0})^{\frac{1}{2}}\Delta_{k})^{\top}W_{k}(\theta_{k,0})^{\frac{1}{2}}\Delta_{k}\}}\\
    &=\rank{\{W_{k}(\theta_{k,0})^{\frac{1}{2}}\Delta_{k}\}}\\
    &=\rank{\Delta_{k}}\\
    &=q_{k}.
\end{align*}
Therefore, $\Delta_{k}^\top W_{k}(\theta_{k,0})\Delta_{k}$ is nonsingular. 
\end{proof}
\begin{prf}
First, we will prove (\ref{the2-1}). 
Since one has the following decomposition
\begin{align*}
    &\quad F_{k,n}(Q_{XX},\Sigma_{k}({\theta_{k}}))\\
    &=(\vech(Q_{XX})-\vech(\Sigma_{k}(\theta_{k})))^{\top}W_{k}^{-1}(\theta_{k})(\vech(Q_{XX})-\vech(\Sigma_{k}(\theta_{k})))\\
    &=(\vech(Q_{XX})-\vech(\Sigma_{k}(\theta_{k,0})))^{\top}W_{k}^{-1}(\theta_{k})(\vech(Q_{XX})-\vech(\Sigma_{k}(\theta_{k,0})))\\
    &\quad +2(\vech(\Sigma_{k}(\theta_{k,0}))-\vech(\Sigma_{k}(\theta_{k})))^{\top}W_{k}^{-1}(\theta_{k})(\vech(Q_{XX})-\vech(\Sigma_{k}(\theta_{k,0})))\\
    &\quad +(\vech(\Sigma_{k}(\theta_{k,0}))-\vech(\Sigma_{k}(\theta_{k})))^{\top}W_{k}^{-1}(\theta_{k})(\vech(\Sigma_{k}(\theta_{k,0}))-\vech(\Sigma_{k}(\theta_{k}))),
\end{align*}
it is shown from Slutsky's theorem that
\begin{align*}
    0&\leq \sup_{\theta_{k}\in\Theta_{k}}\left|F_{k,n}(Q_{XX},\Sigma_{k}({\theta_{k}}))
    -U_{k}(\theta_{k})\right|\\
    &\leq\sup_{\theta_{k}\in\Theta_{k}}\left|(\vech(Q_{XX})-\vech(\Sigma_{k}(\theta_{k,0})))^{\top}W_{k}^{-1}(\theta_{k})(\vech(Q_{XX})-\vech(\Sigma_{k}(\theta_{k,0})))\right|\\
    &\quad +2\sup_{\theta_{k}\in\Theta_{k}}\left|(\vech(\Sigma_{k}(\theta_{k,0}))-\vech(\Sigma_{k}(\theta_{k})))^{\top}W_{k}^{-1}(\theta_{k})(\vech(Q_{XX})-\vech(\Sigma_{k}(\theta_{k,0})))\right|\\    
    &\leq|\vech(Q_{XX})-\vech(\Sigma_{k}(\theta_{k,0}))|\sup_{\theta_{k}\in\Theta_{k}}\|W_{k}^{-1}(\theta_{k})\||\vech(Q_{XX})-\vech(\Sigma_{k}(\theta_{k,0}))|\\
    &\quad +2\sup_{\theta_{k}\in\Theta_{k}}\left|\vech(\Sigma_{k}(\theta_{k,0}))-\vech(\Sigma_{k}(\theta_{k}))\right|\|W_{k}^{-1}(\theta_{k})\|\left|\vech(Q_{XX})-\vech(\Sigma_{k}(\theta_{k,0}))\right|\stackrel{P_{\theta_{k,0}}}{\to}0,
\end{align*}
which implies
\begin{align}
    F_{k,n}(Q_{XX},\Sigma_{k}({\theta_{k}}))\stackrel{P_{\theta_{k,0}}}{\to}U_{k}(\theta_{k})\quad \mbox{uniformly in $\theta_{k}$}.
    \label{th2-1}
\end{align}
Since $W_{k}(\theta_{k})$ is a positive definite matrix from Lemma 5, 
we obtain
\begin{align*}
    U_{k}(\theta_{k})=0 \Leftrightarrow \vech(\Sigma_{k}(\theta_{k,0}))-\vech(\Sigma_{k}(\theta_{k}))=0.
\end{align*}
Furthermore, 
noting that
\begin{align*}
    \vech(\Sigma_{k}(\theta_{k,0}))-\vech(\Sigma_{k}(\theta_{k}))=0
    \Leftrightarrow \Sigma_{k}(\theta_{k,0})=\Sigma_{k}(\theta_{k})
    \Leftrightarrow\theta_{k,0}=\theta_{k}
\end{align*}
from the assumption [C1], 
we have 
\begin{align*}
    \theta_{k}\neq \theta_{k,0}\Rightarrow U_{k}(\theta_{k})>0(=U_{k}(\theta_{k,0})).
\end{align*}
Hence, 
\begin{align}
    \forall\varepsilon>0\ , \ \exists\delta>0\ s.t.\ |\hat{\theta}_{k,n}-\theta_{k,0}|>\varepsilon
    \Rightarrow U_{k}(\hat{\theta}_{k,n})-U_{k}(\theta_{k,0})>\delta.
    \label{th2-2}
\end{align}
Therefore, from (\ref{th2-1}) and (\ref{th2-2}), 
\begin{align*}
    0&\leq \PP_{\theta_{k,0}}(|\hat{\theta}_{k,n}-\theta_{k,0}|>\varepsilon)\\
    &\leq  \PP_{\theta_{k,0}}\left(U_{k}(\hat{\theta}_{k,n})-U_{k}(\theta_{k,0})>\delta\right)\\
    &\leq \PP_{\theta_{k,0}}\left(U_{k}(\hat{\theta}_{k,n})-F_{k,n}(Q_{XX},\Sigma_{k}({\hat{\theta}_{k,n}}))\right.\\
    &\qquad\qquad\left.+F_{k,n}(Q_{XX},\Sigma_{k}({\hat{\theta}_{k,n}}))-F_{k,n}(Q_{XX},\Sigma_{k}({\theta_{k,0}}))+F_{k,n}(Q_{XX},\Sigma_{k}({\theta_{k,0}}))-U_{k}(\theta_{k,0})
    >\delta\right)\\
    &\leq \PP_{\theta_{k,0}}\left(U_{k}(\hat{\theta}_{k,n})-F_{k,n}(Q_{XX},\Sigma_{k}({\hat{\theta}_{k,n}}))>\frac{\delta}{3}\right)\\
    &\quad +\PP_{\theta_{k,0}}\left(F_{k,n}(Q_{XX},\Sigma_{k}({\hat{\theta}_{k,n}}))-F_{k,n}(Q_{XX},\Sigma_{k}({\theta_{k,0}}))>\frac{\delta}{3}\right)\\
    &\quad +\PP_{\theta_{k,0}}\left(F_{k,n}(Q_{XX},\Sigma_{k}({\theta_{k,0}}))-U_{k}(\theta_{k,0})>\frac{\delta}{3}\right)\\
    &\leq  2\PP_{\theta_{k,0}}\left(\sup_{\theta_{k}\in\Theta_{k}}\left|F_{k,n}\left(Q_{XX},\Sigma_{k}({\theta_{k}})\right)-U_{k}(\theta_{k})\right|>\frac{\delta}{3}\right)+0\stackrel{}{\to}0\quad (n\stackrel{}{\to}\infty),
\end{align*}
so that we obtain (\ref{the2-1}). Next, we will prove (\ref{the2-2}). Using a Taylor expansion of \begin{align*}\partial_{\theta_{k}}F_{k,n}(Q_{XX},\Sigma_{k}({\hat{\theta}_{k,n}}))
\end{align*}
around $\theta_{k,0}$, we have
\begin{align*}
\partial_{\theta_{k}}F_{k,n}(Q_{XX},\Sigma_{k}(\hat{\theta}_{k,n}))&=\partial_{\theta_{k}}F_{k,n}(Q_{XX},\Sigma_{k}({\theta_{k,0}}))\\
&\qquad\quad+\int_{0}^{1}\partial^2_{\theta_{k}}F_{k,n}(Q_{XX},\Sigma_{k}(\theta_{k,0}+\lambda(\hat{\theta}_{k,n}-\theta_{k,0})))d\lambda(\hat{\theta}_{k,n}-\theta_{k,0})
\end{align*}
By the definition of $\hat{\theta}_{k,n}$, 
\begin{align*}
    \partial_{\theta_{k}}F_{k,n}(Q_{XX},\Sigma_{k}({\hat{\theta}_{k,n}}))=0,
\end{align*}
and 
\begin{align} -\sqrt{n}\partial_{\theta_{k}}F_{k,n}(Q_{XX},\Sigma_{k}({\theta_{k,0}}))=\int_{0}^{1}\partial^2_{\theta_{k}}F_{k,n}(Q_{XX},\Sigma_{k}({\theta_{k,0}}+\lambda(\hat{\theta}_{k,n}-\theta_{k,0})))d\lambda\sqrt{n}(\hat{\theta}_{k,n}-\theta_{k,0}).
\label{th2-3}
\end{align}
First, we consider the left side of (\ref{th2-3}). 
Note that
\begin{align*}
    &\quad \sqrt{n}\partial_{\theta_{k}^{(i)}}F_{k,n}(Q_{XX},\Sigma_{k}({\theta_{k}}))\\
    &=-2\sqrt{n}\{\partial_{\theta_{k}^{(i)}}\vech(\Sigma_{k}(\theta_{k}))\}^{\top}W_{k}^{-1}(\theta_{k})(\vech(Q_{XX})-\vech(\Sigma_{k}(\theta_{k})))\\
    &\quad +\sqrt{n}\left(\vech(Q_{XX})-\vech(\Sigma_{k}(\theta_{k}))\right)^{\top}\{\partial_{\theta_{k}^{(i)}}W_{k}^{-1}(\theta_{k})\}\left(\vech(Q_{XX})-\vech(\Sigma_{k}(\theta_{k}))\right)\\
    &\qquad\qquad\qquad\qquad\qquad\qquad\qquad\qquad\qquad\qquad\qquad\qquad\qquad\qquad\qquad\qquad(i=1,\cdots,q_{k}).
\end{align*}
Since it is shown that
\begin{align*}
    &\quad \sqrt{n}\partial_{\theta_{k}^{(i)}}F_{k,n}(Q_{XX},\Sigma_{k}({\theta_{k,0}}))\\
    &=-2\{\partial_{\theta_{k}^{(i)}}\vech(\Sigma_{k}(\theta_{k,0}))\}^{\top}W_{k}^{-1}(\theta_{k,0})\sqrt{n}(\vech(Q_{XX})-\vech(\Sigma_{k}(\theta_{k,0})))\\
    &\quad +\sqrt{n}\left(\vech(Q_{XX})-\vech(\Sigma_{k}(\theta_{k,0}))\right)^{\top}\{\partial_{\theta_{k}^{(i)}}W_{k}^{-1}(\theta_{k,0})\}\left(\vech(Q_{XX})-\vech(\Sigma_{k}(\theta_{k,0}))\right)\\
    &=-2\{\partial_{\theta_{k}^{(i)}}\vech(\Sigma_{k}(\theta_{k,0}))\}^{\top}W_{k}^{-1}(\theta_{k,0})\sqrt{n}(\vech(Q_{XX})-\vech(\Sigma_{k}(\theta_{k,0})))+o_{p}(1)\\
    &\qquad\qquad\qquad\qquad\qquad\qquad\qquad\qquad\qquad\qquad\qquad\qquad\qquad\qquad\qquad\qquad\qquad(i=1,\cdots,q_{k}),
\end{align*}
we see from Theorem 1 that
\begin{align}
    \begin{split}
     -\sqrt{n}\partial_{\theta_{k}}F_{k,n}(Q_{XX},\Sigma_{k}({\theta_{k,0}}))
    &=2\Delta_{k}^{\top}\label{th2-4} W_{k}^{-1}(\theta_{k,0})\sqrt{n}(\vech(Q_{XX})-\vech(\Sigma_{k}(\theta_{k,0})))+o_{p}(1)\\
    &\stackrel{d}{\to}2\Delta_{k}^{\top} W_{k}^{-1}(\theta_{k,0})N_{q_{k}}(0,W_{k}(\theta_{k,0}))=
    N_{q_{k}}\left(0,4\Delta_{k}^{\top} W_{k}^{-1}(\theta_{k,0})\Delta_{k}\right).
    \end{split}
\end{align}
Next, we consider the right side of (\ref{th2-3}). 
It follows that
\begin{align*}
    &\quad\ \partial_{\theta_{k}^{(i)}}\partial_{\theta_{k}^{(j)}}F_{k,n}(Q_{XX},\Sigma_{k}({\theta_{k}}))\\
    &=-2\{\partial_{\theta_{k}^{(i)}}\partial_{\theta_{k}^{(j)}}\vech(\Sigma_{k}(\theta_{k}))\}^{\top}W_{k}^{-1}(\theta_{k})(\vech(Q_{XX})-\vech(\Sigma_{k}(\theta_{k})))\\
    &\quad-2\{\partial_{\theta_{k}^{(i)}}\vech(\Sigma_{k}(\theta_{k}))\}^{\top}\{\partial_{\theta_{k}^{(j)}}W_{k}^{-1}(\theta_{k})\}(\vech(Q_{XX})-\vech(\Sigma_{k}(\theta_{k})))\\
    &\quad+2\{\partial_{\theta_{k}^{(i)}}\vech(\Sigma_{k}(\theta_{k}))\}^{\top}W_{k}^{-1}(\theta_{k})\{\partial_{\theta_{k}^{(j)}}\vech(\Sigma_{k}(\theta_{k}))\}\\
    &\quad-2\{\partial_{\theta_{k}^{(j)}}\vech(\Sigma_{k}(\theta_{k}))\}^{\top}\{\partial_{\theta_{k}^{(i)}}W_{k}^{-1}(\theta_{k})\}\left(\vech(Q_{XX})-\vech(\Sigma_{k}(\theta_{k}))\right)\\
    &\quad+\left(\vech(Q_{XX})-\vech(\Sigma_{k}(\theta_{k}))\right)^{\top}\{\partial_{\theta_{k}^{(i)}}\partial_{\theta_{k}^{(j)}}W_{k}^{-1}(\theta_{k})\}\left(\vech(Q_{XX})-\vech(\Sigma_{k}(\theta_{k}))\right)\\
    &=-2\{\partial_{\theta_{k}^{(i)}}\partial_{\theta_{k}^{(j)}}\vech(\Sigma_{k}(\theta_{k}))\}^{\top}W_{k}^{-1}(\theta_{k})(\vech(Q_{XX})-\vech(\Sigma_{k}(\theta_{k,0})))\\
    &\quad -2\{\partial_{\theta_{k}^{(i)}}\partial_{\theta_{k}^{(j)}}\vech(\Sigma_{k}(\theta_{k}))\}^{\top}W_{k}^{-1}(\theta_{k})(\vech(\Sigma_{k}(\theta_{k,0}))-\vech(\Sigma_{k}(\theta_{k})))\\
    &\quad -2\{\partial_{\theta_{k}^{(i)}}\vech(\Sigma_{k}(\theta_{k}))\}^{\top}\{\partial_{\theta_{k}^{(j)}}W_{k}^{-1}(\theta_{k})\}(\vech(Q_{XX})-\vech(\Sigma_{k}(\theta_{k,0})))\\
    &\quad-2\{\partial_{\theta_{k}^{(i)}}\vech(\Sigma_{k}(\theta_{k}))\}^{\top}\{\partial_{\theta_{k}^{(j)}}W_{k}^{-1}(\theta_{k})\}(\vech(\Sigma_{k}(\theta_{k,0}))-\vech(\Sigma_{k}(\theta_{k})))\\
    &\quad +2\{\partial_{\theta_{k}^{(i)}}\vech(\Sigma_{k}(\theta_{k}))\}^{\top}W_{k}^{-1}(\theta_{k})\{\partial_{\theta_{k}^{(j)}}\vech(\Sigma_{k}(\theta_{k}))\}\\
    &\quad-2\{\partial_{\theta_{k}^{(j)}}\vech(\Sigma_{k}(\theta_{k}))\}^{\top}\{\partial_{\theta_{k}^{(i)}}W_{k}^{-1}(\theta_{k})\}\left(\vech(Q_{XX})-\vech(\Sigma_{k}(\theta_{k,0}))\right)\\
    &\quad-2\{\partial_{\theta_{k}^{(j)}}\vech(\Sigma_{k}(\theta_{k}))\}^{\top}\{\partial_{\theta_{k}^{(i)}}W_{k}^{-1}(\theta_{k})\}(\vech(\Sigma_{k}(\theta_{k,0}))-\vech(\Sigma_{k}(\theta_{k})))\\
    &\quad+\left(\vech(Q_{XX})-\vech(\Sigma_{k}(\theta_{k,0}))\right)^{\top}\{\partial_{\theta_{k}^{(i)}}\partial_{\theta_{k}^{(j)}}W_{k}^{-1}(\theta_{k})\}\left(\vech(Q_{XX})-\vech(\Sigma_{k}(\theta_{k,0}))\right)\\
    &\quad+2\left(\vech(Q_{XX})-\vech(\Sigma_{k}(\theta_{k,0}))\right)^{\top}\{\partial_{\theta_{k}^{(i)}}\partial_{\theta_{k}^{(j)}}W_{k}^{-1}(\theta_{k})\}\left(\vech(\Sigma_{k}(\theta_{k,0}))-\vech(\Sigma_{k}(\theta_{k}))\right)\\
    &\quad
    +\left(\vech(\Sigma_{k}(\theta_{k,0}))-\vech(\Sigma_{k}(\theta_{k}))\right)^{\top}\{\partial_{\theta_{k}^{(i)}}\partial_{\theta_{k}^{(j)}}W_{k}^{-1}(\theta_{k})\}\left(\vech(\Sigma_{k}(\theta_{k,0}))-\vech(\Sigma_{k}(\theta_{k}))\right)\\
    &\qquad\qquad\qquad\qquad\qquad\qquad\qquad\qquad\qquad\qquad\qquad\qquad\qquad\qquad\qquad\qquad\quad(i,j=1,\cdots,q_{k}).
\end{align*}
Setting
\begin{align*}
    &\quad\ \ U_{2,k,ij}(\theta_{k})\\
    &=-2\{\partial_{\theta_{k}^{(i)}}\partial_{\theta_{k}^{(j)}}\vech(\Sigma_{k}(\theta_{k}))\}^{\top}W_{k}^{-1}(\theta_{k})(\vech(\Sigma_{k}(\theta_{k,0}))-\vech(\Sigma_{k}(\theta_{k})))\\
    &\quad -2\{\partial_{\theta_{k}^{(i)}}\vech(\Sigma_{k}(\theta_{k}))\}^{\top}\{\partial_{\theta_{k}^{(j)}}W_{k}^{-1}(\theta_{k})\}(\vech(\Sigma_{k}(\theta_{k,0}))-\vech(\Sigma_{k}(\theta_{k})))\\
    &\quad +2\{\partial_{\theta_{k}^{(i)}}\vech(\Sigma_{k}(\theta_{k}))\}^{\top}W_{k}^{-1}(\theta_{k})\{\partial_{\theta_{k}^{(j)}}\vech(\Sigma_{k}(\theta_{k}))\}\\
    &\quad-2\{\partial_{\theta_{k}^{(j)}}\vech(\Sigma_{k}(\theta_{k}))\}^{\top}\{\partial_{\theta_{k}^{(i)}}W_{k}^{-1}(\theta_{k})\}(\vech(\Sigma_{k}(\theta_{k,0}))-\vech(\Sigma_{k}(\theta_{k})))\\
    &\quad+\left(\vech(\Sigma_{k}(\theta_{k,0}))-\vech(\Sigma_{k}(\theta_{k}))\right)^{\top}\{\partial_{\theta_{k}^{(i)}}\partial_{\theta_{k}^{(j)}}W_{k}^{-1}(\theta_{k})\}\left(\vech(\Sigma_{k}(\theta_{k,0}))-\vech(\Sigma_{k}(\theta_{k}))\right)\\
    &\qquad\qquad\qquad\qquad\qquad\qquad\qquad\qquad\qquad \qquad\qquad\qquad\qquad\qquad\qquad\qquad\quad(i,j=1,\cdots,q_{k}),
\end{align*}
we have
\begin{align*}
     0&\leq\sup_{\theta_{k}\in\Theta_{k}}| \partial_{\theta_{k}^{(i)}}\partial_{\theta_{k}^{(j)}}F_{k,n}(Q_{XX},\Sigma_{k}({\theta_{k}}))-U_{2,k,ij}(\theta_{k})|\\
     &\leq2\sup_{\theta_{k}\in\Theta_{k}}|\partial_{\theta_{k}^{(i)}}\partial_{\theta_{k}^{(j)}}\vech(\Sigma_{k}(\theta_{k}))|\|W_{k}^{-1}(\theta_{k})\||\vech(Q_{XX})-\vech(\Sigma_{k}(\theta_{k,0}))|\\
     &\quad +2\sup_{\theta_{k}\in\Theta_{k}}|\partial_{\theta_{k}^{(i)}}\vech(\Sigma_{k}(\theta_{k}))|\|\partial_{\theta_{k}^{(j)}}W_{k}^{-1}(\theta_{k})\||\vech(Q_{XX})-\vech(\Sigma_{k}(\theta_{k,0}))|\\
     &\quad +2\sup_{\theta_{k}\in\Theta_{k}}|\partial_{\theta_{k}^{(j)}}\vech(\Sigma_{k}(\theta_{k}))|\|\partial_{\theta_{k}^{(i)}}W_{k}^{-1}(\theta_{k})\||\vech(Q_{XX})-\vech(\Sigma_{k}(\theta_{k,0}))|\\
     &\quad+|\vech(Q_{XX})-\vech(\Sigma_{k}(\theta_{k,0}))|\sup_{\theta_{k}\in\Theta_{k}}\|\partial_{\theta_{k}^{(i)}}\partial_{\theta_{k}^{(j)}}W_{k}^{-1}(\theta_{k})\||\vech(Q_{XX})-\vech(\Sigma_{k}(\theta_{k,0}))|\\
    &\quad+2|\vech(Q_{XX})-\vech(\Sigma_{k}(\theta_{k,0}))|\sup_{\theta_{k}\in\Theta_{k}}\|\partial_{\theta_{k}^{(i)}}\partial_{\theta_{k}^{(j)}}W_{k}^{-1}(\theta_{k})\||\vech(\Sigma_{k}(\theta_{k,0})-\vech(\Sigma_{k}(\theta_{k}))|\\
     &\stackrel{P_{\theta_0}}{\to}0\quad (i,j=1,\cdots,q_{k}),
\end{align*}
and
\begin{align}
    \partial^2_{\theta_{k}}F_{k,n}(Q_{XX},\Sigma_{k}({\theta_{k}}))\stackrel{P_{\theta_{k,0}}}{\to}U_{2,k}(\theta_{k})\quad 
    \mbox{uniformly in $\theta_{k}$}, \label{th2-5}
\end{align}
where 
\begin{align*}
    U_{2,k}(\theta_{k})=(U_{2,k,ij}(\theta_{k}))_{i,j=1,\cdots,q_k}.
\end{align*} 
Furthermore, 
\begin{align*}
    U_{2,k,ij}(\theta_{k})&\rightarrow2\{\partial_{\theta_{k}^{(i)}}\vech(\Sigma_{k}(\theta_{k,0}))\}^{\top}W_{k}^{-1}(\theta_{k,0})\{\partial_{\theta_{k}^{(j)}}\vech(\Sigma_{k}(\theta_{k,0}))\}\quad(\theta_{k}\rightarrow\theta_{k,0})\\ &\qquad\qquad\qquad\qquad\qquad\qquad\qquad\qquad\qquad\qquad\qquad\qquad\quad(i,j=1,\cdots,q_{k}),
\end{align*}
and
\begin{align}
    U_{2,k}(\theta_{k})\rightarrow2\Delta_{k}^{\top}W_{k}^{-1}(\theta_{k,0})\Delta_{k}\quad(\theta_{k}\rightarrow\theta_{k,0}).
    \label{th2-6}
\end{align}
Let 
\begin{align*}
    A_n= \{|\hat{\theta}_{k,n}-\theta_{k,0}|\leq\rho_n\},
\end{align*}
where $\{\rho_n\}_{n\in\mathbb{N}}$ is the sequence such that $\rho_n\rightarrow0\ (n\rightarrow\infty)$. 
It follows from  $(\ref{th2-5})$ and $(\ref{th2-6})$ that for all $\varepsilon>0$, 
\begin{align*}
    0&\leq \PP_{\theta_{k,0}}\left(\left|\int_{0}^{1}\partial^2_{\theta_{k}}F_{k,n}(Q_{XX},\Sigma_{k}({\theta_{k,0}}+\lambda(\hat{\theta}_{k,n}-\theta_{k,0})))d\lambda-2\Delta_{k}^{\top}W_{k}^{-1}(\theta_{k,0})\Delta_{k}\right|>\varepsilon\right)\\
    &\leq \PP_{\theta_{k,0}}\left(\left\{\left|\int_{0}^{1}\partial^2_{\theta_{k}}F_{k,n}(Q_{XX},\Sigma_{k}({\theta_{k,0}}+\lambda(\hat{\theta}_{k,n}-\theta_{k,0})))d\lambda-2\Delta_{k}^{\top}W_{k}^{-1}(\theta_{k,0})\Delta_{k}\right|>\varepsilon\right\}\cap A_n\right)\\
    &\quad +\PP_{\theta_{k,0}}\left(\left\{\left|\int_{0}^{1}\partial^2_{\theta_{k}}F_{k,n}(Q_{XX},\Sigma_{k}({\theta_{k,0}}+\lambda(\hat{\theta}_{k,n}-\theta_{k,0})))d\lambda-2\Delta_{k}^{\top}W_{k}^{-1}(\theta_{k,0})\Delta_{k}\right|>\varepsilon\right\}\cap A_n^{c}\right)\\
    &\leq \PP_{\theta_{k,0}}\left(\sup_{|\theta_{k}-\theta_{k,0}|\leq\rho_n}|\partial^2_{\theta_{k}}F_{k,n}(Q_{XX},\Sigma_{k}(\theta_{k}))-2\Delta_{k}^{\top}W_{k}^{-1}(\theta_{k,0})\Delta_{k}|>\varepsilon\right)+\PP_{\theta_{k,0}}(A_n^{c})\\
    &\leq \PP_{\theta_{k,0}}\left(\sup_{|\theta_{k}-\theta_{k,0}|\leq\rho_n}|\partial^2_{\theta_{k}}F_{k,n}(Q_{XX},\Sigma_{k}(\theta_{k}))-U_{2,k}(\theta_{k})+U_2(\theta_{k})-2\Delta_{k}^{\top}W_{k}^{-1}(\theta_{k,0})\Delta_{k}|>\varepsilon\right)\\
    &\quad +\PP_{\theta_{k,0}}(A_n^{c})\\
    &\leq\PP_{\theta_{k,0}}\left(\sup_{|\theta_{k}-\theta_{k,0}|\leq\rho_n}|\partial^2_{\theta_{k}}F_{k,n}(Q_{XX},\Sigma_{k}(\theta_{k}))-U_{2,k}(\theta_{k})|>\frac{\varepsilon}{2}\right)\\
    &\quad +\PP_{\theta_{k,0}}\left(\sup_{|\theta_{k}-\theta_{k,0}|\leq\rho_n}|U_{2,k}(\theta_{k})-2\Delta_{k}^{\top}W_{k}^{-1}(\theta_{k,0})\Delta_{k}|>\frac{\varepsilon}{2}\right)+\PP_{\theta_{k,0}}(A_n^{c})\\
    &\leq \PP_{\theta_{k,0}}\left(\sup_{\theta_{k}\in\Theta_{k}}|\partial^2_{\theta_{k}}F_{k,n}(Q_{XX},\Sigma_{k}(\theta_{k}))-U_{2,k}(\theta_{k})|>\frac{\varepsilon}{2}\right)\\
    &\quad +\PP_{\theta_{k,0}}\left(\sup_{|\theta_{k}-\theta_{k,0}|\leq\rho_n}|U_{2,k}(\theta_{k})-2\Delta_{k}^{\top}W_{k}^{-1}(\theta_{k,0})\Delta_{k}|>\frac{\varepsilon}{2}\right)+\PP_{\theta_{k,0}}(A_n^{c})\rightarrow 0,
\end{align*}
so that we obtain
\begin{align}
    \int_{0}^{1}\partial^2_{\theta_{k}}F_{k,n}(Q_{XX},\Sigma_{k}({\theta_{k,0}}+\lambda(\hat{\theta}_{k,n}-\theta_{k,0})))d\lambda\stackrel{P_{\theta_{k,0}}}{\to}
    2\Delta_{k}^{\top}W_{k}^{-1}(\theta_{k,0})\Delta_{k}.
    \label{th2-7}
\end{align}
Therefore, Lemma 5, (\ref{th2-3}), (\ref{th2-4}) and (\ref{th2-7}) imply
\begin{align*}
    \sqrt{n}(\hat{\theta}_{k,n}-\theta_{k,0})&\stackrel{d}{\to}\{2\Delta_{k}^{\top}W_{k}^{-1}(\theta_{k,0})\Delta_{k}\}^{-1}N_{q_{k}}\left(0,4\Delta_{k}^{\top} W_{k}^{-1}(\theta_{k,0})\Delta_{k}\right)\\
    &=N_{q_{k}}\left(0,(\Delta_{k}^{\top} W_{k}^{-1}(\theta_{k,0})\Delta_{k})^{-1}\right).
\end{align*}
\qed
\end{prf}
\begin{lemma}
Let $X \sim N_{p}(0,I_{p})$. If $P\in\mathbb{R}^{p\times p}$ is a projection matrix of rank $r$, 
then
\begin{align*}
    X^{\top}P X\sim \chi^2_{r}.
\end{align*}
\end{lemma}


\begin{proof}
See Lemma 9.3 in Ferguson \cite{ferguson(1996)}.
\end{proof}
\begin{prf}
Using a Taylor expansion of $T_{k^*,n}=nF_{k^*,n}(Q_{XX},\Sigma_{k^*}(\hat{\theta}_{k^*,n}))$ around $\theta_{k^*,0}$, we have
\begin{align}
    \begin{split}
   T_{k^*,n}
    &=nF_{k^*,n}(Q_{XX},\Sigma_{k^*}(\theta_{k^*,0}))\\
    &\quad+n\partial_{\theta_{k^*}}F_{k^*,n}(Q_{XX},\Sigma_{k^*}(\theta_{k^*,0}))^{\top}(\hat{\theta}_{k^*,n}-\theta_{k^*,0})\\
    &\qquad+\frac{n}{2}(\hat{\theta}_{k^*,n}-\theta_{k^*,0})^{\top}\\
    &\qquad\quad\times\left\{\int_{0}^{1}(1-\lambda)\partial^2_{\theta_{k^*}}F_{k^*,n}(Q_{XX},\Sigma_{k^*}(\theta_{k^*,0}+\lambda(\hat{\theta}_{k^*,n}-\theta_{k^*,0})))d\lambda\right\}\\
    &\qquad\qquad\qquad\qquad\qquad\qquad\qquad\qquad\qquad\qquad\qquad\qquad\qquad\times(\hat{\theta}_{k^*,n}-\theta_{k^*,0}).
    \label{th3-1}
    \end{split}
\end{align}
Since it follows from (\ref{th2-3}), (\ref{th2-4}) and (\ref{th2-7}) that
\begin{align*}
    \sqrt{n}(\hat{\theta}_{k^*,n}-\theta_{k^*,0})
    &=(\Delta_{k^*}^{\top}W_{k^*}^{-1}(\theta_{k^*,0})\Delta_{k^*})^{-1}\Delta_{k^*}^{\top}W_{k^*}^{-1}(\theta_{k^*,0})\\
    &\qquad\qquad\qquad\qquad\qquad\times\sqrt{n}(\vech(Q_{XX})-\vech(\Sigma_{k^*}(\theta_{k^*,0})))+o_{p}(1)
\end{align*}
under $H_0$, 
the second term on the right side in (\ref{th3-1}) is
\begin{align*}
    &\quad n\partial_{\theta_{k^*}}F_{k^*,n}(Q_{XX},\Sigma_{k^*}(\theta_{k^*,0}))^{\top}(\hat{\theta}_{k^*,n}-\theta_{k^*,0})\\
    &=\left\{-2\Delta_{k^*}^{\top}W_{k^*}^{-1}(\theta_{k^*,0})\sqrt{n}(\vech(Q_{XX})-\vech(\Sigma_{k^*}(\theta_{k^*,0})))+o_{p}(1)\right\}^{\top}\sqrt{n}(\hat{\theta}_{k^*,n}-\theta_{k^*,0})\\
    &=-2\sqrt{n}(\vech(Q_{XX})-\vech(\Sigma_{k^*}(\theta_{k^*,0})))^{\top}W_{k^*}^{-1}(\theta_{k^*,0})\Delta_{k^*}\\
    &\qquad \times 
    \left\{(\Delta_{k^*}^{\top}W_{k^*}^{-1}(\theta_{k^*,0})\Delta_{k^*})^{-1}\Delta_{k^*}^{\top} W_{k^*}^{-1}(\theta_{k^*,0})\right.\\
    &\qquad\quad\left.\times\sqrt{n}(\vech(Q_{XX})-\vech(\Sigma_{k^*}(\theta_{k^*,0})))+o_{p}(1)\right\}+o_{p}(1)\\
    &=-2\sqrt{n}(\vech(Q_{XX})-\vech(\Sigma_{k^*}(\theta_{k^*,0})))^{\top}H_{k^*}(\theta_{k^*,0})\sqrt{n}(\vech(Q_{XX})-\vech(\Sigma_{k^*}(\theta_{k^*,0})))+o_p(1)
\end{align*}
under $H_0$,  where
\begin{align*}
    H_{k^*}(\theta_{k^*,0})= W_{k^*}^{-1}(\theta_{k^*,0})\Delta_{k^*}(\Delta_{k^*}^{\top}W_{k^*}^{-1}(\theta_{k^*,0})\Delta_{k^*})^{-1}\Delta_{k^*}^{\top} W_{k^*}^{-1}(\theta_{k^*,0}).
\end{align*}
Noting that for all $\varepsilon>0$, 
\begin{align*}
    0\leq&\PP\left(\left|\int_{0}^{1}(1-\lambda)\partial^2_{\theta_{k^*}}F_{k^*,n}(Q_{XX},\Sigma_{k^*}(\theta_{k^*,0}+\lambda(\hat{\theta}_{k^*,n}-\theta_{k^*,0})))d\lambda-2\Delta_{k^*}^{\top}W_{k^*}^{-1}(\theta_{k^*,0})\Delta_{k^*}\right|>\varepsilon\right)\\
    \begin{split}
    &\leq \PP\left(\left\{\left|\int_{0}^{1}(1-\lambda)\partial^2_{\theta_{k^*}}F_{k^*,n}(Q_{XX},\Sigma_{k^*}(\theta_{k^*,0}+\lambda(\hat{\theta}_{k^*,n}-\theta_{k^*,0})))d\lambda\right.\right.\right.\\
    &\qquad\qquad\qquad\qquad\qquad\qquad\qquad\qquad\qquad\qquad\qquad\qquad\left.\left.\left.-2\Delta_{k^*}^{\top}W_{k^*}^{-1}(\theta_{k^*,0})\Delta_{k^*}\right|>\varepsilon\right\}\cap A_n\right)
    \end{split}\\
    &+\PP\left(\left\{\left|\int_{0}^{1}(1-\lambda)\partial^2_{\theta_{k^*}}F_{k^*,n}(Q_{XX},\Sigma_{k^*}(\theta_{k^*,0}+\lambda(\hat{\theta}_{k^*,n}-\theta_{k^*,0})))d\lambda\right.\right.\right.\\
    &\qquad\qquad\qquad\qquad\qquad\qquad\qquad\qquad\qquad\qquad\qquad\qquad\left.\left.\left.-2\Delta_{k^*}^{\top}W_{k^*}^{-1}(\theta_{k^*,0})\Delta_{k^*}\right|>\varepsilon\right\}\cap A_n^{c}\right)\\
    &\leq \PP\left(\sup_{|\theta_{k^*}-\theta_{k^*,0}|\leq \rho_n}\left|\partial^2_{\theta_{k^*}}F_{k^*,n}(Q_{XX},\Sigma_{k^*}(\theta_{k^*}))-2\Delta_{k^*}^{\top}W_{k^*}^{-1}(\theta_{k^*,0})\Delta_{k^*}\right|>\varepsilon\right)+\PP(A_n^{c})\rightarrow0
\end{align*}
under $H_0$, 
we have
\begin{align*}
    \int_{0}^{1}(1-\lambda)\partial^2_{\theta_{k^*}}F_{k^*,n}(Q_{XX},\Sigma_{k^*}(\theta_{k^*,0}+\lambda(\hat{\theta}_{k^*,n}-\theta_{k^*,0})))d\lambda\stackrel{P}{\to}
    2\Delta_{k^*}^{\top}W_{k^*}^{-1}(\theta_{k^*,0})\Delta_{k^*}
\end{align*}
under $H_0$. Hence, the third term on the right side in (\ref{th3-1}) is
\begin{align*}
    &\quad\ \frac{n}{2}(\hat{\theta}_{k^*,n}-\theta_{k^*,0})^\top\left\{\int_{0}^{1}(1-\lambda)\partial^2_{\theta_{k^*}}F_{k^*,n}(Q_{XX},\Sigma_{k^*}(\theta_{k^*,0}+\lambda(\hat{\theta}_{k^*,n}-\theta_{k^*,0})))d\lambda\right\}(\hat{\theta}_{k^*,n}-\theta_{k^*,0})\\
    &=\frac{1}{2}\left\{(\Delta_{k^*}^{\top}W_{k^*}^{-1}(\theta_{k^*,0})\Delta_{k^*})^{-1}\Delta_{k^*}^{\top} W_{k^*}^{-1}(\theta_{k^*,0})\sqrt{n}(\vech(Q_{XX})-\vech(\Sigma_{k^*}(\theta_{k^*,0})))+o_{p}(1)\right\}^{\top}\\
    &\quad \times\left\{2\Delta_{k^*}^{\top}W_{k^*}^{-1}(\theta_{k^*,0})\Delta_{k^*}+o_p(1)\right\}\\
    &\quad \times \left\{(\Delta_{k^*}^{\top}W_{k^*}^{-1}(\theta_{k^*,0})\Delta_{k^*})^{-1}\Delta_{k^*}^{\top}W_{k^*}^{-1}(\theta_{k^*,0})\sqrt{n}(\vech(Q_{XX})-\vech(\Sigma_{k^*}(\theta_{k^*,0})))+o_{p}(1)\right\}\\
    &=\sqrt{n}(\vech(Q_{XX})-\vech(\Sigma_{k^*}(\theta_{k^*,0})))^{\top}H_{k^*}(\theta_{k^*,0})\sqrt{n}(\vech(Q_{XX})-\vech(\Sigma_{k^*}(\theta_{k^*,0})))+o_{p}(1)
\end{align*}
under $H_0$. 
Therefore, 
\begin{align*}
    T_{k^*,n}&=\sqrt{n}(\vech(Q_{XX})-\vech(\Sigma_{k^*}(\theta_{k^*,0})))^{\top}W_{k^*}^{-1}(\theta_{k^*,0})\sqrt{n}(\vech(Q_{XX})-\vech(\Sigma_{k^*}(\theta_{k^*,0})))\\
    &\quad -2\sqrt{n}(\vech(Q_{XX})-\vech(\Sigma_{k^*}(\theta_{k^*,0})))^{\top}H_{k^*}(\theta_{k^*,0})\\
    &\qquad\qquad\qquad\qquad\qquad\qquad\qquad\qquad\times\sqrt{n}(\vech(Q_{XX})-\vech(\Sigma_{k^*}(\theta_{k^*,0})))+o_{p}(1)\\
    &\quad +\sqrt{n}(\vech(Q_{XX})-\vech(\Sigma_{k^*}(\theta_{k^*,0})))^{\top}H_{k^*}(\theta_{k^*,0})\\
    &\qquad\qquad\qquad\qquad\qquad\qquad\qquad\qquad\times\sqrt{n}(\vech(Q_{XX})-\vech(\Sigma_{k^*}(\theta_{k^*,0})))+o_{p}(1)\\
    &=\sqrt{n}(\vech(Q_{XX})-\vech(\Sigma_{k^*}(\theta_{k^*,0})))^{\top}(W_{k^*}^{-1}(\theta_{k^*,0})-H_{k^*}(\theta_{k^*,0}))\\
    &\qquad\qquad\qquad\qquad\qquad\qquad\qquad\qquad\times\sqrt{n}(\vech(Q_{XX})-\vech(\Sigma_{k^*}(\theta_{k^*,0})))+o_{p}(1)
\end{align*}
under $H_0$. 
Setting
\begin{align*}
    \xi_{k^*,n}= W_{k^*}^{-\frac{1}{2}}(\theta_{k^*,0})\sqrt{n}(\vech(Q_{XX})-\vech(\Sigma_{k^*}(\theta_{k^*,0}))), 
\end{align*}
we see from Theorem 1 that
\begin{align*}
    \xi_{k^*,n}\stackrel{d}{\to}N_{\bar{p}}(0,I_{\bar{p}})
\end{align*}
under $H_0$.  
Let
\begin{align*}
    P_{k^*}(\theta_{k^*,0})= W_{k^*}^{\frac{1}{2}}(\theta_{k^*,0})(W_{k^*}^{-1}(\theta_{k^*,0})-H_{k^*}(\theta_{k^*,0}))W_{k^*}^{\frac{1}{2}}(\theta_{k^*,0}).
\end{align*}
It follows from Slutsky's theorem that
\begin{align*}
    T_{k^*,n}=\xi_{k^*,n}^{\top}P_{k^*}(\theta_{k^*,0})\xi_{k^*,n}\stackrel{d}{\to}\xi^{\top}P_{k^*}(\theta_{k^*,0})\xi
\end{align*}
under $H_0$, where
\begin{align*}
    \xi\sim N_{\bar{p}}(0,I_{\bar{p}}).
\end{align*} 
Note that
\begin{align}
P_{k^*}(\theta_{k^*,0})=P_{k^*}(\theta_{k^*,0})^{\top}. 
\end{align}
By the definition of $H_{k^*}(\theta_{k^*,0})$, 
\begin{align*}
    H_{k^*}(\theta_{k^*,0})W_{k^*}(\theta_{k^*,0})H_{k^*}(\theta_{k^*,0})&= W_{k^*}^{-1}(\theta_{k^*,0})\Delta_{k^*}(\Delta_{k^*}^{\top}W_{k^*}^{-1}(\theta_{k^*,0})\Delta_{k^*})^{-1}\Delta_{k^*}^{\top} W_{k^*}^{-1}(\theta_{k^*,0})\\
    &=H_{k^*}(\theta_{k^*,0}),
\end{align*}
and
\begin{align}
    \begin{split}
    P^2_{k^*}(\theta_{k^*,0})&=W_{k^*}^{\frac{1}{2}}(\theta_{k^*,0})(W_{k^*}^{-1}(\theta_{k^*,0})-H_{k^*}(\theta_{k^*,0}))W_{k^*}^{\frac{1}{2}}(\theta_{k^*,0})\\
    &\qquad\qquad\qquad\qquad\qquad\times W_{k^*}^{\frac{1}{2}}(\theta_{k^*,0})(W_{k^*}^{-1}(\theta_{k^*,0})-H_{k^*}(\theta_{k^*,0}))W_{k^*}^{\frac{1}{2}}(\theta_{k^*,0})\\
    &=W_{k^*}^{\frac{1}{2}}(\theta_{k^*,0})(W_{k^*}^{-1}(\theta_{k^*,0})-H_{k^*}(\theta_{k^*,0}))(I_{\bar{p}}-W_{k^*}(\theta_{k^*,0})H_{k^*}(\theta_{k^*,0}))W_{k^*}^{\frac{1}{2}}(\theta_{k^*,0})\\
    &=W_{k^*}^{\frac{1}{2}}(\theta_{k^*,0})(W_{k^*}^{-1}(\theta_{k^*,0})-H_{k^*}(\theta_{k^*,0})-H_{k^*}(\theta_{k^*,0})+H_{k^*}(\theta_{k^*,0}))W_{k^*}^{\frac{1}{2}}(\theta_{k^*,0})\\
    &=W_{k^*}^{\frac{1}{2}}(\theta_{k^*,0})(W_{k^*}^{-1}(\theta_{k^*,0})-H_{k^*}(\theta_{k^*,0}))W_{k^*}^{\frac{1}{2}}(\theta_{k^*,0})=P_{k^*}(\theta_{k^*,0}).
    \end{split}
\end{align}
Furthermore, 
\begin{align*}
    \rank{P_{k^*}(\theta_{k^*,0})}&=\tr{P_{k^*}(\theta_{k^*,0})}\\
    &=\tr{\{I_{\bar{p}}-W_{k^*}^{\frac{1}{2}}(\theta_{k^*,0})H_{k^*}(\theta_{k^*,0})W_{k^*}^{\frac{1}{2}}(\theta_{k^*,0})\}}\\
    &=\tr{I_{\bar{p}}}-\tr{\left\{W_{k^*}(\theta_{k^*,0})H_{k^*}(\theta_{k^*,0})\right\}}\\
    &=\bar{p}-\tr{\left\{\Delta_{k^*}(\Delta_{k^*}^\top W_{k^*}^{-1}(\theta_{k^*,0})\Delta_{k^*})^{-1}\Delta_{k^*}^\top W_{k^*}^{-1}(\theta_{k^*,0})\right\}}\\
    &=\bar{p}-\tr{\left\{(\Delta_{k^*}^\top W_{k^*}^{-1}(\theta_{k^*,0})\Delta_{k^*})^{-1}\Delta_{k^*}^\top W_{k^*}^{-1}(\theta_{k^*,0})\Delta_{k^*}\right\}}\\
    &=\bar{p}-\tr{I_{q_{k^*}}}=\bar{p}-q_{k^*}.
\end{align*}
Hence, 
\begin{align*}
    \xi^{\top}P_{k^*}(\theta_{k^*,0})\xi\sim\chi^2_{\bar{p}-q_{k^*}}
\end{align*}
from Lemma 7. 
Therefore, 
\begin{align*}
    T_{k^*,n}\stackrel{d}{\to}\chi^2_{\bar{p}-q_{k^*}}
\end{align*}under $H_0$.\qed
\end{prf}
\begin{lemma}
Under assumptions  [A1], [A2], [B1], [B2], [C2] and [C3], if $h_n\rightarrow0$ and $nh_n\rightarrow\infty$, then 
\begin{align}
    \hat{\theta}_{k^*,n}\stackrel{P}{\to}\bar{\theta}_{k^*}
    \label{l8}
\end{align}
\end{lemma}
\begin{proof}
We have the following decomposition
\begin{align*}
    &\quad F_{k^*,n}(Q_{XX},\Sigma_{k^*}(\theta_{k^*}))\\
    &=(\vech(Q_{XX})-\vech{(\Sigma_{k^*}(\theta_{k^*}))})^{\top}W_{k^*}^{-1}(\theta_{k^*})(\vech(Q_{XX})-\vech{(\Sigma_{k^*}(\theta_{k^*}))})\\
    &=(\vech(Q_{XX})-\vech{(\Sigma_{k}(\theta_{k,0}))})^{\top}W_{k^*}^{-1}(\theta_{k^*})(\vech(Q_{XX})-\vech{(\Sigma_{k}(\theta_{k,0}))})\\
    &\quad +2(\vech(Q_{XX})-\vech{(\Sigma_{k}(\theta_{k,0}))})^{\top}W_{k^*}^{-1}(\theta_{k^*})(\vech(\Sigma_{k}(\theta_{k,0}))-\vech{(\Sigma_{k^*}(\theta_{k^*})))}\\
    &\quad +(\vech(\Sigma_{k}(\theta_{k,0}))-\vech{(\Sigma_{k^*}(\theta_{k^*}))})^{\top}W_{k^*}^{-1}(\theta_{k^*})(\vech(\Sigma_{k}(\theta_{k,0}))-\vech{(\Sigma_{k^*}(\theta_{k^*}))}).
\end{align*}
Since it follows from Theorem 1 that 
\begin{align*}
    0&\leq \sup_{\theta_{k^*}\in \Theta_{k^*}}|F_{k^*,n}(Q_{XX},\Sigma_{k^*}(\theta_{k^*}))-U_{k^*}(\theta_{k^*})|\\
    &\leq \sup_{\theta_{k^*}\in\Theta_{k^*}}|(\vech(Q_{XX})-\vech{(\Sigma_{k}(\theta_{k,0}))})^{\top}W_{k^*}^{-1}(\theta_{k^*})(\vech(Q_{XX})-\vech{(\Sigma_{k}(\theta_{k,0}))})|\\
    &\quad +2\sup_{\theta_{k^*}\in\Theta_{k^*}}|(\vech(Q_{XX})-\vech{(\Sigma_{k}(\theta_{k,0}))})^{\top}W_{k^*}^{-1}(\theta_{k^*})(\vech(\Sigma_{k}(\theta_{k,0}))-\vech{(\Sigma_{k^*}(\theta_{k^*}))})|\\
    &\leq |\vech(Q_{XX})-\vech{(\Sigma_{k}(\theta_{k,0}))}|\sup_{\theta_{k^*}\in\Theta_{k^*}}\|W_{k^*}^{-1}(\theta_{k^*})\||\vech(Q_{XX})-\vech{(\Sigma_{k}(\theta_{k,0}))}|\\
    &\quad +2|\vech(Q_{XX})-\vech{(\Sigma_{k}(\theta_{k,0}))}|\sup_{\theta_{k^*}\in\Theta_{k^*}}\|W_{k^*}^{-1}(\theta_{k^*})\||\vech(\Sigma_{k}(\theta_{k,0}))-\vech{(\Sigma_{k^*}(\theta_{k^*}))}|\\
    &\stackrel{P}{\to} 0,
\end{align*}
it is shown that
\begin{align}
    F_{k^*,n}(Q_{XX},\Sigma_{k^*}(\theta_{k^*}))\stackrel{P}{\to} U_{k^*}(\theta_{k^*})\quad \mbox{uniformly in $\theta_{k^*}$}.
\end{align}
From the assumption [C3], 
\begin{align*}
    \forall \varepsilon>0,\exists \delta>0 
    \ s.t.\ |\hat{\theta}_{k^*,n}-\bar{\theta}_{k^*}|>\varepsilon \Rightarrow U_{k^*}(\hat{\theta}_{k^*,n})-U_{k^*}(\bar{\theta}_{k^*})>\delta.
\end{align*}
Hence, 
\begin{align*}
    0&\leq \PP(|\hat{\theta}_{k^*,n}-\bar{\theta}_{k^*}|>\varepsilon)\\
    &\leq  \PP\left(U_{k^*}(\hat{\theta}_{k^*,n})-U_{k^*}(\bar{\theta}_{k^*})>\delta\right)\\
    &\leq \PP\left(U_{k^*}(\hat{\theta}_{k^*,n})-F_{k^*,n}(Q_{XX},\Sigma_{k^*}({\hat{\theta}_{k^*,n}}))\right.\\
    &\quad\left.+F_{k^*,n}(Q_{XX},\Sigma_{k^*}({\hat{\theta}_{k^*,n}}))-F_{k^*,n}(Q_{XX},\Sigma_{k^*}(\bar{\theta}_{k^*}))+F_{k^*,n}(Q_{XX},\Sigma_{k^*}(\bar{\theta}_{k^*}))-U_{k^*}(\bar{\theta}_{k^*})
    >\delta\right)\\
    &\leq \PP\left(U_{k^*}(\hat{\theta}_{k^*,n})-F_{k^*,n}(Q_{XX},\Sigma_{k^*}({\hat{\theta}_{k^*,n}}))>\frac{\delta}{3}\right)\\
    &\quad +\PP\left(F_{k^*,n}(Q_{XX},\Sigma_{k^*}({\hat{\theta}_{k^*,n}}))-F_{k^*,n}(Q_{XX},\Sigma_{k^*}(\bar{\theta}_{k^*}))>\frac{\delta}{3}\right)\\
    &\quad +\PP\left(F_{k^*,n}(Q_{XX},\Sigma_{k^*}(\bar{\theta}_{k^*}))-U_{k^*}(\bar{\theta}_{k^*})>\frac{\delta}{3}\right)\\
    &\leq  2\PP\left(\sup_{\theta_{k^*}\in\Theta_{k^*}}\left|F_{k^*,n}\left(Q_{XX},\Sigma_{k^*}({\theta_{k^*}})\right)-U_{k^*}(\theta_{k^*})\right|>\frac{\delta}{3}\right)+0\\
    &\stackrel{P}{\to}0,
\end{align*}
and (\ref{l8}) is deduced.
\end{proof}
\begin{lemma}
Let $X_n\stackrel{p}{\to} c>0$. Then, for all $\varepsilon>0$,
\begin{align*}
    \PP\left(X_n\leq\frac{\varepsilon}{n}\right)\stackrel{}{\to} 0\quad (n\rightarrow \infty)
\end{align*}
\end{lemma}
\begin{proof}
See Lemma 3 in Kitagawa and Uchida \cite{kitagawa(2014)}.
\end{proof}
\begin{prf}
Since $U_{k^*}(\theta_{k^*})$ is a continuous function, 
it follows from the continuous mapping theorem that
\begin{align*}
    U_{k^*}(\hat{\theta}_{k^*,n})\stackrel{P}{\to} U_{k^*}(\bar{\theta}_{k^*})
\end{align*}
under $H_1$.
Hence, noting that for all $\varepsilon>0$,
\begin{align*}
    0&\leq \PP\left(\left|\frac{1}{n}T_{k^*,n}-U_{k^*}(\bar{\theta}_{k^*})\right|>\varepsilon\right)\\
    &=\PP\left(|F_{k^*,n}(Q_{XX},\Sigma_{k^*}(\hat{\theta}_{k^*,n}))-U_{k^*}(\bar{\theta}_{k^*})|>\varepsilon\right)\\
    &\leq \PP\left(|F_{k^*,n}(Q_{XX},\Sigma_{k^*}(\hat{\theta}_{k^*,n}))-U_{k^*}(\hat{\theta}_{k^*,n})|>\frac{\varepsilon}{2}\right)\\
    &\quad+\PP\left(|U_{k^*}(\hat{\theta}_{k^*,n})-U_{k^*}(\bar{\theta}_{k^*})|>\frac{\varepsilon}{2}\right)\\
    &\leq \PP\left(\sup_{\theta_{k^*}\in\Theta_{k^*}}|F_{k^*,n}(Q_{XX},\Sigma_{k^*}(\theta_{k^*}))-U_{k^*}(\theta_{k^*})|>\frac{\varepsilon}{2}\right)\\
    &\quad+\PP\left(|U_{k^*}(\hat{\theta}_{k^*,n})-U_{k^*}(\bar{\theta}_{k^*})|>\frac{\varepsilon}{2}\right)\stackrel{}{\to}0
\end{align*}
under $H_1$, 
we get
\begin{align*}
    \frac{1}{n}T_{k^*,n}\stackrel{P}{\to} U_{k^*}(\bar{\theta}_{k^*})
\end{align*}
under $H_1$. 
By $\Sigma_{k}(\theta_{k,0})\neq \Sigma_{k^*}(\bar{\theta}_{k^*})$ under $H_1$, 
\begin{align*}
    \vech{(\Sigma_{k}(\theta_{k,0}))}-\vech{(\Sigma_{k^*}(\bar{\theta}_{k^*}))}\neq 0
\end{align*}
under $H_1$.
Since it follows from Lemma 5 that $U_{k^*}(\bar{\theta}_{k^{*}})>0$ under $H_1$,
Lemma 9 yields that
\begin{align*}
    \PP(T_{k^*,n}>\chi^2_{\bar{p}-q_{k^*}}(\alpha))&=1-\PP(T_{k^*,n}\leq \chi^2_{\bar{p}-q_{k^*}}(\alpha))\\
    &=1-\PP\left(\frac{1}{n}T_{k^*,n}\leq \frac{1}{n}\chi^2_{\bar{p}-q_{k^*}}(\alpha)\right)\stackrel{}{\to}1
\end{align*}
under $H_1$.\qed
\end{prf}
{\bf Proofs of Theorems 5-8.}
Since $T$ is fix and $nh_n^2=h_nT \rightarrow0$, 
the proofs of Theorem $5$-$8$ is the same as the proofs of Theorem $1$-$4$, respectively.
\qed 

\end{document}